\documentclass[
10pt,
letterpaper,
oneside,
headinclude,footinclude, 
BCOR=5mm, 
]{article}

\usepackage{mathptmx}
\usepackage{helvet}
\usepackage{courier}
\usepackage{type1cm}         

\usepackage{makeidx}         % allows index generation
\usepackage{graphicx}        % standard LaTeX graphics tool
\usepackage[bottom]{footmisc}% places footnotes at page bottom

\usepackage{dcolumn}
\newcolumntype{d}[1]{D{.}{.}{#1}}
\usepackage{amsthm}
\usepackage{mathrsfs}
\usepackage{amsmath}
\usepackage{amssymb}
\usepackage{newtxtext}       % 
\usepackage{newtxmath}       % selects Times Roman as basic font
\usepackage{epigraph}
\usepackage[LGR,T1]{fontenc}
\usepackage{accents}
\usepackage{bbding}
\usepackage{lscape}

\usepackage{hyperref}

\title{\normalfont{Unification of the Fundamental Forces in Higher-Order Riemannian Geometry}} 

\author{{William Edward Bies
} \\ {E-mail: william.bies.phd@gmail.com
}}

\date{June 6, 2024} 

\begin{document}
	
	\numberwithin{equation}{section}
	\theoremstyle{plain}
	\newtheorem{theorem}{Theorem}[section]
	\newtheorem{proposition}{Proposition}[section]
	\newtheorem{lemma}{Lemma}[section]
	\newtheorem{corollary}{Corollary}[section]
	\newtheorem{postulate}{Postulate}[section]
	
	\theoremstyle{definition}
	\newtheorem{definition}{Definition}[section]
	
	\theoremstyle{remark}
	\newtheorem{remark}{Remark}[section]
	
	\renewcommand{\qedsymbol}{\small \RectangleBold}
	
	\maketitle
	
	\setcounter{tocdepth}{2} 
	
	\tableofcontents
	
	\newpage
	
	\section*{Abstract}
	
	In Part I of the present series of papers, we adumbrate our idea of Riemannian geometry to higher order in the infinitesimals and derive expressions for the appropriate generalizations of parallel transport and the Riemannian curvature tensor. In Part II, the implications of higher-order geometry for the general theory of relativity beyond Einstein are developed. In the present Part III, we expand on the framework of Part II so as to take up the problem of field-theoretical unification. Employing the form of the Einstein-Hilbert action to higher order as proposed in Part II, we show how in nearly flat space the higher-order terms give rise to a gauge theory of Yang-Mills type. At the 2-jet level, the electroweak force emerges after imposition of gauge fixing. In fact, the proposed form of the Einstein-Hilbert action permits us to say more: we argue that the equivalence principle results in a Proca term that brings about the spontaneous symmetry breaking of the standard model. Two empirical predictions support our reasoning: first, we obtain a theoretical value for the Weinberg angle and second, we find that---without any adjustable parameters---the implied value of Coulomb's constant agrees well with experiment. The final section examines the 3-jet level. The same mechanism that produces Glashow-Weinberg-Salam electroweak theory at the 2-jet level eventuates in a chromodynamical force having $SU(3)$ symmetry at the 3-jet level.
		
			\vspace{12 pt}
		
		2020 Mathematics Subject Classification. Primary: 53A45 Differential geometric aspects in vector and tensor analysis. Secondary: 83C50 Electromagnetic fields in general relativity and gravitational theory; 83E99 Unified, higher-dimensional and super field theories, none of the above.
		
		\vspace{12 pt}
		
		Keywords: higher-order differential geometry, general theory of relativity, unified field theory 
		
		\vspace{12 pt}
		
		Subject Classifications: Riemannian geometry, local differential geometry, analysis on manifolds
		
		\vspace{12 pt}
		
		%Declarations of interest: None. This research did not receive any specific grant from funding agencies in the public, commercial, or not-for-profit sectors.	

	\include{section_1}
	\section{Prospective Unification of the Fundamental Forces of Nature}\label{chapter_8}

There is an antecedent reason to expect higher-order differential geometry to have significant applications in physics as it corresponds to a profound idea about spatial relations and, as we know, the general theory of relativity is inherently geometrical. Therefore, one might venture to suppose that our novel ideas on the nature of space and time in the infinitely small, as developed in Parts I and II of the present series of papers, could lead to an alternative framework for the unification of forces, in the spirit of Weyl and Einstein rather than of Kaluza and Klein (the only one seriously attempted to date in existing research programs into fundamental physics). 

Prima facie one should expect novel phenomena to result from the coupling of jets to all higher orders. For if gravitational interactions receive a natural description in terms of the jets of first order, other phenomena like in kind ought to be attributable to the presence of higher jets---that is to say, they would manifest themselves dynamically through the jet geodesic equation and be interpretable as additional contributions to the action at a distance, or what comes to the same thing, as fundamental forces. The Lorentz force law in electrodynamics may be obtained as a result of the coupling between jets of first and second order in the jet geodesic equation. But general covariance demands that we press on to unification of gravitational and electroweak forces. The reduction of the post-Einsteinian field equation at second order in the jets, in the weak-field limit, to non-abelian $\mathrm{u}(1)_Y \oplus \mathrm{su}(2)$ Yang-Mills gauge theory is here carried out in full ({\S}\ref{chapter_8}). In particular, spontaneous symmetry breaking occurs as a direct consequence of the natural generalization of the equivalence principle. These same ideas will be applied at the 3-jet level in order to \textit{construct} classical chromodynamics ({\S}\ref{chapter_9}).

\subsection{Preliminary Observations}

Before we can undertake an analysis of the place of gauge field theory in post-Einsteinian general theory of relativity, we must first prepare the groundwork. The method of investigation to follow in {\S}{\S}\ref{unification_of_gr_and_em}-\ref{unification_gravity_electroweak} is clear. All one has to do is to substitute the extended Riemannian curvature tensor into the field equations of Einstein’s general theory of relativity and seek to interpret the phenomenology of the higher sectors. We limit ourselves in this section to the first non-trivial order beyond the ordinary theory.

\subsection{Unification of Gravity and Electrodynamics}\label{unification_of_gr_and_em}

A natural suggestion would be the following: in ordinary differential geometry, replacement of the coordinate derivative with the covariant derivative leads to an intrinsically defined geodesic equation. Thus we would like to seek to recast Lorentz' force law in an intrinsic manner in post-Einsteinian theory by viewing it as part of the jet geodesic equation. We could just write out the jet geodesic equation and inspect it directly for what we are seeking, but a perhaps more assumption-free approach would be to derive the equation of motion itself of a massive charged body from a principle of least action, just as is done in the ordinary theory for a massive neutral body moving through curved space. Then, if we are fortunate, this procedure should give the Newtonian gravitational force as expected while the additional higher-order terms should yield the Lorentz force law for a charged particle.

\subsubsection{Electrodynamics from the post-Einsteinian Point of View}\label{electrodynam_post_einstein}

Our first item of business will be to derive the jet geodesic equation to subleading order, i.e., to find the dominant post-Einsteinian corrections. Start with an action functional for trajectories on a Riemannian manifold of the following form: 
\begin{equation}
	A = \int \gamma^* \omega_g = \int \sqrt{L} \theta^s,
\end{equation}
where the spatial dependence of the Lagrangian $L$ is given implicitly through a finite number of fields $u_{1,\ldots,k}$ and all of their derivatives. 
\begin{definition}
	Consider a function $f$ whose spatial dependence can be expressed implicitly through a finite number of other functions $u_{1,\ldots,k}$: for $p \in M$, $f(p) = f(u_1(p),\ldots,u_k(p))$. The action is stationary with respect to variation of $u_q$, $q=1,\ldots,k$ if for test functions $\eta$ of compact support on $M$ the functional derivative vanishes:
	\begin{equation}
		\frac{\delta A}{\delta u_q} = \lim_{\varepsilon \rightarrow 0} \frac{1}{\varepsilon} \int_M \left( L(u_q+\varepsilon \eta) - L(u_q) \right) \omega_g = 0,
	\end{equation}  
	where $u_1,\ldots,\hat{u}_q,\ldots,u_k$ are held fixed.
\end{definition}
Olver's \cite{olver} Theorem 4.4 tells us that a set $u_{1,\ldots,k}$ will be stationary with respect to variations if we have $\vvmathbb{E}_\nu L = 0$, $\nu=1,\ldots,k$ where for all smooth functions $\varphi$ the $\nu$-th Euler operator is given by
\begin{equation}\label{euler_operator}
	\vvmathbb{E}_\nu \varphi := \sum_\alpha (-D)^\alpha \frac{\partial}{\partial u_\nu^\alpha} \varphi,
\end{equation}
setting $D = D_{\alpha_1} \cdots D_{\alpha_n}$ with $D_\lambda$ the total derivative with respect to the coordinate $x^\lambda$ given by
\begin{equation}
	D_\lambda = \frac{\partial}{\partial x^\lambda} + u_\nu^{\alpha+\lambda} \frac{\partial}{\partial u_\nu^\alpha};	\qquad u_\nu^{\alpha+\mu} := \frac{\partial^{|\alpha|+1} u_\nu}{\partial x^\mu \partial x^\alpha}
\end{equation}
and the domain of definition corresponds to the total jet bundle.

\begin{remark}
	We will want to use $L = \dfrac{1}{2m}g(X,X)$, where $X$ is obtained by pushing forward $\partial_s + \partial_{ss} + \cdots + \partial_{s \cdots s}$ under the embedding of the body's world-line into space-time. Actually, one should use momentum $P$; i.e., $L = \dfrac{1}{2m}g(P,P)$ where now it is not necessarily the case that $P=mX$. For the remainder of this section, we shall set $P=mX_1+eX_2$, $m$ being rest mass, $e$ electric charge and $X_1$ resp. $X_2$ the 1-vector resp. 2-vector parts of $X$ with respect to some inertial frame. In {\S}\ref{unification_gravity_electroweak} we investigate more closely the question of what form to take for $X_2$.
\end{remark}

\begin{lemma}\label{metric_lemma}
	We may replace $A=\int \sqrt{L}$ with $A=\int L$ and obtain equivalent equations of motion.
\end{lemma}
\begin{proof}
	Consider $A=\int f(L)$. Thirring's derivation (\cite{thirring_vol_1}, {\S}5.6.14) follows from the Euler-Lagrange equations together with independence of the given form of $L$ resp. $\partial f/\partial L$ with $f = \sqrt{}$ on $(\vvmathbb{R},\delta_1)$ on-shell hence can be divided on both sides of the equation. We must show same or substantial equivalent using the higher Euler-Lagrange equations (cf. Olver, \cite{olver}). Now, in analogy to what Thirring does we may normalize $s$ by requiring that
	\begin{equation}
		L(x^{(s)}_\nu: 1 \le \nu \le n, 0 \le s \le r) = \frac{1}{2} g_{\alpha\beta} X^\alpha X^\beta = ~\mathrm{constant}.
	\end{equation}
	Let the trajectory by given by $\gamma: s \mapsto (x_1(s),\ldots,x_n(s))$. Then the components of higher velocity $X^\alpha$ with the multi-index $\alpha=(\alpha_1,\ldots,\alpha_n)$ may be evaluated in terms of the derivatives $\dot{x}_\nu, \ddot{x}_\nu,\ldots,x_\nu^{(r)}$ by the formula for the pushforward:
	\begin{align}
		\left( \gamma_* \partial_s \right) (h) &=
		\frac{\partial h}{\partial x_\nu} \dot{x}_\nu \nonumber \\
		\left( \gamma_* \partial_{ss} \right) (h) &=
		\frac{\partial h}{\partial x_\nu} \ddot{x}_\nu + \frac{\partial^2 u}{\partial x_\nu \partial x_\mu} \dot{x}_\nu \dot{x}_\mu,
	\end{align} 
	etc. From equation (\ref{euler_operator}), the Euler operator becomes
	\begin{equation}
		(-1)^r \frac{d^r}{ds^r} \frac{\partial f}{\partial L} \frac{\partial L}{\partial x_\nu^{(r)}} + \cdots - \frac{d}{ds} \frac{\partial f}{\partial L} \frac{\partial L}{\partial \dot{x}_\nu} + \frac{\partial f}{\partial L} \frac{\partial L}{\partial x_\nu} = 0.
	\end{equation}
	But since by our convention, $L$ will be independent of $s$ along solutions, so will be $\partial f/\partial L$. Then we may commute all such coefficients past the derivatives with respect to $s$ and collect them as a constant prefactor. We obtain thereby a common factor in the coefficient of the differential of first order in time and thus an equivalent equation of motion for $f(L)$ as for $L$ itself. 
	
	To second order, we would have an expression of the following form for the first variation in the action:
	\begin{equation}
		\delta A = \int \delta f(L) ds^\prime + \int ds^\prime \int ds^{\prime\prime} \delta f(L).
	\end{equation}
	Evidently, one finds again a dependence on $\partial f/\partial L$ in the second-order term, even though its contribution to $\delta A$ is to be determined by integrating twice. Hence, as before we can bring out a common factor and obtain
	equivalent equations of motion. For the same reason, the same can be done at each successive higher order. This concludes the proof.
\end{proof}

\begin{theorem}[Cf. Thirring \cite{thirring_vol_1}, {\S}5.6.9]\label{jet_geodesic_from_action}
	Let $M$ be a space-time manifold on which a Riemannian metric $g$ is defined in the extended sense, and for any given trajectory $\gamma: [a,b] \rightarrow M$ set $A = \int_a^b \sqrt{ \langle X,X \rangle} \theta^s$ for the induced action. Then the jet geodesic equation for the canonical Levi-Civita connection results from a principle of least action thus expressed.
\end{theorem}
\begin{proof}
	Let $Z$ denote a variation field. In view of lemma \ref{metric_lemma}, we may take as our integrand just $\langle X,X \rangle$. If $X$ is a stationary solution to the variational problem, we have that the integral of $\langle Z,X \rangle$ must vanish. In addition, our normalization above implies that $Z \langle X,X \rangle = 0$. First, taking the derivatives under the integral sign we may write,
	\begin{equation}\label{variational_jet_geodesic_1}
		X \int \langle Z, X \rangle = \int X \langle Z,X \rangle = \int \big( \langle \nabla_X Z,X \rangle + \cdots +
		\langle Z, \nabla_X X \rangle \big) = 0.
	\end{equation}
	Here, we drop the cross-terms in the weak-field limit. But it is evident in the same approximation that
	\begin{equation}
		Z \langle X,X \rangle = \langle \nabla_Z X, X \rangle + \langle X, \nabla_Z X \rangle = 2  \langle \nabla_Z X, X \rangle = 0.
	\end{equation}
	Furthermore, owing to the symmetry of the Levi-Civita jet connection, we may write,
	\begin{equation}
		\nabla_X Z = \nabla_Z X + [X,Z] = \nabla_Z X + XZ - ZX.
	\end{equation}
	Equation (\ref{variational_jet_geodesic_1}) becomes
	\begin{equation}
		\int \big( \langle XZ, X \rangle - \langle ZX,X \rangle + \langle Z, \nabla_X X \rangle \big) = 0.
	\end{equation}
	If, however, $Z$ is a variation field so will be $XZ$ resp. $ZX$; hence, the integrals of the first two terms vanish by virtue of the stationarity of $X$ leaving us with
	\begin{equation}\label{variational_jet_geodesic_2}
		\int \langle Z, \nabla_X X \rangle = 0.
	\end{equation}
	Since equation (\ref{variational_jet_geodesic_2}) is to hold for all $Z$, we conclude that $\nabla_X X = 0$ identically.
\end{proof}

\begin{remark}
	Weak, in this context, means relative to the Planck scale. As we shall see in a moment, this criterion is well satisfied by fields encountered in the terrestrial laboratory or in outer space (except perhaps near the singularity inside the event horizon of a black hole).
\end{remark}

\begin{corollary}
	In the weak-field limit, the second-order correction to the jet geodesic equation yields Lorentz' electrodynamical force law for charged moving bodies.
\end{corollary}
\begin{proof}
	The 11-terms reduce as usual to the conventional geodesic equation and lead to Newton's law of universal gravitation in almost flat space, as Einstein finds for the first time in \cite{einstein_1915}. For the 12-terms, examine the expression for the Levi-Civita connection (I.4.13) with $|\mu,\alpha|=1$ and $|\beta|=2$:
	\begin{align}\label{Levi_Civita_formula_quoted}
		\Gamma^\mu_{\alpha\beta} = & \frac{1}{2} g^{\mu\delta} \bigg(
		\partial_\alpha g_{\beta\delta} + \partial_\beta g_{\delta\alpha} - \partial_\delta g_{\alpha\beta} - \nonumber \\ &\frac{\alpha!}{\alpha_1!\alpha_2!}\bigg|_{\alpha_{1,2}\ne 0} g_{\mu\nu}
		\Gamma^\mu_{\alpha_1\beta}\Gamma^\nu_{\alpha_2\delta} -
		\frac{\beta!}{\beta_1!\beta_2!}\bigg|_{\beta_{1,2}\ne 0} g_{\mu\nu}
		\Gamma^\mu_{\beta_1\alpha}\Gamma^\nu_{\beta_2\delta} +
		\frac{\delta!}{\delta_1!\delta_2!}\bigg|_{\delta_{1,2}\ne 0} g_{\mu\nu}
		\Gamma^\mu_{\delta_1\alpha}\Gamma^\nu_{\delta_2\beta}
		\bigg).
	\end{align}	
	Now, in the weak-field limit we will be justified in discarding the cross-terms, leaving	
	\begin{equation}\label{Levi_Civita_formula_weak_field_a}
		\Gamma^\mu_{\alpha\beta} = \frac{1}{2} g^{\mu\delta} \bigg(
		\partial_\alpha g_{\beta\delta} + \partial_\beta g_{\delta\alpha} - \partial_\delta g_{\alpha\beta} 
		\bigg),
	\end{equation}	
	formally the same as the conventional expression for $|\mu|=|\alpha|=|\beta|=1$.
	In the weak-field limit we may without hazard replace the prefactor by the Minkowskian metric:
	\begin{equation}\label{Levi_Civita_formula_weak_field_b}
		\Gamma^\mu_{\alpha\beta} = \frac{1}{2} \bigg(
		\partial_\alpha g_{\beta}^{\mu} + \partial_\beta g^{\mu}_{\alpha} - \partial^\mu g_{\alpha\beta} 
		\bigg) =
		\frac{1}{2} \bigg(
		\partial_\beta g_{\alpha}^{\mu} + \partial_\alpha g^{\mu}_{\beta} - \partial^\mu g_{\alpha\beta} 
		\bigg),
	\end{equation}	
	where the latter equation follows by symmetry of $\nabla$. When $|\beta|=2$, terms are majorized by $\beta=00$ (for bodies traveling at less than the speed of light).
	
	Why one may expect the symmetric part of the ten Christoffel symbols to be negligible under ordinary conditions. In the weak-field limit, the antisymmetric part involves first derivatives of metric components corresponding to second-order jets while the symmetric part involves second derivatives of the metric components corresponding to first-order jets, i.e., the conventional Riemannian metric in Einsteinian general theory of relativity which reflect the presence of gravitational forces. Now, in our experience with the terrestrial laboratory, the solar system, interstellar or intergalactic space (apart from the vicinity of any black holes), the gravitational field is slowly varying both spatially and temporally, at scales ranging from kilometers down to {\AA}ngstroms and from hours and minutes down to the orbital periods of atomically bound electrons. Thus, any terms of second order in the derivatives should be small compared to those of first order. There may be gravitational corrections to the inferred electroweak gauge fields close-in near point singularities (if elementary particles may thus be represented), in the static case, or in the case of very rapidly moving bodies, such as the final stages of inspiralling black holes. At solar-system scales (measured by the astronomical unit), the gravitational part of the metric tensor does indeed vary appreciably but one need not here be concerned with the electroweak part in view of the fact that all known astrophysical objects have almost vanishing net electric and weak charges.

	Here $|\mu,\alpha|=1$ correspond to ordinary 1-velocities. Up to the factor of one half, we are left with a formula identical to Lorentz' force law for a moving body subject to a Maxwellian electromagnetic field with 4-vector potentials $g^\mu_\beta$, where any value of $\beta \ne 00$ cancels in the co-moving frame (under suitable assumption of a linear relation between momentum and velocity)
	
	Namely, set $F^\mu_{\nu} = - 2 \Gamma^\mu_{\nu,00} = \partial^\mu g_{\nu,00} - \partial_\nu g^\mu_{00}$ where we pay attention to $\mu,\nu$ only when they run over spatio-temporal indices 0,1,2,3 and to $\gamma=00$ only, in the charged moving body's rest frame. Let $\sigma$ denote the charge-to-mass ratio of a body relative to fundamental units. In these terms, the 
	body's 2-velocity with respect to the inertial frame in which the putative electromagnetic field is measured becomes
	\begin{equation}
		X = \frac{d}{d\lambda} x^\mu + \left( \frac{d}{d\lambda} x^\mu \right) \otimes \left( \frac{d}{d\lambda} x^\nu \right)
	\end{equation}
	where $\lambda$ designates proper time along its world-line and the second-order contribution to the jet geodesic equation reads (with $q = m \sigma$),
	\begin{equation}
		m \left( \frac{d}{d\lambda} \right)^2 x^\mu =  - \Gamma^\mu_{\nu\gamma} X^\nu P^\gamma = q F^\mu_\nu \frac{d}{d\lambda} x^\nu,
	\end{equation}
	viz., the familiar Lorentz force law in tensorial form, cf. \cite{weyl_raum_zeit_materie}, \cite{jackson}. Note the factor of $\frac{1}{2}$ drops out because we sum over $\nu,00$ and $00,\nu$ and the minus sign goes away in virtue of the definition of the field-strength tensor $F^\mu_\nu$. Remember that strictly speaking both the Lorentz force law and the jet geodesic equation are understood to be written in terms of the moving body's \textit{momentum}.
\end{proof}

\subsubsection{Status of the Equivalence Principle in post-Einsteinian Theory}\label{status_of_equivalence_principle}

Another vital point calls for attention in this section. The equivalence principle as proposed by Einstein permits one to relate the motion of gravitating bodies to the geodesics of the pseudo-Riemannian metric governing space-time. For if the body's action functional enters the problem as arclength of its world-line preceded by a coefficient proportional to the body's rest mass $m$, then the stationary solution for its motion derived therefrom will involve its \textit{momentum} $P$, where $P = mX$; i.e., reads
\begin{equation}
	\nabla_X P = 0; \qquad \mathrm{equivalently}, \qquad \dot{P} = - \Gamma X P.
\end{equation}
That is to say, the body's momentum is transported parallely by its velocity. Since both sides are linearly proportional to rest mass, we may divide through to obtain the geodesic equation as
\begin{equation}
	\dot{X} = - \Gamma X X,
\end{equation}
which agrees with the formula of Riemannian geometry. The \textit{physical} point is
that all bodies follow the same trajectories in a given gravitational field, independent of their material constitution. If we are to frame an analogue of the equivalence principle to higher order, what suggests itself immediately would be merely to put $P=mX$ where now $P$ and $X$ are to be understood as higher tangent vectors solving the jet geodesic equation $\nabla_X P = 0$. 

We wish, however, the momentum to reflect both rest mass and charge of a moving body. Therefore, one could try breaking down the velocity into 1-vector and 2-vector components $X=X_1+X_2$ and writing $P = m X_1 + q X_2$. If this decomposition were to hold in nature, then one could preserve the equivalence principle by supposing every body to carry an electric charge equal to its rest mass: if $P = m X_1 + q X_2 = m X$, it would imply $q=m$. If therefore we interpret charge and rest mass as interchangeable, this supposition might suggest that, for charged bodies with equal and opposite charges, say the electron and the proton, $q = \pm e$ would mean that they behave differently with respect to a gravitational field, as the deviation of $q$ from $m$ would be different for the two bodies. A startling implication would be a potential conflict with sensitive observations \cite{millikan}, \cite{king}, \cite{fraser_carlson_hughes}, \cite{hughes_fraser_carlson}. The resulting tension with experiment prompts another proposal as to what form the 2-vector part of the momentum should take---one that apparently better respects the inherent content behind the equivalence principle. 
If we hypothesize that electric charges are to be calibrated relative to the body's rest mass as $\alpha^{1/2}m$, the equivalence principle would imply that all charged bodies fall at same rate in the absence of electromagnetic forces, together with any contribution from the Lorentz force law in the presence of a non-zero electromagnetic field.\footnote{To justify the factor of $\alpha^{1/2}$, note that in geometrical units the charge of the electron is given by $e = -\alpha^{1/2} m_P$. If we are to displace the origin relative to which charge is measured, what makes sense would be to use the electron's rest mass multiplied by the same factor of $\alpha^{1/2}$, namely, $q_e = \alpha^{1/2}(m_e-m_P)$ and correspondingly for the proton, $q_p = \alpha^{1/2}(m_p+m_P)$. For the time being, this proposal has the status of but a surmise but as we shall see, if we decide on this choice, it causes the electromagnetic coupling constant in our theory to work out just right in order to yield the right hand side of Maxwell's equations.} For the moment, to save the phenomena by means of a device such as this gives the impression of being ad hoc. Nevertheless, there are other and stronger reasons in its favor, to which we shall recur below in {\S}\ref{unification_gravity_electroweak}. 

\subsection{Unification of Gravity with the Electroweak Force}\label{unification_gravity_electroweak}

In view of the possibility of applying a Lorentz transformation to second order, the $g_{0,00}$ component of the generalized Riemannian metric tensor cannot be taken in isolation. A winsome aspect of the present approach is that its covariance demands we proceed to electroweak unification. This we do in {\S}\ref{electroweak_construction}, but first we wish briefly to meditate on the conceptual basis standing behind the procedure (which therefore ought to be applicable to jets of any order) in the following subsection.

\subsubsection{From Jets to Gauge Fields}\label{from_jets_to_gauge}

The procedure sketched in the preceding section, in which imposition of spatial variation in the coefficient of the $0,00$ term in the generalized metric tensor gives rise to forces of electromagnetic type, does not exhaust the range of possibilities. For according to the formalism adopted, there will also be non-zero coefficients of the $0,01$; $0,02$ and $0,03$ terms. Our prima facie expectation, then, would be that just as transformation from Minkowskian space to a temporally difform time alone wrests an electromagnetic field from the 0,00 term, so too should transformation to spatially difform time yield non-zero weak fields from the 0,01; 0,02; 0,03 terms. (Needless to say, not every possible electroweak field configuration can be so derived from the vacuum in free space, but the procedure adopted in deference to Einstein illustrates some of the implications of general covariance in the post-Einsteinian setting.)

We have a little more work to do in order to flesh out the theory. In determining whether the affine connection 1-forms behave like a collection of gauge fields, what is demanded is to show that they fit together component-wise into a Lie-algebra valued 1-form whose covariant derivative (in a sense to be specified) yields a Lie-algebra valued curvature 2-form, corresponding to the field strength at least approximately, and moreover that these couple to charged material bodies via the jet geodesic equation according to the Lorentz force law. The vector potentials must not just span a space of the requisite dimension but transform among themselves with the structure constants appropriate to the respective Lie algebra. Here, this invites the conjecture that the gauge transformations themselves will become more than just abstract formal operations but be directly interpretable in terms of infinitesimal changes of Cartan's co-frame basis in space-time.

The salient characteristic of the momentum vector associated with a given 1-velocity vector is that it behaves the same way with respect to rotations, both in ordinary classical mechanics (symmetry group $SO(3)$) and in the special relativistic case (symmetry group $SO(1,3)$); since the representations are irreducible, Schur’s lemma implies that the only possibility is to multiply the velocity by a constant: $p=m\varv$. But the representation of $SO(1,3)$ on $r$-vectors resp. jets is not irreducible when $r \ge 2$ (in what follows, it will prove helpful to refer to the treatment of spinor representations in Streater and Wightman, \cite{streater_wightman}). Here, then, the natural proposal to generalize the concept of momentum would be a homomorphism from the space of $r$-vectors into itself that behaves like an intertwinor with respect to $SO(1,3)$. Now, we expect such an intertwinor to be characterized by a number of parameters, one for each irreducible piece. In the case of 2-vectors resp. jets, 

\begin{equation}\label{Lorentz_irrep_decomp}
	J^2 = J^2_1 + J^2_2 = \mathscr{D}^{(1/2,1/2)} \oplus \mathrm{Sym} \left( \mathscr{D}^{(1/2,1/2)} \otimes \mathscr{D}^{(1/2,1/2)} \right) = \mathscr{D}^{(1/2,1/2)} \oplus \mathscr{D}^{(0,0)} \oplus \mathscr{D}^{(1,1)}.
\end{equation}
There will thus, by Schur's lemma, be three parameters corresponding to the three irreducible pieces in the direct sum, which tentatively may be identified as the mass, the weak hypercharge and the weak isospin, respectively.

Let us expand on the significance of the decomposition expressed in equation (\ref{Lorentz_irrep_decomp}). For the key point to realize is that general covariance implies that field configurations in two given coordinate frames that go into one another upon change of coordinate are physically equivalent and ought to be identified. Hence, we shall be interested in equivalence classes in the above decomposition under the action of $SO(1,3)$. In the $\mathscr{D}^{(1/2,1/2)}$ sector, one has three sets of equivalence classes, namely, the concentric hyperboloids of two sheets, one in the future-directed light-cone and the other in the past-directed light-cone and lastly, concentric hyperboloids of one sheet in the space-like region. Now, without any firm basis other than a certain intuition, we suggest that only the future-directed light-cone will be physically relevant. The associated parameter delivered by Schur's lemma will be just the rest mass $m \in \vvmathbb{R}_+$, equal for a given hyperboloid to the distance from the origin of its intersection with the positive time axis.

As for the 2-jets, obviously there can only be one parameter in the scalar representation $\mathscr{D}^{(0,0)}$, namely electric charge $e \in \vvmathbb{R}$. Clearly, $\mathscr{D}^{(0,0)}$ is spanned by $d^{tt}-d^{xx}-d^{yy}-d^{zz}$. The interesting part is to be found in the remaining $\mathscr{D}^{(1,1)}$ representation. Let $A$ be a typical element, viz., a traceless and symmetric $4 \times 4$ matrix on which $SO(1,3)$ acts via $A \mapsto (\Lambda^t)^{-1} \otimes (\Lambda^t)^{-1} A$. We claim that the space of orbits in $\mathscr{D}^{(1,1)}$ under the action of $SO(1,3)$ reduces to 3d weak isospin. Why? The 6d Lorentz group $SO(1,3)$ acts effectively on the 9d carrier space, leaving 3d of equivalence classes which can be identified with weak isospin as follows. Remember, the typical element is a traceless and symmetric $4 \times 4$ matrix. The spatial $3 \times 3$ block can be diagonalized via a spatial rotation, giving us something of the form,
\begin{equation}
	A = \begin{pmatrix} A_{11}+A_{22}+A_{33} & * & * & * \cr * & A_{11} & 0 & 0 \cr
		* & 0 & A_{22} & 0 \cr * & 0 & 0 & A_{33} \cr \end{pmatrix}.
\end{equation}
Now, consider a starting matrix of the form
\begin{equation}
	A = \begin{pmatrix} 0 & A_{01} & A_{02} & A_{03} \cr A_{01} & 0 & 0 & 0 \cr
		A_{02} & 0 & 0 & 0 \cr A_{03} & 0 & 0 & 0 \cr \end{pmatrix}
\end{equation}
and apply the Lorentz boost $\Lambda=\Lambda_3\Lambda_2\Lambda_1$ where
\begin{equation}
	\Lambda_1 = \begin{pmatrix} \cosh \phi_1 & \sinh \phi_1 & 0 & 0 \cr
		\sinh \phi_1 & \cosh \phi_1 & 0 & 0 \cr
		0 & 0 & 1 & 0 \cr 0 & 0 & 0 & 1 \cr \end{pmatrix},
\end{equation}
\begin{equation}
	\Lambda_2 = \begin{pmatrix} \cosh \phi_2 & 0 & \sinh \phi_2 & 0 \cr
		0 & 1 & 0 & 0 \cr \sinh \phi_2 & 0 & \cosh \phi_2 & 0 \cr
		0 & 0 & 0 & 1 \cr \end{pmatrix}
\end{equation}
and
\begin{equation}
	\Lambda_3 = \begin{pmatrix} \cosh \phi_3 & 0 & 0 & \sinh \phi_3 \cr
		0 & 1 & 0 & 0 \cr 0 & 0 & 1 & 0 \cr
		\sinh \phi_3 & 0 & 0 & \cosh \phi_3 \cr \end{pmatrix}.
\end{equation}
Then perform a spatial rotation to diagonalize the $3 \times 3$ block of $\Lambda_3 \Lambda_2 \Lambda_1 A \Lambda_1 \Lambda_2 \Lambda_3$. We now have nine parameters at our disposal which we can select so as to reproduce $A_{11}, A_{22}$ and $A_{33}$ while ensuring that the upper $1 \times 3$ block vanish identically. An arbitrary traceless symmetric 2-tensor can be obtained from this plus anything yielding
\begin{equation}
	\begin{pmatrix} 0 & A_1 & A_2 & A_3 \cr A_1 & 0 & 0 & 0 \cr A_2 & 0 & 0 & 0 \cr
		A_3 & 0 & 0 & 0 \cr \end{pmatrix}
\end{equation}
after the Lorentz transformation $\Lambda$. In other words, we may regard the $1 \times 3$ block spanned by $d^{tx}$, $d^{ty}$ and $d^{tz}$ as a complete set of equivalence classes of $\mathscr{D}^{(1,1)}/SO(1,3)$. These transform among themselves under rotations $SO(3) \subset SO(1,3)$, of course, where here we apply the rotation matrix $R \in SO(3)$ corresponding to the \textit{second-order} Lorentz transformation $\Lambda$ with $\Lambda_0=\Lambda_1=1$, $\Lambda_2=R$. Therefore, as far as the momentum intertwinor is concerned, the weak isospin charge $q \in \vvmathbb{R}$ parametrizes concentric spheres in the space of $SO(1,3)$-equivalence classes. 

Note, when the carrier space of the weak isospin gauge fields is viewed as a set of equivalence classes under the $SO(1,3)$ action, it becomes trivial that nothing happens upon thereafter performing a change of inertial frame (i.e., a first-order Lorentz transform). But there is nothing that says that the breakdown into irreducible parts $\mathscr{D}^{(0,0)}$ and $\mathscr{D}^{(1,1)}$ need be respected by a non-inertial change of frame. In this picture, a second-order Lorentz transform would just correspond to a gauge transformation in the conventional theory. Here is how we envision the scenario: we may consider that the parameters defining the second-order Lorentz transform are very slowing varying on the Planck scale although they may have material spatial variation on the Yang-Mills gauge scale, twenty orders of magnitude larger.

In summary, the analysis so far leads to a provisional identification of 2-jets with the carrier space on which gauge fields act via the ajoint representation as follows:
\begin{align}
	\mathrm{u}(1)_Y &= \mathscr{D}^{(0,0)}/SO(1,3) = \mathrm{span} \left( d^{tt}-d^{xx}-d^{yy}-d^{zz} \right) \\
	\mathrm{su}(2) &= \mathscr{D}^{(1,1)}/SO(1,3) = \mathrm{span}
	\left( d^{tx}, d^{ty}, d^{tz} \right)
\end{align}

An approximation is necessary. The weak field approximation taken by itself is insufficient; we will want moreover to impose a condition of slow spatio-temporal variation for reasons that will become apparent in a moment. Hence, we wish to neglect contributions cubic in $A$ or quadratic in $A$ but involving at least one spatial derivative (the contribution quadratic in $A$ itself cannot be dropped as it gives the self-coupling of the non-abelian gauge field).\footnote{In order to check the consistency of this approximation, we make an order-of-magnitude estimate. From the lifetime of the $W^\pm$ boson, $\tau = 3 \times 10^{-25} ~\mathrm{s}$, the characteristic range over which the weak forces act is subnuclear, $c\tau = 9 \times 10^{-15} ~\mathrm{cm}$. Therefore, the gauge potentials will be of the order of (in geometrical units) $A = (e/c\tau)(G^{1/2}/c^2) = 1.5 \times 10^{-20}$. Likewise, any spatial derivatives will be lesser in magnitude by a factor of $\ell_P/c\tau = 1.8 \times 10^{-19}$.} Now, in {\S}\ref{unification_of_gr_and_em}, the $g_{\mu,00}$ component of the metric plays the dominant role in determining electromagnetic force. Here too, similarly, we shall have to do with all $g_{\mu\beta}$, $|\beta|=2$ components at once. The notation employed suggests that the second-order components with $|\beta|=2$ need not correspond to the canonical choice of basis in terms of symmetrized pairs of spatial-temporal indices. Rather, we are free to take any basis we please and what proves expeditious is for us to decompose the second-order part of tangent space (non-canonically, but relative to the given coordinate system) into the irreducible components shown above. 

It will prove helpful to consider the jet affine connection 1-forms $\omega^\alpha_\beta$ in more detail. First, define
\begin{equation}
	\omega = \omega^\alpha_{\beta\lambda} e^\lambda \otimes E_\alpha \wedge e^\beta \in \mathscr{J}^{\infty *} \otimes \mathscr{J}^\infty \wedge \mathscr{J}^{\infty *}.
\end{equation}
Decompose into sectors as follows (with respect to Cartesian coordinates in a given inertial frame):
\begin{equation}\label{decomp_connection_1_form}
	\omega = \bigoplus_{1 \le |\alpha|,|\beta| < \infty} \omega^{(|\alpha|,|\beta|)}.
\end{equation}
Since $\omega$ is anti-symmetric in its last two indices, it could be regarded as an $\mathrm{so}(N_r,\vvmathbb{R})$-valued 1-form, where $N_r$ is the dimension of the space of jets up to order $r$. Perhaps it may be interpreted actively as yielding an infinitesimal rotation of the vielbein frame $e^\alpha$, $1 \le |\alpha| < \infty$, in other words, an infinitesimal motion of space including its infinitesimal directions. Notice that there is no conflict here with the Minkowskian structure, for we have been considering the jet affine connection forms with one index raised. If one considers instead $\omega_{\alpha\beta}$, then in place of purely spatial rotations alone one also has the possibility of infinitesimal Lorentz boosts. In the 1-jet case, then, the gauge group would be $\mathrm{so}(1,3)$. For higher jets, let $p$ be index of negative definite directions and $q$ the index of positive definite directions, where $p+q=N_r$. The relevant signature up to any order in the jets can be found by examining the generalized Minkowskian metric $\hat{\eta}$. Explicitly, up to third order we have
\begin{align}\label{jet_signature}
	J^1_- &= \mathrm{span}~ (d^t) \nonumber \\
	J^1_+ &= \mathrm{span}~ (d^x,d^y,d^z) \nonumber \\
	J^2_- &= \mathrm{span}~ (d^{tx},d^{ty},d^{tz}) \nonumber \\
	J^2_+ &= \mathrm{span}~ (d^{tt},d^{xx},d^{xy},d^{xz},d^{yy},d^{yz},d^{zz}) \nonumber \\
	J^3_- &= \mathrm{span}~ (d^{ttt},d^{txx},d^{txy},d^{txz},d^{tyy},d^{tyz},d^{tzz}) \nonumber \\
	J^3_+ &= \mathrm{span}~ (d^{ttx},d^{tty},d^{ttz},d^{xxx},d^{xxy},d^{xxz},d^{xyy},d^{xyz},d^{xzz},d^{yyy},d^{yyz},d^{yzz},d^{zzz}).
\end{align}
Thus, if one were to stop at second order, the gauge group would be isomorphic to $\mathrm{so}(4,10)$, while going to third order it would be isomorphic to $\mathrm{so}(11,23)$. It will turn out that a good deal of the analytical power of the theory will be rooted in the intersection between the block decomposition of the jet connection 1-forms in equation (\ref{decomp_connection_1_form}) and the pseudo-Riemannian structure of space-time in the infinitesimally small as reflected in equation (\ref{jet_signature}).

\subsubsection{Equation of Motion for a non-abelian Charged Body}\label{equation_of_motion}

Before we can proceed to derive the gauge field equations, a remark on principle is necessary. The Einstein-Hilbert action says nothing about how matter couples with the metric degrees of freedom. In present-day physical theory, one has no fundamental understanding of what matter is and therefore has to resort to plausible hypotheses, supported by heuristics and empirical conditions. The main observations to put forward in this connection are twofold: first, from the field equations of Maxwell and Einstein, one gathers that it must be a matter current that serves as source of the electromagnetic resp. gravitational field; and second, from the Hamiltonian formulation of the electrodynamics of moving charged bodies, it is evident that one wants to replace the ordinary charge current $J_0$ with $J_0+eA$, where $A$ is the electromagnetic potential (cf. Thirring \cite{thirring_vol_1}, {\S}{\S}5.1-5.5). 

Thus, the following hypothesis in connection with non-abelian gauge fields appears to be warranted: the relevant current to use is not the ordinary current $J_0$ itself but the gauge-index valued current from {\S}II.2.2. To connect with post-Einsteinian general theory of relativity, we invoke the jet affine connection 1-forms $\omega^\alpha_\beta$. Up to now, the $\alpha$ and $\beta$ have played the role of mere labels. In view of our objective of arriving at something like the non-abelian Lie-algebra valued current of conventional Yang-Mills theory, let us define (writing out explicitly the 1-form components as well)
\begin{equation}
	\omega := \omega^\alpha_{\lambda\beta} e^\lambda \otimes E_\alpha \otimes e^\beta \in \mathscr{J}^{\infty *} \otimes \mathscr{J}^\infty \otimes \mathscr{J}^{\infty *}.
\end{equation}
Then we have an enlarged sense of velocity as the mapping $\mathscr{J}^{\infty *} \otimes \mathscr{J}^\infty \rightarrow  \mathscr{J}^{\infty *} \otimes \mathscr{J}^\infty$ given by
\begin{equation}\label{enlarged_velocity_gtr}
	X = \left( X_1 + X_2 \right) \otimes \mathrm{id} + \omega \wedge.
\end{equation}
Here, $X_{1,2}$ are the 1-jet resp. 2-jet components of the jet field lying in $\mathscr{J}^{\infty *}$ obtained from a higher tangent vector field by lowering the index. If $Y$ is another element of $\mathscr{J}^{\infty *} \otimes \mathscr{J}^{\infty *} \otimes \mathscr{J}^\infty$, the composition of $X$ with $Y$ should be given in components by
\begin{equation}
	X \circ Y = i_{X^\alpha_\gamma} \text{\th} Y^\gamma_\beta + X^\alpha_\gamma \wedge Y^\gamma_\beta.
\end{equation}
The outcome of the above considerations is that we may regard our current in the enlarged sense to be a Lie-algebra-valued 1-form in the usual sense of Yang-Mills gauge theory. In view of the proposition and lemma immediately to follow, equation (\ref{enlarged_velocity_gtr}) becomes
\begin{equation}\label{enlarged_velocity_gauge}
	X = \left( X_1 + X_2 \right) \otimes \mathrm{id} + \frac{1}{2} \mathrm{ad}_{A \wedge}.
\end{equation}
The composition of two such currents $X \circ Y \in \Omega^2 \otimes \mathscr{J}^\infty \otimes \mathscr{J}^{\infty *}$ with the above identifications understood.

\begin{proposition}\label{gauge_covariant_deriv}
	The total covariant derivative restricted to 1-vector fields applied to a 1-form may be written as $\nabla = \nabla^1 + 2 A \wedge$. Here, $\nabla^1$ denotes the spatial part of the total covariant derivative which takes 1-vector fields as its argument and acts independent of the gauge index.
\end{proposition}
\begin{proof}
	Let $X_1$ be a vector field in the image of the canonical injection of $\mathscr{T}$ into $\mathscr{J}^\infty$ and $Y \in \mathscr{J}^{\infty*} \otimes \mathscr{J}^\infty \otimes \mathscr{J}^{\infty*}$ be a 1-form. Then since $X_1$ involves only first-order derivatives we compute via Leibniz' rule,
	\begin{equation}\label{current_cov_deriv}
		\nabla_{X_1} Y = \left( \nabla_{X_1} Y^\alpha_\beta \right) \otimes E_\alpha \otimes e^\beta +
		Y^\alpha_\beta \otimes \omega_\alpha^\gamma(X_1)E_\gamma \wedge e^\beta -
		Y^\alpha_\beta \otimes E_\alpha \wedge \omega^\beta_\delta(X_1) e^\delta.
	\end{equation}
	If we write out the contraction with $X_1$ more explicitly we find
	\begin{align}
		\frac{1}{2} Y^\alpha_\beta \omega^\gamma_\alpha(X_1)E_\gamma \wedge e^\beta &=
		Y^\alpha_{\beta\nu} e^\nu \omega^\gamma_{\alpha\lambda} X_1^\lambda \otimes E_\gamma \wedge e^\beta 
		\nonumber \\ &= -\frac{1}{2}
		Y^\beta_{\alpha\nu} e^\nu \omega^\alpha_{\gamma\lambda} X_1^\lambda \otimes E_\beta \wedge e^\gamma
		\nonumber \\ &= \frac{1}{2}
		\omega^\alpha_\beta Y_\alpha^\gamma(X_1) E_\gamma \wedge e^\beta. 
	\end{align}
	Compare with 
	\begin{equation}\label{expand_cov_deriv}
		\mathrm{tr}~ \omega \wedge Y = \omega^\alpha_{\beta\lambda} e^\lambda \wedge Y^\beta_{\gamma\nu} e^\nu \otimes E_\alpha \wedge e^\gamma.
	\end{equation}
	We see that the penultimate term in equation (\ref{current_cov_deriv}) is obtained by contraction of this with $X_1$: $i_{X_1} \omega \wedge Y = \omega(X_1)Y - \omega Y(X_1)$. The last term in equation (\ref{current_cov_deriv}) may be handled similarly and supplies the factor of 2.
	
	Now consider applying $\nabla_{X_1}$ to a product of 1-forms $Y^{(1)} \wedge \cdots \wedge Y^{(k)}$ (its action on an arbitrary differential form may be found by extending linearly). As before since $X_1$ is of first order only, we may expand without cross-terms thus,
	\begin{align}
		\nabla_{X_1} Y &= \left( \nabla_{X_1} Y^{(1)\alpha}_{\beta_1} \wedge \cdots \wedge Y^{(k)\beta_{k-1}}_\beta \right) \otimes E_\alpha \wedge e^\beta +
		Y^{(1)\alpha}_{\beta_1} \wedge \cdots \wedge Y^{(k)\beta_{k-1}}_\beta \otimes 
		\omega^\gamma_{\beta_{k-1}}(X_1) E_\gamma \wedge e^\beta \nonumber \\
		& - Y^{(1)\alpha}_{\beta_1} \wedge \cdots \wedge Y^{(k)\beta_{k-1}}_\beta \otimes 
		E_{\beta_{k-1}} \wedge \omega^{\beta_{k-1}}_\delta(X_1) e^\delta.	
	\end{align}
	Therefore,
	\begin{align}
		\nabla_{X_1} Y^{(1)} \wedge \cdots \wedge Y^{(k)} &= \nabla^1  Y^{(1)} \wedge \cdots \wedge Y^{(k)} +  Y^{(1)} \wedge \cdots \wedge Y^{(k)} \wedge \omega \nonumber \\
		&=  \nabla^1  Y^{(1)} \wedge \cdots \wedge Y^{(k)} + \omega \wedge Y^{(1)} \wedge \cdots \wedge Y^{(k)}.
	\end{align}
	Here, we may commute the $\omega$ past factors of $Y^{(j)}$ because
	\begin{align}
		X^\alpha_\gamma \wedge Y^\gamma_\beta &= X^\alpha_{\gamma\nu} e^\nu \wedge Y^\gamma_{\beta\lambda} e^\lambda \nonumber \\
		&= - Y^\gamma_{\beta\nu} e^\nu \wedge X^\alpha_{\gamma\lambda} e^\lambda \nonumber \\
		&= - Y^\beta_{\gamma\nu} e^\nu \wedge X^\gamma_{\alpha\lambda} e^\lambda \nonumber \\
		&= Y^\alpha_{\gamma\nu} e^\nu \wedge X^\gamma_{\beta\lambda} e^\lambda = Y^\alpha_\gamma \wedge X^\gamma_\beta.
	\end{align}
	In view of the definition, we have derived for the total covariant derivative acting on Lie-algebra-valued 1-forms the simple expression $\nabla = \nabla^1 + A \wedge$, where henceforward with denote the jet affine connection 1-forms with the letter $A$ when they are to be viewed as being valued in the gauge Lie algebra.
\end{proof}
\begin{remark}
	Note, $\nabla^1 = \text{\th}$ if we can treat gravitational forces as small compared to the other gauge forces.
\end{remark}

\begin{lemma}\label{wedge_to_ad}
	The wedge product in the covariant derivative of proposition \ref{gauge_covariant_deriv} may be replaced with the adjoint action; thus,
	\begin{equation}
		X^\alpha_\gamma \wedge Y^\gamma_\beta = \frac{1}{2} \mathrm{ad}_{X \wedge} Y,
	\end{equation}
	where the $\wedge$ indicates that a wedge product is to be taken with respect to the spatial indices.
\end{lemma}
\begin{proof}
	It is evident that
	\begin{align}
		X^\alpha_\gamma \wedge Y^\gamma_\beta &= \frac{1}{2} X^\alpha_\gamma \wedge Y^\gamma_\beta + \frac{1}{2} X^\alpha_\gamma \wedge Y^\gamma_\beta \nonumber \\
		&= \frac{1}{2} X^\alpha_{\gamma\nu} e^\nu \wedge Y^\gamma_{\beta\lambda} e^\lambda + \frac{1}{2} X^\alpha_{\gamma\nu} e^\nu \wedge Y^\gamma_{\beta\lambda} e^\lambda \nonumber \\
		&= \frac{1}{2} \left( X_\nu Y_\lambda + X_\nu^t Y_\lambda^t \right) e^\nu \wedge e^\lambda \nonumber \\
		&= \frac{1}{2} \left( X_\nu Y_\lambda + (Y_\lambda X_\nu)^t \right) e^\nu \wedge e^\lambda \nonumber \\
		&= \frac{1}{4} \left( X_\nu Y_\lambda +  Y_\nu X_\lambda + (X_\lambda Y_\nu + Y_\lambda X_\nu)^t \right) e^\nu \wedge e^\lambda \nonumber \\
		&= \frac{1}{4} \left(  X_\nu Y_\lambda -  Y_\lambda X_\nu - (X_\nu Y_\lambda -  Y_\lambda X_\nu )^t \right) e^\nu \wedge e^\lambda \nonumber \\
		&= \frac{1}{4} \left( \mathrm{ad}_{X_\nu} Y_\lambda - \left( \mathrm{ad}_{X_\nu} Y_\lambda \right)^t \right) e^\nu \wedge e^\lambda \nonumber \\
		&= \frac{1}{2} \mathrm{ad}_{X_\nu} Y_\lambda e^\nu \wedge e^\lambda \nonumber \\
		&= \frac{1}{2} \mathrm{ad}_{X \wedge} Y,
	\end{align}
	as was to be shown. Due to the commutator identity for matrices in the gauge indices, namely,
	\begin{equation}
		[X,Y^{(1)} \cdots Y^{(k)}] = [X,Y^{(1)}]Y^{(2)}\cdots Y^{(k)} + \cdots Y^{(1)}\cdots [X,Y^{(k)}],
	\end{equation}
	the result extends from 1-forms to $k$-forms as well.
\end{proof}

\begin{remark}
	A nice feature of the derivation is that it explains the reason behind the so-called gauge covariant derivative, which otherwise has to be put in by hand on empirical grounds and the origin of which is somewhat mysterious in conventional gauge theory. 
\end{remark}

We are now prepared to introduce the canonical momentum current 1-form. Taking our cue from the canonical generalized momentum in electrodynamics, we want it to describe the transport of matter by means of two terms, the first being the ordinary momentum, proportional to the product of rest mass and velocity, and the second reflecting the interaction of charge with the gauge field, proportional to the product of a charge coefficient and the gauge potential. Here, a macroscopic body is to be understood as a collection of elementary particles, each of which transforms in a low-lying representation of the gauge Lie algebra. Their collective effect is to be treated phenomenologically by the following device: the macroscopic body is to be obtained from an excitation of the underlying quantum fields in a many-body state which transforms in a high-lying product representation, such that the spacing between the roots in the Cartan subalgebra becomes small enough on the macroscopic scale to be approximated as a continuum. Then, the net effect of the many-body state results in a continuous charge coefficient for each independent axis in the Cartan subalgebra (which need not be semi-simple but will consist in a direct sum of the Cartan subalgebras corresponding to each level in the jets). Thus, the momentum of the macroscopic body may be obtained from the 1-form current $X$ by applying a charge coefficient (denoted $q$) to each component of $X$, where the coefficient to use depends on which irreducible piece of the $SO(1,3)$ representation the component of $X$ in question lies in. Now, the charge acts as an operator on the gauge indices only. We further suppose it to be symmetric with respect to the Killing inner product. For future reference, the Killing form for $\mathrm{so}(1,3)$ in the adjoint representation satisfies
\begin{equation}\label{killing_form}
	K( \mathrm{ad}_X \circ \mathrm{ad}_Y) = 2 ~\mathrm{tr}~ XY.
\end{equation}
Especial care must be taken to understand what is meant by the trace in this context. The point is that the Killing form acts on $X$ and $Y$ considered as belonging to $\mathrm{so}(1,3)$ which can be obtained from the jet connection 1-forms $\omega^\alpha_\beta$ by \textit{lowering} the first index. But in what follows it is more convenient to regard the $X$ and $Y$ as belonging to $\mathrm{so}(4)$, and this is what we shall do. If so, then to form the Killing inner product properly one must now lower the index on $XY$ after multiplying the matrices in $\mathrm{so}(4)$. That means that the trace in equation (\ref{killing_form}) should be understood as involving the Minkowskian metric, i.e., as actually $\mathrm{tr} ~\hat{\eta} XY$ where $X, Y \in \mathrm{so}(4)$. This observation is not supererogatory as including the Minkowskian factor is necessary not only to make the signs come out right in the formulae below, but also to derive the correct expression for the Killing form in the appropriate basis in the first place (the reader is invited to experiment with what happens if the Minkowskian factor were to be omitted). 

In addition, we wish to include the cosmological term since it will prove to be relevant. Suppose the body to carry a non-abelian charge $I$ and to move in the gauge potential $A$, where we employ a shorthand notation,
\begin{equation}
	I = \sum_{\alpha<\beta} I_\alpha^\beta B^{|\alpha|+|\beta|-2} e^\alpha \wedge e_\beta; \qquad A = \sum_{\alpha<\beta} A_\alpha^\beta B^{|\alpha|+|\beta|-2} e^\alpha \wedge e_\beta.
\end{equation}
Thus, schematically $P=qX$ which when written out explicitly gives,
\begin{align}\label{momentum_current_density}
	P &= q \left( (X_1+X_2) \otimes \left( \mathrm{id} + I \right)+ \mathrm{ad}_A \right) \nonumber \\
	&= 
	m_0 (X_1+X_2) \otimes \mathrm{id} + (X_1+X_2) \otimes q_0 I + \mathrm{ad}_{q_0 A} \nonumber \\
	&= 
	\frac{1}{2} \lambda^2 m_1 (X_1+X_2) \otimes \mathrm{id} + (X_1+X_2) \otimes \lambda^2 q_1 I + \mathrm{ad}_{\frac{1}{2} \lambda^2 q_1 A}.
\end{align}
Thereby the interaction term to use in the Lagrangian becomes,\footnote{Note that we adopt a sign convention for the charges such that the interaction appears with a sign the same as that with which $R_{\alpha\beta} \wedge * R^{\alpha\beta}$ enters the Einstein-Hilbert action so that the timelike directions will be \textit{positive} definite, as is required on physical grounds (see the right hand side of the field equation in theorem \ref{inhomogeneous_equ} below).}
\begin{align}\label{gauge_matter_coupling}
	\mathscr{L}_{\mathrm{int}} =& - \frac{\lambda^2/2}{4} K( X \wedge *P) \nonumber \\
	=& - \frac{1}{2} ~\mathrm{tr}~ \left( \frac{1}{2} (X_1 + X_2) \otimes \mathrm{id} + (X_1+X_2) \otimes I + \mathrm{ad}_A \right) \wedge \nonumber \\
	&* \frac{1}{4} \lambda^4 \left( m_1 (X_1 + X_2) \otimes \mathrm{id} + (X_1+X_2) \otimes q_1 I + \mathrm{ad}_{q_1 A} \right),
\end{align}
where in taking the trace we contract indices in gauge field space using the metric tensor. In practice, in the weak-field limit we may employ the Minkowskian tensor for this purpose since the product of any gravitational term with a gauge potential will be negligible. At present, no justification can be given for this Ansatz for the matter-field coupling other than the fact that it leads to the right field equation. The derivation to be rehearsed in a moment shows how the right hand side of the gauge field equation of motion may be seen as coming from the 2-jet part of the momentum. The term proportional to mass alone separates out to yield the expression previously used in Part II. Here, the interaction picks up another term quadratic in the vector potentials.

The problem of the motion of a classical body bearing non-abelian charges has been considered before by Wong \cite{wong}, who derives an equation of motion equivalent to ours to be advanced in proposition \ref{derivation_of_non_abelian_equ_of_motion} below by considering the classical limit of the quantum field theory (for $SU(2)$ but his argument evidently applies to any semi-simple Lie group). Here, we derive its equation of motion from what seems to be a natural action principle. To this end, let the momentum current by given by
\begin{equation}
	P = ( q_0 + m ) X \otimes I + (q_0 + m) A
\end{equation}
and define the contribution of the body to the action by
\begin{equation}\label{non_abelian_action}
	A = - \frac{\lambda^2/2}{2m} \int K(g(P,P)) \theta^s. 
\end{equation}
Now we may expand the integrand as
\begin{align}
	\frac{1}{2m} K(g(P,P)) &= \frac{1}{2m} K g((q_0+m) X + (q_0+m) A^\alpha_\gamma,(q_0+m) X + (q_0+m) A^\gamma_\beta) \nonumber \\
	&= K \left( \frac{1}{2} m g(X,X) +  q_0 g(X,X) + \frac{q_0^2}{2m} g(X,X) + \right. \nonumber \\
	& \qquad \left. \frac{1}{2} m g(A,A) + q_0 g(A,A) +  \frac{q_0^2}{2m} g(A,A) + \right. \nonumber \\
	& \qquad \left. \frac{1}{m} g(X,A) + 2 q_0 g(X,A) + \frac{q_0^2}{m} g(X,A) \right) \nonumber \\
	& = K \left( \frac{1}{2} m g(X+A,X+A) + q_0 g(X,X) + 2 q_0 g(X,A) + \right. \nonumber \\
	& \qquad \left. q_0 g(A,A) + \frac{q_0^2}{2m} g(X+A,X+A) \right). 
\end{align}

\begin{remark}
	Notice, of course, that for any two currents $Y$ and $Z$, $g(Y,Z) = \int Y \wedge * Z,$ where the integral extends over space. Thus, equation (\ref{non_abelian_action}) is equivalent to our earlier definition in equation (\ref{gauge_matter_coupling}), if the term quadratic in the charge may be neglected. As we shall see in a moment, this will in fact be the case for all known elementary particles and a fortiori for bodies composed of them. To see how this comes about, implement the following replacements:
	\begin{align}
		\frac{1}{2m} K(g(P,P)) &= K \left( 
		\frac{1}{2} m g(X \otimes I + A, X \otimes I + A) + 
		\frac{1}{2} \lambda^2 \frac{q_0}{\frac{1}{2} \lambda^2} g(X \otimes I, X \otimes I) + \lambda^2 \frac{q_0}{\frac{1}{2} \lambda^2} g(X \otimes I,A) + \right. \nonumber \\
		& \qquad \left. \lambda^2 \frac{q_0}{\frac{1}{2} \lambda^2} g(A,A) + \frac{1}{4} \lambda^4 \frac{(q_0/\frac{1}{2}\lambda^2)^2}{2m} g(X \otimes I + A, X \otimes I + A). \right) \nonumber \\
		&= \frac{1}{2} m_1 g(X,X) + K \left(
		 \frac{1}{2} \lambda^2 q_1 g(X \otimes I,X \otimes I) + \lambda^2 q_1 g(X \otimes I,A) + \right. \nonumber \\
		& \qquad \left. \frac{1}{2} \lambda^2 q_1 g(A,A) + \frac{1}{4} \lambda^4 \frac{q_1^2}{2m} g(X \otimes I + A,X \otimes I + A) \right),
	\end{align}
	with $m_1 = K(I,I) m = (I \cdot I) m $ and where we use the fact that $A \ll X$ (see the footnote in {\S}\ref{from_jets_to_gauge} above, where we estimate that $A$ is of order $10^{-20}$ while $X$ is of order unity). That is, the observed mass $m_1$ is to be treated as arising from a weighted contribution of $m$ per gauge degree of freedom.
	
	Now, this expression idealizes the body as a point particle whereas equation (\ref{gauge_matter_coupling}) is written for a continuous distribution of charge.  To pass from one to the other, suppose that a body of total mass $M$ and total charge $Q$ may be represented as a collection of smaller pointlike constituents of mass $m_1$ and charge $q_1$; so $M = \sum m_1$ and $Q = \sum q_1$, occupying a volume $V$. Take a limit as $N$ and $V$ tend to infinity in such a way as to ensure that $N/V$ tends to $n$, the local number density of representative points (by infinite, we mean large compared to the intercomponent separation but still small compared to the typical size of the body). Then, one obtains the mass density $\varrho$ and the charge density $\sigma$ as follows:
	\begin{equation}
		\varrho = \frac{1}{V} \sum m_1 \rightarrow m_1 n; \qquad \sigma = \frac{1}{V} \sum q_1 \rightarrow q_1 n,
	\end{equation}
	but in the same limit the quadratic terms behave as
	\begin{equation}
		\frac{1}{V} \sum \frac{q_1^2}{2m_1} = \frac{1}{V} \sum \frac{(Q/N)^2}{2(M/N)} =\frac{1}{V} \sum \frac{1}{N} \frac{Q^2}{2M} \rightarrow \frac{1}{N} \frac{Q^2 n}{2M} \rightarrow 0,
	\end{equation}
	i.e., they become negligible in the continuum limit and thus equation (\ref{non_abelian_action}) is tantamount to equation (\ref{gauge_matter_coupling}).
\end{remark}

\begin{proposition}\label{derivation_of_non_abelian_equ_of_motion}
	The equation of motion for a non-abelian charged body is given by $\nabla_X P_1 = 0$ with $P_1 = m_1 X + q A \cdot I$.
\end{proposition}
\begin{proof}
	The proof follows along lines similar to what has already been done above in theorem \ref{jet_geodesic_from_action}. If $Z$ is a variation field, we have
	\begin{align}
		g(X \otimes I +Z \otimes I + A,X \otimes I + Z \otimes I + A) - g(X \otimes I + A, X \otimes I + A) = 2 g(Z \otimes I, X \otimes I + A) + o(Z).
	\end{align}
	From the condition of stationarity,
	\begin{align}
		0 & = X \int K \langle q Z \otimes I, q X \otimes I + q A \rangle \nonumber \\ 
		&= \int K \bigg( \langle q \nabla_X Z \otimes I, q X \otimes I + q A \rangle + \langle q Z \otimes I, q \nabla_X X \otimes I + q \nabla_X A \rangle \bigg).
	\end{align}
	But due to the vanishing of torsion we may substitute
	\begin{align}
		\nabla_X Z = \nabla_Z X + X Z - Z X \nonumber \\
		\nabla_X A = \nabla_Z A + A Z - Z A.
	\end{align}
	As before, this leads to
	\begin{equation}
		\int K \bigg(\langle X Z \otimes I - Z X \otimes I - A Z \otimes I + Z \otimes I A, m X \otimes I + q_1 A \rangle + \langle m Z \otimes I, m \nabla_X X \otimes I + q_1 \nabla_X A \rangle \bigg) = 0,
	\end{equation}
	whence
	\begin{equation}
	  m \nabla_X X I \cdot I + q \nabla_X A \cdot I = m_1 \nabla_X X + q \nabla_X A \cdot I = 0,
	\end{equation}
	as was to be shown. Here, we employ the rescaled mass $m_1$ in the expression for momentum.
\end{proof}
Therefore, the momentum current relevant to the motion of a macroscopic body results from a sum over the gauge potentials, each entering in proportion to its charge. From the empirical point of view, the only requirement is that we recover the Lorentz force law of classical electrodynamics, since no macroscopic classical bodies carrying non-zero non-abelian charges have ever been observed.

\begin{proposition}\label{non_abelian_equ_of_motion}
	The equation of motion for a charged body $\nabla_X P_1=0$ reduces to a sum of a gravitational force plus the Yang-Mills version of the Lorentz force law, viz., $m_1 \nabla_X X = q ( F \cdot I ) X$, where $F$ is the field-strength tensor given by equation (\ref{non_abelian_cartan}).
\end{proposition}
\begin{proof}
	Substitute the momentum $P_1$ into the result of proposition \ref{gauge_covariant_deriv}, restricting to the 1-jet part of the current:
	\begin{align}
		\nabla P_1 &= \nabla \left( m_1 X_1 + q A \cdot I \right) \nonumber \\
		&= \left( \text{\th} + A \wedge \right) \left( m_1 X_1 + q A \cdot I \right) \nonumber \\
		&= m_1 \left( \text{\th} + A \wedge \right) X_1 + q \text{\th} A \cdot I + q A \wedge A \cdot I \nonumber \\
		&= m_1 \nabla^1 X_1 + q F \cdot I.
	\end{align}
	Therefore, $m_1 \nabla^1 X_1 = -q F \cdot I$. If we now contract the total covariant derivative with the 1-vector field $X_1$ we find
	\begin{equation}
		m_1 \nabla^1_{X_1} X_1 = -q X_1 F \cdot I = q ( F \cdot I) X_1,
	\end{equation}
	where in the intermediate expression $X_1$ contracts against the first spatial index of the field-strength 2-form while in the final result it contracts against the second (thereby introducing a relative minus sign due to the antisymmetry of $F$).
\end{proof}

\subsubsection{Derivation of the Yang-Mills Field Equations}\label{yang_mills_derivation}

As we now show, the Einstein-Hilbert action of {\S}II.2.2 along with our postulated interaction term lead directly to field equations of Yang-Mills type when the higher jet degrees of freedom are interpreted as gauge fields. In {\S}\ref{equation_of_motion}, we change the notion for the affine jet connection from $\omega^\alpha_\beta$ to $A^\alpha_\beta$ in order to emphasize that we wish to regard the latter as the vector potential of a gauge field. Accordingly also denote the curvature 2-form by $F^\mu_\nu = \Omega^\mu_\nu$. Thus Cartan's second structural equation reads,
\begin{equation}\label{curvature_gtr}
	F^\mu_\nu = \text{\th} A^\mu_\nu + A^\mu_\lambda \wedge A^\lambda_\nu,
\end{equation}
or in a compressed notation,
\begin{equation}\label{non_abelian_cartan}
	F = \text{\TH} A := \text{\th} A + A \wedge A.
\end{equation}
Equation (\ref{non_abelian_cartan}) looks familiar: it resembles the going expression for the field strength tensor in Yang-Mills gauge theory, cf. Thirring \cite{thirring_vol_2}. Here, the difference consists in the fact that, for us, equation (\ref{non_abelian_cartan}) involves space-time currents valued in 2-forms whereas in Yang-Mills gauge theory one could, in principle, contemplate any simple gauge group or direct product of simple gauge groups (as happens in the standard model). 

\begin{proposition}\label{cov_deriv_square} 
	The covariant exterior derivative satisfies the relation, $\text{\TH}^2 = \Omega \wedge$.
\end{proposition}
\begin{proof}
	Let $\eta$ be a differential form to which we apply $\text{\TH}$ twice:
	\begin{align}
		\text{\TH}^2 \eta &= \left( \text{\th} + A \wedge \right) \left( \text{\th} + A \wedge \right) \eta \nonumber \\ &= \text{\th}^2 \eta + \left( A \wedge \text{\th} + \text{\th} \wedge A \wedge \right) \eta + A \wedge A \wedge \eta.
	\end{align}
	Now, we want $(\text{\th}A + A \wedge A ) \wedge \eta$ on the right hand side. In order to see why one ends up with this, write
	\begin{equation}
		\text{\th} \left( A \wedge \eta \right) = \left( \text{\th} A \right) \wedge \eta - A \wedge \text{\th} \eta + \cdots;
	\end{equation}
	but the cross-terms cancel when $A$ is a pure 1-form. If $\eta$ is pure odd, this happens because $A \wedge \eta = - \eta \wedge A$. If $\eta$ is pure even, the cross-terms cancel for a different reason, namely, that they involve $(-1)^k$ where $k=1$ on the left hand side and $k$ is even on the right hand side. Therefore we have that $\text{\TH}^2 \eta = \Omega \wedge \eta = F \wedge \eta$ (under the notational convention here in force).
\end{proof}

\begin{corollary}\label{homogeneous_equ}
	In the indicated approximation, non-abelian gauge fields obey the homogeneous Maxwell's relation $\text{\TH} F = 0$.
\end{corollary}
\begin{proof}
	From the expression for $\text{\TH} F = \text{\TH}^2 A$ we have
	\begin{align}
		\text{\TH} F &= \left( \text{\th} + A \wedge \right)  \left( \text{\th} + A \wedge \right) A \nonumber \\
		&= \text{\th}^2 A + A \wedge \text{\th} A + \text{\th} A \wedge A + A \wedge A \wedge A,
	\end{align}and we see that the additional terms beyond $\text{\th}^2 A$ are of third order in $A$ together with derivatives which are negligible to the degree of approximation with which we are working. But $\text{\th}^2=0$ identically; hence, effectively $\text{\TH}F=0$.
\end{proof}

\begin{lemma}
	The perturbation of the field strength $F_A$ depending on the vector potential $A$ when $A$ goes to $A + t a$, where $t$ is a small real parameter, may be expressed by the following formula:
	\begin{equation}
		F_{A+ta} = F_A + t \text{\TH}_A a + \frac{1}{2} t^2 \mathrm{ad}_a a.
	\end{equation}
\end{lemma}
\begin{proof}
	This corresponds to a standard result. In view of lemma \ref{wedge_to_ad}, we may write out,
	\begin{align}
		F_{A+ta} &= \text{\th} (A+ta) + \frac{1}{2} \mathrm{ad}_{A+ta} (A+ta) \nonumber \\
		&= \text{\th} A + t \text{\th} a + \frac{1}{2} \mathrm{ad}_A A + \frac{1}{2} t \mathrm{ad}_a A + \frac{1}{2} t \mathrm{ad}_A a + \frac{1}{2} t^2 \mathrm{ad}_a a \nonumber \\
		&= F_A + t \left( \text{\th} + \mathrm{ad}_A \right) a + \frac{1}{2} t^2 \mathrm{ad}_a a \nonumber \\
		&= F_A + t \text{\TH}_A a + \frac{1}{2} t^2 \mathrm{ad}_a a.
	\end{align}
	This completes the proof.
\end{proof}

\begin{theorem}\label{inhomogeneous_equ}
	In the indicated approximation again, non-abelian gauge fields obey as well Maxwell's inhomogeneous relation $-\text{\TH}^*_{-A} * F = - 2 \pi \varkappa \lambda^4 B^2 * J_1$, where $\text{\TH}^*$ indicates the adjoint operator corresponding to $\text{\TH}$. Note: the $J_1$ on the right hand side represents an ordinary 1-form quantifying the current as source of the field, as is usual in classical electrodynamics.
\end{theorem}
\begin{proof}
	Here, we wish to regard the Einstein-Hilbert action as a functional of $e^\alpha$ and $\omega^\alpha_\beta$ as independent of each other and to secure their proper relationship via an undetermined multiplier $f$. Thus (cf. equation (II.2.25)),
	\begin{align}
		\mathscr{L}_{EH} = &\frac{1}{16 \pi \varkappa} \left( R_{\alpha\beta} + \lambda^2 B^{|\alpha|+|\beta|-2} e_\alpha \wedge e_\beta \right) \wedge * 
		\left(  R^{\alpha\beta} + \lambda^2 B^{|\alpha|+|\beta|-2} e^\alpha \wedge e^\beta \right) + \nonumber \\
		& f \left( \text{\th} e^\alpha - e^\beta \wedge \omega^\alpha_\beta \right) \wedge * \left(  e_\alpha - e_\beta \wedge \omega_\alpha^\beta \right).
	\end{align}	
	Variation with respect to $e^\alpha$ and $\omega^\alpha_\beta$ separately yields
	(cf. equation (II.2.47))
	\begin{align}
		\delta e^\alpha &\wedge \left[ 2 t_\alpha +
		\frac{1}{8 \pi \varkappa} * e_{\alpha\beta\gamma} \wedge R^{\beta\gamma} +
		\frac{1}{16 \pi \varkappa} \lambda^2 * e_\alpha - 2 \text{\th}^* f * \left(
		\text{\th} e_\alpha - e_\beta \wedge \omega_\alpha^\beta
		\right) \right] = 0 \\
		&2 \delta F \wedge * F + \frac{1}{8 \pi \varkappa} \delta R_{\alpha\beta} \wedge * e^{\alpha\beta} + \delta \mathscr{L}_{\mathrm{int}} - 2 f e^\alpha \wedge *  \left(
		\text{\th} e_\alpha - e_\beta \wedge \omega_\alpha^\beta 
		\right) = 0.
	\end{align}	
	The undetermined multiplier can be eliminated on both lines by imposing the condition that
	\begin{equation}
		\text{\th} e^\alpha - e^\beta \wedge \omega^\alpha_\beta = 0
	\end{equation}
	identically. Then stationarity with respect to the variation $\delta e^\alpha$ produces the gravitational field equations as before (in theorem II.2.1 above), while stationarity with respect to the varation $\delta \omega$ produces another set of field equations we have now to derive. Notice that, as before, the second term on the second line drops out again in view of lemma II.2.1. Next evaluate the variation of the curvature term with respect to $A$:	
	\begin{align}
		\delta \left( F \wedge *F \right) &= 2 \delta F \wedge *F \nonumber \\
		&= 2 \left( \text{\th} \delta A + \delta A \wedge A + A \wedge \delta A \right) \wedge *F \nonumber \\
		&= 2 \delta A \wedge \left( - \text{\th}^* + 2 A \wedge \right) *F \nonumber \\
		&= 2 \delta A \wedge \left( - \text{\TH}^*_{-A} *F \right),
	\end{align}	
	since when we integrate by parts we have that under the integral sign,
	\begin{align}
		\left( \text{\th} \delta A \right) \wedge *F &= d^\alpha \wedge \partial_\alpha \delta A \wedge *F \nonumber \\
		&= - \partial_\alpha \delta A \wedge d^\alpha \wedge *F \nonumber \\
		&= - (-1)^{|\alpha|} \delta A \wedge d^\alpha \partial_\alpha \wedge *F \nonumber \\
		&= -\delta A \wedge \text{\th}^* *F.
	\end{align}
	As for the variation of the interaction term, we have from
	\begin{align}
		\mathscr{L}_{\mathrm{int}} 
		=& - \frac{1}{2} ~\mathrm{tr}~ \left( (X_1 + X_2) \otimes \mathrm{id} + (X_1+X_2) \otimes I +  \mathrm{ad}_A \right) \wedge \nonumber \\
		&* \frac{1}{4}\lambda^4 \left( m (X_1 + X_2) \otimes \mathrm{id} + \frac{1}{2} (X_1+X_2) \otimes q_1 I + \mathrm{ad}_{\frac{1}{2} q_1 A} \right)
	\end{align}
	that
	\begin{align}
		\delta \mathscr{L}_{\mathrm{int}} &= - \frac{1}{2} ~\mathrm{tr}~
		\left(
		\mathrm{ad}_{\delta A} \wedge * P + \left( (X_1+X_2) \otimes \mathrm{id} + (X_1+X_2) \otimes I \right) \wedge * ~\mathrm{ad}_{\frac{1}{2} \lambda^2 q_1 \delta A} \right) \nonumber \\
		&= - \frac{1}{2} ~\mathrm{tr}~ \delta A \wedge * \left( P + \frac{1}{2} \lambda^2 m (X_1+X_2) \otimes \mathrm{id} + \frac{1}{2} (X_1+X_2) \otimes \lambda^2 q_1 I \right) \nonumber \\
		&= - \frac{1}{2} ~\mathrm{tr}~ \delta A \wedge * \left( m (X_1+X_2) \otimes \mathrm{id} + (X_1+X_2) \otimes q_1 I + \mathrm{ad}_{\frac{1}{2} q_1 A} \right) \nonumber \\
		&= - \mathrm{tr}~ \delta A \wedge * \frac{1}{4} \lambda^4 J_1
	\end{align}
	In the last line, the term proportional to the mass can be ignored because it lies on the diagonal, but $\delta A$ does not have any variation there. Therefore for the total variation we find
	\begin{equation}
		\mathrm{tr}~ \delta A \wedge \left[\frac{2}{16\pi\varkappa} \left( - \text{\TH}^*_{-A} *F \right) + * \frac{1}{4} \lambda^4 J_1 \right] = 0.
	\end{equation}
	Equating the term inside the brackets to zero yields
	\begin{equation}\label{final_field_equ}
		- \text{\TH}^*_{-A} *F = - 2 \pi \varkappa \lambda^4 B^2 * J_1,
	\end{equation}
	which is what was to be shown.
\end{proof}
\begin{remark}
	In the last line, we pull out a factor of $\frac{1}{4} \lambda^4 B^2$ in front. This implies that the momentum current that appears on the right hand side of the field equation becomes,
	\begin{equation}\label{effective_momentum_current}
		J_1 = q_1 (X_1+X_2) \otimes I + \mathrm{ad}_{q_1 A}.
	\end{equation}
	Rewrite the coefficient of the velocity and the gauge potential as $q_1 = \sigma_1 m_1 = \sigma_0 m_0/\frac{1}{2} \lambda^2$, where $m$ denotes the rest mass of the moving body. For typical macroscopic bodies, the electrical charges of their constituent electrons and protons nearly cancel, leading to a charge-to-mass ratio small compared to unity. If however, for an elementary particle in which no such cancellation occurs, we suppose the charge-to-mass ratio $\sigma_0$ to be of order unity, then the effective charge-to-mass ratio $\sigma_1$ would be of order $1/\frac{1}{2} \lambda^2 \ell_P^2 = 3.4624 \times 10^{121}$. Thus, our proposed form of the dark energy along with the smallness of the cosmological constant would explain why electromagnetic forces are so much stronger than gravity.\footnote{Even though in Planck units the absolute charge-to-mass ratio $\sigma_0$ will still be very small compared to unity. The so-called hierarchy problem comprises two questions, first, as to why the electroweak scale is so small compared to the Planck scale and, second, as to why the cosmological constant is so small in Planck units. Equations (\ref{final_field_equ}) and (\ref{effective_momentum_current}) simplify the situation by suggesting that the two are aspects of the same thing (vide {\S}\ref{spont_symm_breaking}) below).} For this same reason, we shall suppose the cosmological term in equation (\ref{effective_momentum_current}) to be negligible. It is not entirely negligible in the gravitational (i.e., 1-jet) sector of the field equation, where the cosmological constant may be regarded as contributing to the mass density of matter in the universe an additional term going as
	\begin{equation}
		8 \pi \varkappa \varrho_0 = \Lambda = \frac{1}{2} \lambda^2.
	\end{equation} 
	This cosmological background density will play a role in the mechanism of spontaneous symmetry breaking, as we shall see below in {\S}\ref{spont_symm_breaking}. 
\end{remark}

For the remainder of this paper, we shall be interested only in the leading behavior in which spatial differential forms are limited to jets of first order. In this case, 
\begin{equation}
	- \text{\TH}^*_{-A} = - \text{\th}^* + \mathrm{ad}_A = - (-1)^1 \text{\th} + \mathrm{ad}_A = \text{\th} + \mathrm{ad}_A = \text{\TH}_A,
\end{equation}
and, as discussed above, we recover the non-abelian Yang-Mills equation as usually stated for the gauge group $\mathrm{so}(p,q)$ if we set $|\alpha|=|\beta|=2$. In other words, the field equation in the (2,2)-sector becomes just
\begin{equation}\label{yang_mills_field_equ}
	\text{\TH}_A *F = - 2 \pi \varkappa \lambda^4 B^2 * J_1.
\end{equation} 

\begin{remark}
	It is no surprise that a field equation of Yang-Mills type would result from the presence of the $R_{\alpha\beta} \wedge * R^{\alpha\beta}$ term in the Einstein-Hilbert action. What is novel here is that, in the post-Einsteinian general theory of relativity, it arises in a \text{natural} way  as soon as one identifies higher jets with gauge indices. Two consequences flow from this finding: 1) the presence of non-abelian gauge fields in the standard model receives a compelling justification in terms of an intrinsically \textit{spatio-temporal origin}, whereas in the conventional approach these merely must be postulated blankly; 2) the gauge groups that now must appear in the present formalism are \textit{uniquely fixed} in advance by the differentiable manifold structure of Minkowskian space-time! Hence, in stark contrast to all other existing approaches to a unified physics beyond the standard model, post-Einsteinian theory has no adjustable parameters or any landscape problem whatsoever. One simply has to inspect what is implied by its general principles in order to see whether it agrees with nature, or not.
\end{remark}

\begin{remark}
At this juncture, we can subject the theory to a strong test of its consistency. In the case of electromagnetism, the coupling constant on the right hand side of equation (\ref{yang_mills_field_equ}) ought to equal $4 \pi$ in cgs units. Using the empirical value of the cosmological constant and the value for $B$ worked out in {\S}\ref{spont_symm_breaking} below, one indeed obtains the correct value to within a few percent. In view of the fact that we have not included any quantum-field-theoretical corrections, the agreement is surprisingly good. A priori there is no reason to expect any relationship between Newton's gravitational constant and the electroweak symmetry breaking scale. See {\S}\ref{exptl_confirmation} below.
\end{remark}

Before proceeding, we wish to clarify how coupling constants enter into the picture, since there arises below a subtle mathematical point that has major implications for the development of the theory. In elementary particle physics, one introduces coupling constants to describe how gauge forces interact with matter. If we denote our coupling constant by the letter $g$, one would have for the gauge covariant derivative and the field strength, respectively,
\begin{equation}
	\nabla = \text{\th} \otimes \mathrm{id} + g A \wedge = \text{\th} \otimes \mathrm{id} +\frac{1}{2} g ~\mathrm{ad}_{A \wedge}
\end{equation}
and
\begin{equation}
	F^\alpha_\beta = \text{\th} A^\alpha_\beta + \frac{1}{2} g A^\alpha_\gamma \wedge A^\gamma_\beta.
\end{equation}
Therefore, in terms of this conventional definition we would have $g=2$. Note that one can easily pass back and forth between different definitions of the coupling constant by a rescaling. The Yang-Mills action is ordinarily written in the form,
\begin{equation}
	\mathscr{L}_{YM} = \frac{1}{g^2} \int F \wedge * F,
\end{equation}
if the units are taken such that $F = \text{\th} A + A \wedge A = \text{\th} A + \dfrac{1}{2} ~\mathrm{ad}_A A$. Now, let us entertain a rescaling of the vector potential and the field strength of the form,
\begin{equation}\label{rescaled_gauge_field}
	\tilde{A} = \frac{1}{g} A \qquad \mathrm{resp.} \qquad \tilde{F} = \text{\th} \tilde{A} + g \tilde{A} \wedge \tilde{A}.
\end{equation}
It will readily be seen that
\begin{equation}\label{rescaled_YM_lagrangian}
	\mathscr{L}_{YM} = \frac{1}{g^2} \int F \wedge * F = \int \tilde{F} \wedge * \tilde{F}.
\end{equation}
Thus, we could get rid of the troubling factor of two by dividing the action by a factor of four. Keep in mind that the coupling constant also enters into the interaction term, which may be written as
\begin{align}
	\mathscr{L}_{\mathrm{int}} =& - \frac{1}{2} ~\mathrm{tr}~ \left( \frac{1}{2} (X_1 + X_2) \otimes \mathrm{id} + (X_1+X_2) \otimes I + \frac{1}{2} g \mathrm{ad}_A \right) \wedge \nonumber \\
	&* \frac{1}{4} \lambda^4 \left( m (X_1 + X_2) \otimes \mathrm{id} + (X_1+X_2) \otimes q I + \frac{1}{2} g \mathrm{ad}_{q A} \right).
\end{align}
There will be more to say on this topic in the next subsection when we separate out the electromagnetic and weak gauge fields.

\subsubsection{Construction of Electroweak forces from Jets at Second Order}\label{electroweak_construction}

So far, we have justified an field equation of Yang-Mills type for the \text{gravitational} gauge group generated by infinitesimal changes of frame, which as we have seen lie in the Lie algebra $\mathrm{so}(p,q)$. If our theory is to correspond to nature, however, we must explain the route by which these reduce to the observed fundamental forces represented by the standard model. The logical place to start would be with the electromagnetic and weak forces. The discussion above motivates why one might expect these to arise in the $(2,2)$ sector, i.e., at the first non-trivial order in the jets beyond the first (which have already been shown in Part II to correspond to Einstein's general theory of relativity for the gravitational force).

Let us henceforth denote by $A = \omega^{(2,2)}$ the jet affine connection 1-forms $\omega^\alpha_\beta$ restricted to $|\alpha|=|\beta|=2$. Now, the space of jets in 4d space-time is ten dimensional, with signature (3,7). In {\S}\ref{yang_mills_derivation} above, however, it is shown that the space of equivalence classes after identifying as physically equivalent gauge field configurations conjugate to one another under local Lorentz transform is four-dimensional and may be taken to be spanned by $d^{tt}-d^{xx}-d^{yy}-d^{zz}, d^{tx}, d^{ty}, d^{tz}$. Therefore, working for the remainder of this section in the space of equivalence classes, we have that the gauge group at second order in the jets reduces effectively to $\mathrm{so}(1,3)$, if we employ lowered indices ($\omega_{\alpha\beta}$), or to $\mathrm{so}(4)$ if we raise one of the indices ($\omega^\alpha_\beta$).

It proves helpful to look at an explicit realization of
\begin{equation}
	\mathrm{so}(4) = \{ X \in M_4(\vvmathbb{R}): X^t = -X \}.
\end{equation} 
Define six generators as follows:
\begin{align}
	A_1 &= \begin{pmatrix} 0 & 0 & 0 & 0 \cr 0 & 0 & -1 & 0 \cr 0 & 1 & 0 & 0 \cr 0 & 0 & 0 & 0 \cr \end{pmatrix} \qquad
	A_2 = \begin{pmatrix} 0 & 0 & 1 & 0 \cr 0 & 0 & 0 & 0 \cr -1 & 0 & 0 & 0 \cr 0 & 0 & 0 & 0 \cr \end{pmatrix} \qquad
	A_3 = \begin{pmatrix} 0 & -1 & 0 & 0 \cr 1 & 0 & 0 & 0 \cr 0 & 0 & 0 & 0 \cr 0 & 0 & 0 & 0 \cr \end{pmatrix} \qquad \nonumber \\
	B_1 &= \begin{pmatrix} 0 & 0 & 0 & -1 \cr 0 & 0 & 0 & 0 \cr 0 & 0 & 0 & 0 \cr 1 & 0 & 0 & 0 \cr \end{pmatrix} \qquad
	B_2 = \begin{pmatrix} 0 & 0 & 0 & 0 \cr 0 & 0 & 0 & -1 \cr 0 & 0 & 0 & 0 \cr 0 & 1 & 0 & 0 \cr \end{pmatrix} \qquad
	B_3 = \begin{pmatrix} 0 & 0 & 0 & 0 \cr 0 & 0 & 0 & 0 \cr 0 & 0 & 0 & -1 \cr 0 & 0 & 1 & 0 \cr \end{pmatrix}.
\end{align}
One checks the commutation relations
\begin{equation}
	[A_i,A_j] = \epsilon_{ijk} A_k, \qquad [B_i,B_j] = \epsilon_{ijk}A_k, \qquad [A_i,B_j] = \epsilon_{ijk}B_k,
\end{equation}
where $\epsilon_{ijk}$ denotes the totally anti-symmetric tensor. The purpose in selecting this representation of $\mathrm{so}(4)$ among all others is that in these terms one can readily display the isomorphism $\mathrm{so}(4) \cong \mathrm{so}(3) \oplus \mathrm{so}(3)$. Namely, define
\begin{equation}
	X_i = \frac{1}{2} (A_i+B_i), \qquad Y_i = \frac{1}{2} (A_i - B_i), \qquad i=1,2,3,
\end{equation}
and verify the commutation relations,
\begin{equation}\label{so4_commutation_relations}
	[X_i,X_j] = \epsilon_{ijk} X_k, \qquad [Y_i,Y_j] = \epsilon_{ijk}Y_k, \qquad [X_i,Y_j] = 0.
\end{equation}
Since the isomorphism of Lie algebras is also very well known that $\mathrm{so}(3) \cong \mathrm{su}(2)$, we are close to having achieved a representation of the electroweak force, which in the standard model transforms in the gauge algebra $\mathrm{u}(1)_Y \oplus \mathrm{su}(2)$. Here, the subscript $Y$ indicates that the unbroken electroweak gauge fields are spanned by what is called weak hypercharge (the $\mathrm{u}(1)_Y)$ factor) and weak isospin (the $\mathrm{su}(2)$ factor). All that is needed is a mechanism whereby one of the factors in $\mathrm{so}(4) \cong \mathrm{so}(3) \oplus \mathrm{so}(3)$ breaks to $\mathrm{u}(1)_Y$. This topic will occupy us in the following subsection.

\subsubsection{Spontaneous Symmetry Breaking from the Perspective of post-Einsteinian General Theory of Relativity}\label{spont_symm_breaking}

The derivation of the previous subsection has been purely formal. In this subsection, we complete our geometrodynamical theory of electroweak unification by investigating what happens in the Higgs sector. If the post-Einsteinian general theory of relativity as developed here is to describe the natural world, there must be \textit{some} mechanism by which to implement a breaking of the gauge symmetry. So far, the non-abelian symmetry remains unbroken, but can we rewrite the equations in such a way as to mimic the spontaneous symmetry breaking that occurs in the standard model of particle physics---in particular, to confer masses on the $Z^0$ and $W^\pm$ and to reduce the non-abelian gauge group from $\mathrm{u}(1)_Y \oplus \mathrm{su}(2)$ to the abelian group $\mathrm{u}(1)_\mathrm{em}$ (which is what survives in nature)? The signature of spontaneous symmetry breaking would be the existence of a non-zero linear contribution from one or more of the gauge fields, for in the Proca equation of motion a linear term appearing with coefficient $M^2$ corresponds to a vector-valued field with mass $M$. 

The problem seems---what is strange---to be related to the equivalence principle. As discussed above in {\S}\ref{status_of_equivalence_principle}, the stringent experimental limits on the net charge of atoms in molecular beams indicate that one wants bodies in gravitational free fall to behave in the same manner independent of whatever electrical charge they may carry. But we also found that electrically neutral bodies behave at second order in the jets as if they possessed an electrical charge equal to their rest mass times $\alpha^{1/2}$. The right perspective on the matter would seem to be to incorporate all rest mass terms into the gravitational interaction. In other words, one ought to measure the charge relative to the rest mass. The jet geodesic equation may be rewritten in the form
\begin{align}
	\nabla_X P &= \left( \nabla_X P \right)_1 + \Gamma X P \nonumber \\
	&= \left( \nabla_X P \right)_1 + \Gamma X (\alpha^{1/2} m + q) X \nonumber \\
	&= \left( \nabla_X P \right)_{12} + q \Gamma X X,
\end{align}
where the term indicated with the subscript 1 contains the ordinary gravitational contribution from 1-jets while the term indicated with subscript 12 contains in addition the gravitational contribution from 2-jets. The remaining terms are equivalent by proposition \ref{non_abelian_equ_of_motion} to the Lorentz force law for the forces originating in gauge fields. As we see on the second line above, what we have done in going to the third line is to regard a charged massive body as having the charge $q_1$ measured relative to $m$, or $q_0 = q_1 + \alpha^{1/2} m$, replacing the initial charge $q_0$. 

For then all massive bodies would behave the same in free fall in the absence of non-gravitational forces, while these latter would depend on the body’s respective charge-to-mass ratios. This would tend to lessen any tension in the precision experiments on the equality of the electrical charge of the proton versus that of the electron, because we may now measure electric charge with respect to the shifted origin. It was found above that there might be a contradiction with experiment if we want $|e_e| = |\alpha^{1/2}m_e-q_e| = |e_p| = |\alpha^{1/2}m_p+q_p|$, but now that we may disregard the rest mass, we could very well still set $|q_e| = |q_p|$ to obtain agreement with experiment.

To give the proposal of the previous paragraph a name, we may call the central rest mass hypothesis. Let us see what consequence the central rest mass hypothesis would have from the gauge-theoretic point of view. Introduce the following notation: denote the generators of one $\mathrm{so}(3)$ factor by $\hat{B}^0_{1,2,3}$ and of the other factor by $\hat{A}_{0,1,2}$ (we follow a standard convention in the literature in which the $\hat{B}$ is destined to become the weak hypercharge field while the $\hat{A}$ will correspond to weak isospin). In terms of the notation of the previous section, $B^0$ may be taken to be $X$ while $A$ may be taken to be $Y$. Thus, in explicit terms we have
\begin{equation}
	A = \frac{1}{2} \begin{pmatrix} 
		0 & -A_2 - B_2 & A_1 + B_1 & A_0 - B_0 \cr
		A_2 + B_2 & 0 & -A_0 - B_0 & A_1 - B_1 \cr
		-A_1 - B_1 & A_0 + B_0 & 0 & A_2 - B_2 \cr
		-A_0 + B_0 & -A_1 + B_1 & -A_2 + B_2 & 0 \cr
	\end{pmatrix}.
\end{equation}
Adopt the further convention of referring to the macroscopic weak hypercharge as $e$ and the macroscopic weak isospin as $f$. The Yang-Mills equation then reads as follows:
\begin{align}
	\text{\TH} * \text{\TH} A = - &2 \pi \varkappa \lambda^4 B^2 * J_1 \nonumber \\
	= &- 2 \pi \varkappa \lambda^4 B^2 *  \begin{pmatrix} 
		0 & (f+e)X_1 I_{12} & -(f+e)X_1 I_{13} & (e-f)X_1 I_{14} \cr
		-(f+e)X_1 I_{12} & 0 & -(f+e)X_1 I_{23} & (f-e)X_1 I_{24} \cr
		(f+e)X_1 I_{13} & (f+e)X_1 I_{23} & 0 & (f-e)X_1 I_{34} \cr
		(f-e)X_1 I_{14} & (e-f)X_1 I_{24} & (e-f)X_1 I_{34} & 0 \cr
	\end{pmatrix} \nonumber \\
	&- 2 \pi \varkappa \lambda^4 B^2 *  \begin{pmatrix} 
		0 & fA_2 + eB_2 & -fA_1 - eB_1 & -fA_0 + eB_0) \cr
		-fA_2 - eB_2 & 0 & -fA_0 - eB_0 & fA_1 - eB_1 \cr
		fA_1 + eB_1 & fA_0 + eB_0 & 0 & fA_2 - eB_2 \cr
		fA_0 - eB_0 & -fA_1 + eB_1 & -fA_2 + eB_2 & 0 \cr
	\end{pmatrix}.
\end{align}
In this context, the central rest mass hypothesis is tantamount to replacing $e$ with $e+\alpha^{1/2}m$ resp. $f$ with $q+\alpha^{1/2}m$ in the coefficients of $\hat{B}^0$ and $\hat{A}$ in the above equation. Transfer the rest-mass terms to the left hand side to obtain the non-abelian Proca Yang-Mills field equation in the following guise:
\begin{align}\label{proca_yang_mills}
	\text{\TH} * \text{\TH} A & +  2 \pi \varkappa \lambda^4 B^2 \alpha^{1/2} \varrho * 
	\begin{pmatrix} 
		0 & X_1 I_{12} & -X_1 I_{13} & 0 \cr
		-X_1 I_{12} & 0 & -X_1 I_{23} & 0 \cr
		X_1 I_{13} & X_1 I_{23} & 0 & 0 \cr
		0 & 0 & 0 & 0 \cr
	\end{pmatrix}
	\nonumber \\
	&+ 2 \pi \varkappa \lambda^4 B^2 \alpha^{1/2} \varrho *  
	\begin{pmatrix} 
		0 & A_2 + B_2 & -A_1 - B_1 & -A_0 + B_0) \cr
		-A_2 - B_2 & 0 & -A_0 - B_0 & A_1 - B_1 \cr
		A_1 + B_1 & A_0 + B_0 & 0 & A_2 - B_2 \cr
		A_0 - B_0 & -A_1 + B_1 & -A_2 + B_2 & 0 \cr
	\end{pmatrix}
	= - 2 \pi \varkappa \lambda^4 B^2 * J_1.
\end{align}
The reason for the nomenclature is that the terms linear in the gauge fields on the left hand side act to confer a non-zero mass on the vector gauge fields in the presence of other matter. Remember here that from equation (\ref{effective_momentum_current}), the term proportional to $\mathrm{ad}_A$ contains a factor of $\frac{1}{2}\lambda^2$ in the denominator which cancels against the cosmological density $\varrho_0$ in the numerator when we transfer the Proca term to the left hand side. We wish to recast equation (\ref{proca_yang_mills}) into a more perspicuous form. To this end, define a mass matrix
\begin{equation}
	\Delta \mathscr{L} = 4 \pi \varkappa \lambda^4 B^2 \alpha^{1/2} \varrho ~\mathrm{tr}~ \eta ~\mathrm{ad}_A ~\mathrm{ad}_A.
\end{equation}
Here,
\begin{equation}
	A = (A_0+X_1) \hat{A_0} + (A_1+X_1) \hat{A_1} + (A_2+X_1) \hat{A_2} + (B_0+X_1) \hat{B^0_0} + (B_1+X_1) \hat{B^0_1} + (B_2+X_1) \hat{B^0_2}.
\end{equation}
Then the Proca term in the Proca-Yang-Mills equation will be produced by variation of $\Delta \mathscr{L}$ with respect to $\mathbf{A}$ and $\mathbf{B}$. We can evaluate $\Delta \mathscr{L}$ directly to yield an expression in terms of its components,
\begin{align}\label{mass_matrix_reduced}
	\Delta \mathscr{L} &= \frac{1}{2} \pi \varkappa \lambda^4 B^2 \alpha^{1/2} \varrho \left(
	(A_0+B_0+ 2 X_1 I_{23})^2 + (A_1-B_1)^2 + (A_2-B_2)^2 \right) \nonumber \\
	&= \frac{1}{2} \pi \varkappa \lambda^4 B^2 \alpha^{1/2} \varrho \left(
	\mathbf{A}^2 + \mathbf{B}^2 + 2 A_0 B_0 - 2 A_1 B_1 - 2 A_2 B_2 + 4 (A_0+B_0)X_1 I_{23}+4 X_1^2 I_{23}^2 \right).
\end{align}
The term linear in $X_1$ can be eliminated by choice of gauge such that $\langle (A_0+B_0), X_1 \rangle = 0$, while the term quadratic in $X_1$ is entirely negligible in comparison with the contribution of the ordinary kinetic energy to the stress-energy, which goes as $4 \pi \varkappa \varrho X_1^2$, due to the smallness of the cosmological constant (as will be estimated below). From our definitions, however, $\mathbf{A}$ and $\mathbf{B}$ transform independently under changes of gauge generated by $\mathrm{so}(3) \oplus \mathrm{so}(3)$ (cf. equation (\ref{so4_commutation_relations}) above). The terms quadratic in $\mathbf{A}$ resp. $\mathbf{B}$ alone are invariant of course, but the inner product is not. Thus, the Proca term in the effective action breaks the gauge symmetry down to a subgroup of $\mathrm{so}(4)$. In order to identify what the residual gauge symmetry will be, observe that the inner product in equation (\ref{mass_matrix_reduced}) does remain invariant with respect to the diagonal subgroup. Therefore, it is expedient to redefine $B_0$ as the component of $\mathbf{B}$ parallel to $\mathbf{A}$. The two components of $\mathbf{B}$ in the perpendicular direction drop out of the problem, leaving us with a mass matrix of the form (with respect to the redefined variables),
\begin{equation}
	\Delta \mathscr{L} = \frac{1}{2} \pi \varkappa \lambda^4 B^2 \alpha^{1/2} \varrho \left( B_0^2 + 2 A_0 B_0 + A_0^2 + A_1^2 + A_2^2 \right).
\end{equation}

Note first that equation (\ref{proca_yang_mills}) differs from the Proca equation as ordinarily formulated in having a factor of density attached to the mass terms on the left hand side. In order to motivate its correctness, we engage in a thought-experiment: imagine a world in which bodies never collide; we could posit a unit mass $\delta m$ and regard every body as a collection of unit masses traveling together in a bound state. Then one could replace the Newtonian coupling constant $\varkappa$ with $\varkappa \delta m$ on the right hand side and never thereafter have to use any concept of mass, just the number of unit masses in a body (divide the jet geodesic equation through by mass so it really involves only charge-to-mass ratios). If we restore a concept of variable mass, we may sum over unit mass delta functions to produce a continuum limit on both sides yielding the term $* \varrho A$ on the left hand side, in other words, the modified form of Proca equation as it appears in equation (\ref{proca_yang_mills}). In elementary particle physics itself one never deals with macroscopic bound states and can avoid this subtlety which is why the usual Proca equation has no factor of $\varrho$ in it.

All that remains is to exhibit a change of basis in gauge field space that brings out the spontaneous symmetry breaking. At this point, the subtlety adverted to in the previous subsection becomes pertinent. Let us adopt the standard basis of $\mathrm{su}(2)$, viz.,
\begin{equation}
	E_1 = \frac{1}{2} \begin{pmatrix} i & 0 \cr 0 & -i \cr \end{pmatrix}; \qquad
	E_2 = \frac{1}{2} \begin{pmatrix} 0 & i \cr i & 0 \cr \end{pmatrix}; \qquad
	E_3 = \frac{1}{2} \begin{pmatrix} 0 & -1 \cr 1 & 0 \cr \end{pmatrix},
\end{equation}
satisfying the commutation relations $[E_1,E_2]=E_3$, $[E_2,E_3]=E_1$ and $[E_3,E_1]=E_2$ (\cite{hall}, Example 3.27). In terms of the adjoint representation of $\mathrm{su}(2)$ we have been using, the standard basis should be represented by
$E_{1,2,3} = A_{0,1,2}$ so as to satisfy the same commutation relations. But the Lie algebra $\mathrm{su}(2)$ by definition comprises anti-hermitian traceless matrices $X \in M_2(\vvmathbb{C})$ such that $X^* = - X$ and $\mathrm{tr}~ X=0$. Let us parametrize its three real dimensions via $a, b, c \in \vvmathbb{R}$, where
\begin{equation}
	X = \begin{pmatrix} ia & b+ic \cr -b + ic & -ia \cr \end{pmatrix}.
\end{equation}
The Lie algebra spanned by the $E_{1,2,3}$ resp. $A_{0,1,2}$ is isomorphic to three-dimensional Euclidean space equipped with the vector cross product, denoted $(\vvmathbb{R}^3,\times)$. It may easily be checked that the linear map $\phi: \mathrm{su}(2) \rightarrow (\vvmathbb{R}^3,\times)$ defined by $\phi(X) = (2a,2b,2c)$ establishes an isomorphism of Lie algebras over $\vvmathbb{R}$; that is, $\phi([X,Y])=\phi(X)\times\phi(Y)$ (cf. Hall, \cite{hall}, Problem 3.10). The salient point is that, to go in the other direction, one has to multiply the gauge field components by a factor of $\frac{1}{2}$. Hence, for the weak sector only we change the scale with which the gauge potential and field strength are measured, from $A_{0,1,2}$ to $\tilde{A}_{0,1,2} = \frac{1}{2} A_{0,1,2}$ or $A_{0,1,2} = 2 \tilde{A}_{0,1,2}$. In light of equations (\ref{rescaled_gauge_field}) and (\ref{rescaled_YM_lagrangian}), this is tantamount to a doubling of the weak coupling constant while leaving the weak hypercharge coupling constant unchanged. 

The development thus far permits an extraordinary statement: in post-Einsteinian theory, Weinberg's weak mixing angle is fixed to a definite value. That is, in units where the weak hypercharge coupling $g^\prime=2$, we have necessarily $g=4$. Weinberg's weak mixing angle is defined by
\begin{equation}
	\cos \theta_W = \frac{g}{\sqrt{g^2+g^{\prime2}}}; \qquad \sin \theta_W = \frac{g^\prime}{\sqrt{g^2+g^{\prime2}}}; \qquad \tan \theta_W = \frac{g^\prime}{g}.
\end{equation}
Therefore, the predicted value for $\theta_W$ in the classical limit as wavenumber tends to zero can be extracted from $\tan \theta_W = \frac{1}{2}$ to be $\theta_W = \sin^{-1} \dfrac{1}{\sqrt{5}} = 26.565^\circ$, or equivalently $\sin^2 \theta_W = \frac{1}{5}$. Phenomenologically, this represents what would be observed if it were possible to carry out experiments with macroscopically large bodies bearing appropriate weak hypercharges and weak isospins. This being impossible, we must have recourse to quantum field theory, from the standard model of which $\theta_W$ can be extracted as the ratio of $Z^0$ to $W^\pm$ gauge boson masses: $m_Z=m_W/\cos \theta_W$ \cite{glashow_weinberg}. Precision measurements place the weak mixing angle at $\theta_W = 28.1780^\circ$ or $\sin^2 \theta_W = 0.23120 \pm 0.00015$ at a momentum transfer equal to the mass of the Z boson. The running of the weak mixing angle, however, is a rather complicated affair and the experimental data are sparse (consult in this connection the thorough review by Kumar et al., \cite{kumar_et_al}). Suffice it to say that the near agreement between the theoretical value of $26.565^\circ$ and the measured value of $28.1780^\circ$ is encouraging (however the relation between the two is to be interpreted).

Rewrite $\Delta \mathscr{L}$ as
\begin{align}\label{mass_matrix}
	\Delta \mathscr{L} &= \frac{1}{2}
	\pi \varkappa \lambda^4 B^2 \alpha^{1/2} \varrho \bigg[ g^{\prime 2} (B^0)^2 + 2 g^\prime g B^0 A^0 + g^2 (A^0)^2 + 
	g^2 (A^1)^2 + g^2 (A^2)^2 \bigg] \nonumber \\
	&=
	2 \pi \varkappa \lambda^4 B^2 \alpha^{1/2} \varrho \bigg[ g^{\prime 2} (B^0)^2 + 2 g^\prime g B^0 A^0 + g^2 (A^0)^2 + 
	g^2 (A^1)^2 + g^2 (A^2)^2 \bigg],
\end{align}
where in the last line we set $g^\prime=1$, $g=2$ instead of $g^\prime=2$, $g=4$. The mass matrix becomes (dropping the prefactor which must be nearly constant on the scale over which the gauge fields vary)
\begin{equation}
	\frac{1}{2} \begin{pmatrix}
		g^\prime B^0 & g A^0 & g A^1 & g A^2 \cr \end{pmatrix}
	\begin{pmatrix}
		1 & 1 & 0 & 0 \cr
		1 & 1 & 0 & 0 \cr
		0 & 0 & 1 & 0 \cr
		0 & 0 & 0 & 1 \cr
	\end{pmatrix}
	\begin{pmatrix}
		g^\prime B^0 \cr g A^0 \cr g A^1 \cr g A^2 \cr \end{pmatrix}.
\end{equation}
This may equivalently be put into the form,
\begin{equation}
	\frac{1}{2} \begin{pmatrix}
		B^0 & A^0 & A^1 & A^2 \cr \end{pmatrix}
	\begin{pmatrix}
		g^{\prime 2} & g g^\prime & 0 & 0 \cr
		g g^\prime & g^2 & 0 & 0 \cr
		0 & 0 & g^2 & 0 \cr
		0 & 0 & 0 & g^2 \cr
	\end{pmatrix}
	\begin{pmatrix}
		B^0 \cr A^0 \cr A^1 \cr A^2 \cr \end{pmatrix}.
\end{equation}
The symmetric matrix $\begin{pmatrix} g^{\prime 2} & g g^\prime \cr g g^\prime & g^2 \cr \end{pmatrix}$ has eigenvalues zero and $g^2 + g^{\prime 2}$. Therefore, it can be diagonalized to yield a zero mode with a weak mixing transformation, which can be written as
\begin{equation}\label{weak_mixing}
	\begin{pmatrix} \gamma^0 \cr Z^0 \cr A^1 \cr A^2 \cr \end{pmatrix} = \begin{pmatrix}
		\cos \theta_W & \sin \theta_W & 0 & 0 \cr - \sin \theta_W & \cos \theta_W & 0 & 0 \cr 0 & 0 & 1 & 0 \cr 0 & 0 & 0 & 1 \end{pmatrix}
	\begin{pmatrix} B^0 \cr A^0 \cr A^1 \cr A^2 \cr \end{pmatrix},
\end{equation}
where $\theta_W$ is the Weinberg angle. Elementary computation yields for the mass matrix,
\begin{align}\label{intermediate_mass_matrix}
	&\frac{1}{2} \begin{pmatrix}
		\gamma^0 \cr Z^0 \cr A^1 \cr A^2 \cr \end{pmatrix}^t \times \nonumber \\ &\begin{pmatrix}
		(g^\prime \cos \theta_W - g \sin \theta_W)^2 & 
		g g^\prime \cos 2 \theta_W  + (g^{\prime 2} - g^2) \cos \theta_W \sin \theta_W & 0 & 0 \cr
		g g^\prime \cos 2 \theta_W  + (g^{\prime 2} - g^2) \cos \theta_W \sin \theta_W  &
		(g \cos \theta_W + g^\prime \sin \theta_W)^2 & 0 & 0 \cr
		0 & 0 &  g^2 & 0 \cr
		0 & 0 & 0 & g^2 \cr
	\end{pmatrix} \times \nonumber \\
	&\begin{pmatrix}
		\gamma^0 \cr Z^0 \cr A^1 \cr A^2 \cr \end{pmatrix}.
\end{align}
The value of the Weinberg angle that diagonalizes the mass matrix is, as above, $\theta_W = \sin^{-1} \frac{1}{\sqrt{5}}$. After substitution we get
\begin{equation}
	\frac{1}{2} \begin{pmatrix}
		\gamma^0 & Z^0 & A^1 & A^2 \cr \end{pmatrix}
	\begin{pmatrix}
		\frac{1}{5} (g - 2 g^\prime)^2 & -\frac{1}{5} (g-2g^\prime)(2g+g^\prime) & 0 & 0 \cr
		-\frac{1}{5} (g-2g^\prime)(2g+g^\prime) & \frac{1}{5} (2g + g^\prime)^2 & 0 & 0 \cr
		0 & 0 & g^2 & 0 \cr
		0 & 0 & 0 & g^2 \cr
	\end{pmatrix}
	\begin{pmatrix}
		\gamma^0 \cr Z^0 \cr A^1 \cr A^2 \cr \end{pmatrix}.
\end{equation}
Lastly, with the values of $g^\prime=1$, $g=2$ required by the theory, this reduces to
\begin{equation}
	\frac{1}{2} \begin{pmatrix}
		\gamma^0 & Z^0 & A^1 & A^2 \cr \end{pmatrix}
	\begin{pmatrix}
		0 & 0 & 0 & 0 \cr
		0 & 5 & 0 & 0 \cr
		0 & 0 & 4 & 0 \cr
		0 & 0 & 0 & 4 \cr
	\end{pmatrix}
	\begin{pmatrix}
		\gamma^0 \cr Z^0 \cr A^1 \cr A^2 \cr \end{pmatrix}.
\end{equation}
Hence, we end up with a massless photon $\gamma^0$ and massive $Z^0, W^+, W^-$ bosons obeying the mass ratio $M_{Z^0}/M_{W^\pm} = \sec \theta_W = \frac{\sqrt{5}}{2}$.
Thus, we have exhibited how spontaneous symmetry breaking occurs in post-Einsteinian theory and determined the theoretical value of the \textit{relative} masses of the $Z^0$ and $W^\pm$ bosons. As for their \textit{absolute} masses, equation (\ref{proca_yang_mills}) suggests that they scale with the local density of matter and would vanish in vacuum. But the relevant density of matter should include contributions from all sources, presumably including the cosmological constant which does not vary across all of space and would confer masses on the $Z^0$ and $W^\pm$ in vacuo. In the remainder of this paragraph, we undertake a preliminary discussion of how the absolute value of the masses of weak gauge bosons must come about, given the observed value of the cosmological constant 
\begin{equation}
	\Lambda = 1.1056 \times 10^{-56} ~\mathrm{cm}^{-2},
\end{equation} 
as determined by the Planck collaboration in 2018 \cite{planck_collaboration_2018}. In the 1-jet sector, as usual, the cosmological constant $\Lambda$ may be identified with a vacuum energy density $\varrho_1 c^2$, where the two are related by
\begin{equation}
	\Lambda = \frac{8 \pi G}{c^2} \varrho_1 = \frac{1}{2} \lambda^2.
\end{equation}
But $\varrho = \varrho_1/\frac{1}{2}\lambda^2$ so we have $8 \pi \varkappa \varrho c^2 = 1$. Comparing now with equation (\ref{proca_yang_mills}), we see that the dark energy term arising on the left hand side of the Proca-Yang-Mills equation will have coefficient $2 \pi \varkappa \lambda^4 B^2 \alpha^{1/2} \varrho$. In order for this to give rise to the mass terms for the $Z^0$ resp. $W^\pm$ bosons, in view of the the fact that $g^\prime = 2$ and $g=4$ we must equate (inserting the dimensional factors)
\begin{equation}
	\frac{1}{4} g^2 \lambda^4 B^2 \alpha^{1/2} = 4 \lambda^4 B^2 \ell_P^2 \alpha^{1/2} = \frac{M_{W^\pm}^2}{m_P^2},
\end{equation}
from which we evaluate $B = 3.1514 \times 10^{71} ~\mathrm{cm} = 1.9498 \times 10^{104}$ in geometrical units. Now, the phenomenology of the Higgs sector as it appears in post-Einsteinian general theory of relativity remains to be worked out. Ideally one would want an heuristic understanding of why the conjectured form of the dark energy term and its apparent relation to the Higgs scalar field hold, with the given numerical coefficients.

To conclude, let us remark on the status of the central rest mass hypothesis. This conjecture appears at once to be natural from an a priori point of view, as the logical extension of Einstein's equivalence principle, and to receive support, directly, from our derivation of Anderson's prediction formula for the flyby anomaly and indirectly, from the stringent experimental limits on the difference in magnitude between the elementary charges of the proton and the electron. The startling implication of this situation may perhaps most judiciously be circumscribed by saying that spontaneous symmetry breaking does not actually occur in nature, as many theoretical physicists have supposed for half a century, but represents merely an artifact of our parochial decision to describe the infinitesimal degrees of freedom in terms of conventional non-abelian gauge fields rather than through Cartan's vielbein, which is more intrinsically rooted in the Minkowskian structure of the space-time manifold in the infinitely small.

\subsection{Experimental Confirmation of the Relation among Coupling Constants}\label{exptl_confirmation}

The post-Einsteinian theory is subject to a very strong consistency condition. Namely, the coupling constant on the right-hand side of the Yang-Mills field equation (\ref{yang_mills_field_equ}) derived above must agree with the empirical value. Stated in SI units, the field equation reads
\begin{equation}\label{maxwell_in_SI}
	\text{\TH} * F = - \frac{1}{\varepsilon_0} * J_1,
\end{equation}
where $\varepsilon_0$ is the permittivity of free space and is given by
\begin{equation}\label{permittivity_def}
	\frac{1}{\varepsilon_0} = \frac{4 \pi \alpha \hslash c}{e^2} = 4 \pi = 12.5664 ~\mathrm{[cgs]},
\end{equation}
It is convenient to employ cgs units so that the right hand side becomes a pure number. Note that the factor $e^2/\alpha \hslash c = 1$ in cgs but $e_\mathrm{SI}^2/e_\mathrm{cgs}^2$ in SI. Now, at the end of {\S}\ref{spont_symm_breaking}, we determine the parameter $B$ so that $4 \lambda^4 B^2 \ell_P^2 = M_W^2/m_P^2 \alpha^{1/2}$. Therefore, for the sake of consistency we should have
\begin{equation}\label{em_consistency_condition}
	\frac{1}{\varepsilon_0} = 2 \pi \varkappa \lambda^4 B^2 = \frac{\pi \varkappa M_W^2}{2 m_P^2} = \frac{2 \pi \varkappa M_Z^2}{5 m_P^2}.
\end{equation}
A conversion factor is needed to go back and forth between geometrical units and any other system of units. In geometrical units, one works with dimensionless measures obtained by dividing the dimensional measure through by an elementary constant having the same dimension. Thus, for charge, $q_0 = q/e$ and for the potential $A_0 = (e/m_p c^2) A = (G e/c^4 \ell_P)A$. Therefore, on the right hand side of equation (\ref{maxwell_in_SI}) we replace
\begin{align}
	2 \pi \lambda^4 B^2 q_0 A_0 &= 2 \pi \frac{M_W^2}{4 m_P^2} \frac{c^2}{G^{1/2} \alpha^{1/2}} \frac{G}{c^4 \ell_P} q A \nonumber \\
	&= \frac{\pi M_W^2}{2 m_P^2} \frac{\alpha^{1/2} \ell_P^2 c^2}{G^{1/2} e^2 \ell_P} q A \nonumber \\
	&= \frac{\pi M_W^2}{2 m_P^2} \frac{e}{e^2} q A [\mathrm{cgs}] \nonumber \\
	&= \frac{\pi M_W^2}{2 m_P^2} \frac{10^7 \alpha \hslash c}{e_\mathrm{SI}^2} \left( \frac{10^{-7} qA}{e_\mathrm{SI}} \right) [\mathrm{SI}] \nonumber \\
	&= \left( \frac{\pi M_W^2}{2 m_P^2} \right) \left( \frac{10^7 \alpha \hslash c}{e_\mathrm{SI}^2} \right) \left( \frac{q A [\mathrm{cgs}]}{e_\mathrm{SI} (e_\mathrm{cgs}/e_\mathrm{SI})} \right) \nonumber \\
	&= \left( \frac{\pi M_W^2}{2 m_P^2} \right)  \iota q A [\mathrm{cgs}] \nonumber \\
	&= 4 \pi q A [\mathrm{cgs}]. 
\end{align}
The conversion factor would therefore be\footnote{This number itself has no fundamental significance since SI is not a natural system of units. Note that the current $J_1$ has units of $[\mathrm{charge}]\times [\mathrm{potential}]$ measured in coulombs and joules per coulomb respectively. In geometrical units, however, the gravitational coupling constant $\varkappa = G/c^4 = 1$ and charge is to be measured in multiples of the electron charge. In geometrical units, the conversion factor corresponds to the quantity $\iota = 8 \pi \chi^2/\alpha^{1/2} \ell_P \varkappa$. If formulae like equations (\ref{first_W_Z_relation}), (\ref{second_W_Z_relation}) and (\ref{third_W_Z_relation}) appear strange, remember that in equation (\ref{em_consistency_empirical}) the conversion factor $\iota$ is implictly multiplied by a coupling constant having unit magnitude in cgs whose units are such as to cause all dimensional factors to cancel, leaving the pure number $4 \pi$.}
\begin{align}
	\iota = &
	\left( \frac{4.80312 \times 10^{-10} ~\mathrm{esu}}{1.602176634 \times 10^{-19} ~\mathrm{coulomb}} \right)^2 \left( \frac{1}{1.602176634 \times 10^{-19} ~\mathrm{coulomb}} \right) \times \nonumber \\
	& \left( \frac{10^7 ~\mathrm{erg/joule}}{4.80312 \times 10^{-10} \mathrm{esu}/1.602176634 \times 10^{-19} ~\mathrm{coulomb}} \right).
\end{align}
The empirical values of the right hand side of equation (\ref{em_consistency_condition}) are as follows (using the values $M_W = 80.377 ~\mathrm{GeV}/c^2$ and $M_Z= 91.1876 ~\mathrm{GeV}/c^2$ reported by the Particle Data Group, \cite{particle_data_group}):
\begin{align}\label{em_consistency_empirical}
	\frac{\pi M_W^2}{2 m_P^2} \iota &=
	(6.80851 \times 10^{-35}) (1.87112 \times 10^{35}) = 12.7395 ~\mathrm{[cgs]} \nonumber \\
	\frac{2 \pi M_Z^2}{5 m_P^2} \iota &= (7.01017 \times 10^{-35}) (1.87112 \times 10^{35}) = 13.1169 ~\mathrm{[cgs]}.
\end{align}
In view of the fact that we have not included any quantum-field-theoretical corrections, the agreement of equation (\ref{em_consistency_empirical}) with equation (\ref{permittivity_def}) is surprisingly good. The departure of the right hand side from $4 \pi$ is systematic and turns out to be close to 
\begin{align}\label{first_W_Z_relation}
	\frac{\pi M_W^2}{2 m_P^2} \iota &= 4 \pi \chi \nonumber \\
	\frac{2 \pi M_Z^2}{5 m_P^2} \iota &= 4 \pi \chi^3,
\end{align}
where
\begin{equation}
	\chi = \frac{\cos \theta_{W,\mathrm{theory}}}{\cos \theta_{W,\mathrm{expt}}} = 
	\frac{2 M_Z}{\sqrt{5}M_W} = \frac{2/\sqrt{5}}{\cos 28.1780^\circ} = 1.014701 =1 + 2 \alpha.
\end{equation}
But the parameters $M_Z$, $M_W$ and $\theta_W$ are not independent of one another. Thus, it makes sense to remove the derived parameter $\theta_W$ from the problem and to express the result in terms of the measured quantities $M_Z$ and $M_W$ alone. The above may then with a little algebra be solved to yield another independent relation, namely,
\begin{equation}\label{second_W_Z_relation}
	\frac{M_W^3}{M_Z m_P^2} \iota = \frac{16}{\sqrt{5}}, \end{equation}
which holds good to a relative error of 0.99911. Another way of saying the same thing would be to solve equations (\ref{first_W_Z_relation}) and (\ref{second_W_Z_relation}) for $M_W$ and $M_Z$,
\begin{align}\label{third_W_Z_relation}
	M_W &= 2 \sqrt{2} (1 + \alpha) m_P/\iota^{1/2} \nonumber \\
	M_Z &= \sqrt{10} (1 + 3 \alpha) m_P/\iota^{1/2}.
\end{align}
These two semi-empirical prediction formulae are valid to the experimental precision of five significant figures. At the present stage of development of the theory, such numerical coincidences are by no means understood but they are highly suggestive of an as-yet-unknown underlying law that governs the process of spontaneous symmetry breaking.

	\section{Unification of Gravity with the Strong Nuclear Force}\label{chapter_9}

\def\su{\mathrm{su}}
\def\u{\mathrm{u}}
\def\so{\mathrm{so}}
\def\tr{\mathrm{tr}}
\def\spn{\mathrm{span}}
\def\imo{\mathrm{Im}~\vvmathbb{O}}
\def\der{\mathrm{der}}
\def\g{\mathrm{g}}
\def\ad{\mathrm{\ad}}

If the theory of fundamental forces of nature advanced in {\S}\ref{chapter_8} be cogent, it becomes incumbent on us to show how the strong nuclear force may be incorporated into the picture. Just as electroweak forces arise from 2-jets in the extended Riemannian metric, it is plausible to conjecture that including 3-jets ought likewise to yield another fundamental force of nature, to be identified with the strong nuclear force. To demonstrate the reality of this surmise will call upon all the apparatus previously developed. We begin by examining the 3-jet sector and show how higher tangents at third order may be rewritten as generators of gauge field degrees of freedom, just as has been done at second order in {\S}\ref{chapter_8}. Then, our postulated form of the equivalence principle is extended to third order and the pattern of spontaneous symmetry breaking it implies worked out in detail. The gauge degrees of freedom at third order in a space-time of four dimensions decompose into massive and massless modes. Those among the latter that survive symmetry breaking may be identified with generators of the real Lie algebra $\mathrm{su}(3)$, just as the massless mode among the second-order jets may be equated with the photon. 

\subsection{A Theoretical Framework for Describing the Dynamical Process of Nature}\label{dynam_process}

By the dynamical process of nature \cite{schelling} we understand a logical connection of ideas which are taken to represent the world of experience in theoretical terms (that is to say, the unfolding of a logical implication and not a temporal happening). The guiding principle, then, will be a unitary concept of force as the cause of change in the material world, which becomes differentiated due to the simultaneous action of the principle of polarity. In particular, an elaboration of the dynamical process will make apparent the nature of the fundamental forces and their relation to the inner structure of space and time, as is to be deduced from the elementary concepts of an infinitesimal geometry of higher order once applied to the context of relativistic physics.

It will be convenient to gather into one place a statement of the principles that have emerged over the course of our investigations thus far:

1) One should seek natural constructions depending on inherent properties of the mathematical objects to be considered rather than on gratuitous and artificial mechanistic hypotheses; and in so far as possible, maintain connection with experiment as motivation for the formalism to be introduced. This statement is meant not as idle, but as of central methodological importance to our undertaking and differentiates our approach from all others in currently being pursued in fundamental physics, for we seek not merely the \textit{hypothetical} necessity that attaches to a contingent mechanical model, however devised, but the \textit{apodictic} certainty proper to a natural science, as such. Such certainty is to be secured via the method of construction in which one first posits the general physical principles then, second, the mathematical formalism by which to realize them.

2) Kinematics as the science of change is to be relied upon as a fruitful source of ideas about how to extend physics beyond what is known. As we have seen above in Part I, kinematics appears in an altogether different light once one contemplates the potential role of infinitesimals of higher than first order. It must be supposed that we stand only at the beginning of an exploration of the implications of, for instance, the jet geodesic equation and the concept of a diffeomorphism with inertia, as described above in {\S}I.3. Infinitesimals are inherently rooted in the differentiable-manifold structure of space-time; hence, it is very natural to entertain the possibility that they may have non-trivial consequences for physics. Therefore, rather than to truncate at first order as has been done in all of theoretical physics from Newton to Einstein, we wish, in principle, to include infinitesimals to all higher orders (in practice, this will mean a series of approximations as one goes up order by order). The 
advantages of this approach are manifest: it eventuates in an almost unique specification of the sequence of models one ought to employ as one increases the order of infinitesimals retained. Thus, it respects Ockham's razor: for we shall have no reason to multiply entities beyond necessity (as certainly happens with supersymmetric models) and, what is just as important, have no cause to be diverted into the investigation of toy models, in as much as the phenomenology of the real degrees of freedom supposed to be present turns out to be simple enough to admit of analysis in its own right.

3) Third, we posit general covariance, viz., that all entities contemplated in the theory ought to correspond to intrinsically defined geometrical invariants such as the extended Riemannian curvature tensor. From a \textit{physical} point of view, the principle of general covariance is not entirely vacuous as some authors would contend, who maintain that it serves merely as a spur to the theorist's ingenuity to devise a covariant formulation for whatever theory he pleases (cf. Norton \cite{norton_physical_content}), for the following reason. Granted, while it is true that pretty much any theory (including Newtonian mechanics and classical electrodynamics) found out by a process of induction from experience not explicitly appealing to relativistic considerations may indeed be recast in covariant terms, when one does so it calls for the artificial introduction of additional structures with which one has to equip the manifold of space-time. In contrast, when one \textit{starts} from general covariance as Einstein understands it in his general theory of relativity, one has already a natural set of geometrical invariants with which to work, namely, the Riemannian metric tensor and objects derived from it, such as the Levi-Civita connection and the Riemannian curvature tensor. As we shall demonstrate in what follows, this situation continues to hold in the setting of higher-order infinitesimals.

4) The pseudo-Riemmanian signature of space-time is to be extended to higher orders in the jets in the most natural conceivable manner, as described above in {\S}II.2. Here, the principle of plenitude as introduced above in {\S}I.5 on purely mathematical grounds begins to have non-trivial implications for physics.

5) We propose the identification of higher jets as gauge indices, for which a physical interpretation in spatio-temporal terms exists, namely, that matter is to be viewed as a flow of infinitesimal rotations of space-time (i.e., a diffeomorphism with inertia in the terminology of {\S}I.3). Thus, the connection of post-Einsteinian theory with known physics is to be found out by means of an appropriate reinterpretation of the basic entities of the Cartan formalism and their translation into gauge-theoretical terms with which physicists are currently familiar, as the theoretical ground of the standard model as conventionally described.

6) We propose to view dark energy as a residual background curvature of space-time, i.e., as not just a scalar constant but as a curvature 2-form, as described in {\S}II.2.

7) The transition from pure mathematics to physics properly speaking involves the introduction of a concept of \textit{quantity of matter}. Thus, kinematics has to be replaced by dynamics. Wherefore, not only does kinematics have to be fundamentally reconceived from the perspective of higher-order geometry, so also with dynamics. This implies a central role of higher inertia, as suggested above in {\S}{\S}I.3 and \ref{unification_of_gr_and_em}, the right understanding of which must be reckoned as being far from complete. For higher inertia seems to involve a curious intersection between the properties of jets as part of a Minkowskian space-time and their properties when represented in terms of gauge field degrees of freedom.

8) The central rest mass hypothesis, as sketched above in {\S}\ref{spont_symm_breaking}, is to be viewed as the proper generalization of the equivalence principle. For we may describe a body in motion either from the spatio-temporal point of view or from the gauge-theoretical perspective, and the central rest mass hypothesis governs their interrelationship in such a way as to ensure that all moving bodies exhibit the same behavior, viz., inertial motion in free fall, where now one must apply the notion of higher inertia to incorporate the higher infinitesimals in a consistent manner. 

9) The symmetry breaking mechanism is to be accomplished not via a Higgs field whose potential term remains entirely unknown in the conventional approach to electroweak theory and its extensions to grand unified theories, but via the interaction of gauge fields with the cosmological term, for which the central rest mass hypothesis supplies a unique proposed Proca term in a completely natural and unforced way as the Killing form inherited from the Minkowskian metric on spaces of higher jets. Therefore, one is duty-bound to investigate its consequences before turning to other, essentially more complicated and unmotivated mechanistic models.

10) The principles of stepwise covariance, order-by-order ascent and microscopic indiscernability, as expounded below in {\S}\ref{selection_effective_gauge_algebra}, supply the physical interpretation behind the mathematical procedure by which the form of the fundamental forces realized in nature at energies far below the Planck scale is to be determined. For the infinitesimals are not merely kinematical but endowed with a genuinely dynamical status, which constrains the conditions under which they may manifest themselves in everyday physics.

\subsection{Unification of Gravity and Chromodynamics}\label{unification_of_gr_and_cd}

After a quick review of our starting point in {\S}\ref{preliminaries}, we describe in {\S}{\S}\ref{structure_of_gauge_alg}, \ref{reduction_to_isotropic}, \ref{identification_of_chromodynamics} how the general physical principles of post-Einsteinian general theory of relativity lead naturally to a unification of all four fundamental forces appearing in the standard model, when the infinitesimal degrees of freedom at higher order receive an interpretation as equivalent in the weak-field limit to non-abelian gauge fields.

\subsubsection{Preliminary Considerations}\label{preliminaries}

We have already seen in {\S}\ref{chapter_8} that when one adopts a physical interpretation of matter as a flow of infinitesimal deformations to the spatial curvature 2-forms and posits an appropriate interaction with the jet affine connection 1-forms in the Einstein-Hilbert action, the variational formulation yields a field equation of non-abelian Yang-Mills type. There, it is further shown that at the 2-jet level the implied gauge fields can be described by electroweak theory. Here, we wish to investigate what happens when one undertakes the obvious extension to 3-jets. To start with, recall {\S}\ref{unification_of_gr_and_em}. It is evident that the jet geodesic equation for the motion of a body must contain a series of contributions coming from higher tangents in its velocity field $X$, which to leading order generates
\begin{equation}
	\frac{d X}{d \lambda} = - \Gamma X^0 X^0 - 2 \Gamma X^0 X^{00} - 2 \Gamma X^0 X^{000} + \cdots, 
\end{equation}
where now we truncate the expansion at \textit{third} order. Proposition \ref{non_abelian_equ_of_motion} shows that the terms at second order and beyond on the right hand side may in fact be interpreted as non-abelian Yang-Mills forces. Given our success in explaining electroweak theory in {\S}\ref{chapter_8}, it is natural to expect that the term at third order, whatever it may be, ought to reduce to the strong nuclear force as described by classical chromodynamics, viz., a non-abelian Yang-Mills gauge theory whose effective degrees of freedom display an $SU(3)$ symmetry.

According to the general principles laid out in {\S}\ref{dynam_process}, the procedure we must follow is clear. We have to analyze the Proca term implied by the central rest mass hypothesis in order to see what initial 3-jet gauge degrees of freedom survive as a true gauge symmetry of the effective theory. The structure of the jet spaces becomes ever more complicated as one goes to higher and higher order, of course, so that our task will become that much the more involved. But, fortunately, the physical principles we already have at hand suffice to conduct us to the final result.

\subsubsection{Structure of the Gauge Lie Algebra at Third Order in the Jets}\label{structure_of_gauge_alg}

It will be helpful to begin by demonstrating the symmetry of the Yang-Mills field equation derived in {\S}\ref{yang_mills_derivation} with respect to the full gauge group. Let us denote a local change of frame by $\Lambda: M \rightarrow SO(p+q)$, where $p$ denotes the number of timelike directions and $q$ the number of spacelike directions in the generalized Riemannian metric tensor $g$ defined on jets up to a given order in the space-time $M$. This change of frame may be considered as arising by exponentiation of a gauge field $A: M \rightarrow \mathrm{so}(p+q)$.

\begin{proposition}
	The Yang-Mills equation is symmetric with respect to the change of gauge $\Lambda$.
\end{proposition}
\begin{proof}
	In general, Lie-algebra valued currents transform in the adjoint representation:
	$A \rightarrow \Lambda A \Lambda^{-1}$.
	This action extends to $k$-forms via $Y^{(1)} \wedge \cdots \wedge Y^{(k)} \rightarrow \Lambda Y^{(1)} \Lambda^{-1} \wedge \cdots \wedge \Lambda Y^{(k)} \Lambda^{-1} = \Lambda Y^{(1)} \wedge \cdots \wedge Y^{(k)} \Lambda^{-1}$. Therefore, the covariant derivative transforms as
	\begin{align}
		\text{\TH} Y = \left( \text{\th} \otimes \mathrm{id} + A \wedge \right) Y \rightarrow 
		& \left( \text{\th} \otimes \mathrm{id} + \Lambda A \Lambda^{-1} \wedge \right) \Lambda Y \Lambda^{-1} \nonumber \\
		&= \Lambda ( \text{\th} Y) \Lambda^{-1} + \Lambda A \Lambda^{-1} \wedge \Lambda Y \Lambda^{-1} \nonumber \\
		&= \Lambda \left( \text{\th} Y + A \wedge Y \right) \Lambda^{-1} \nonumber \\
		&= \Lambda \text{\TH} Y \Lambda^{-1}.
	\end{align}
	One could also write the covariant derivative in terms of the adjoint operation. Clearly, $\mathrm{ad}_{\Lambda A \Lambda^{-1}} = \Lambda ~\mathrm{ad}_A \Lambda^{-1}$. Now, what one has to notice is that the Hodge-* operates only on spatial indices and not on gauge indices; hence,
	\begin{equation}
		* \left( \Lambda A \Lambda^{-1} \right) = \Lambda \left( * A \right) \Lambda^{-1}.
	\end{equation}
	Putting these two facts together, it is evident that
	\begin{equation}
		- \text{\TH}^*_{-\Lambda A \Lambda^{-1}} * \text{\TH}_{\Lambda A \Lambda^{-1}} \left( \Lambda A \Lambda^{-1} \right) = - \frac{1}{\varepsilon_0} * \left( \Lambda J_1 \Lambda^{-1} \right) 
	\end{equation}
	becomes
	\begin{equation}
		- \Lambda \left( \text{\TH}^*_{-A} * \text{\TH}_A A \right) \Lambda^{-1} = - \frac{1}{\varepsilon_0} \Lambda (* J_1) \Lambda^{-1}. 
	\end{equation}
	Therefore, if the field configuration described by $A$ is a solution, so will be that corresponding to $\Lambda A \Lambda^{-1}$, as was to be shown.
\end{proof}

Let us particularize to the third-order case; i.e., we truncate all jets to their expansion up to the third order. Here, the number of timelike directions will be $p=1+3+7=11$ and the number of spacelike directions, $q=3+7+13=23$. Recall the decomposition of the jet affine connection 1-forms by the grading by jet degree:
\begin{equation}
	A = A^{(1,1)} + A^{(1,2)} + A^{(1,3)} + A^{(2,2)} + A^{(2,3)} + A^{(3,3)}. 
\end{equation}
We have seen in {\S}\ref{unification_gravity_electroweak} that in the (2,2)-sector, when treated up to equivalence under change of gauge, the number of effective degrees of freedom among the 2-jets goes from ten to four (one spatial, three timelike); that is, we can now replace $\mathrm{so}(11,23)$ with $\mathrm{so}(11,17)$. 

The guiding principle behind the mechanism of spontaneous symmetry breaking is that the full gauge symmetry of $\mathrm{so}(11,17)$ becomes broken by the Proca term, however. The reason for this is that the mass matrix does not respect the full gauge symmetry, but only a subgroup of it. The source of the non-invariance lies in an incompatibility between the Lie algebraic structure of the gauge field potentials $\omega^\alpha_\beta$, which belong to $\mathrm{so}(20)$, and their pseudo-Riemannian spatio-temporal structure as currents: when one raises the index from $\omega_{\alpha\beta}$ resp. lowers the index from $\omega^{\alpha\beta}$ it introduces minus signs in the timelike factors of the jets, which cause the Proca term to have a Killing inner product appropriate to $\mathrm{so}(7,13)$ instead of $\mathrm{so}(20)$. This may be seen most perspicuously by writing the Proca term as a quadratic form in the gauge fields. We found above in {\S}\ref{spont_symm_breaking} the following representation of the Proca term:
\begin{equation}
	\Delta \mathscr{L} = \frac{1}{4} \pi \varkappa \lambda^4 B^2 \left( B_0^2 + 2 A_0 B_0 + A_0^2 + A_1^2 + A_2^2 \right).
\end{equation}
Here, $B_0$ represents the generator of weak hypercharge and $A_{0,1,2}$ the generators of weak isospin, which in terms of the Cartesian coordinates of the inertial frame look like
\begin{equation}\label{electroweak_sector}
	A^{(2,2)} = \frac{1}{2} \begin{pmatrix}
		0 & - 2 A_2 & 2A_1 & 2A_0 - B_0 \cr
		2A_2 & 0 & 2A_0 + B_0 & -2A_1 \cr	
		-2A_1 & -2A_0 - B_0 & 0 & -2A_2 \cr
		2A_0 + B_0 & 2A_1 & 2A_2 & 0 \cr
	\end{pmatrix}.	
\end{equation}
We further saw that the Proca term in the 2-jets could be written as (after eliminating $B_{1,2}$ and $X_1$),
\begin{align}
	\Delta \mathscr{L}_{(2,2)} &=  \frac{1}{4} \pi \varkappa \lambda^4 B^2 \left(
	(2A_0+B_0)^2 + 4A_1^2 + 4A_2^2 \right) \nonumber \\
	&= \frac{1}{2} \pi \varkappa \lambda^4 B^2 ~\mathrm{tr}~ h^{(2,2)} A^{(2,2)} A^{(2,2)}
\end{align}
with
\begin{equation}
	h^{(2,2)} = \begin{pmatrix} 1 & 0 & 0 & 0 \cr 0 & -1 & 0 & 0 \cr 0 & 0 & -1 & 0 \cr 0 & 0 & 0 & -1 \cr \end{pmatrix}.
\end{equation}
We would like to extend the representation of the Proca term as a Killing form to 3-jets. To do so, notice first the peculiarity that we do not have any Proca term in the (1,1)-sector, for we choose to keep the mass term on the right-hand side of Einstein's field equations; that is, as usual, one views the stress-energy tensor as a source of the gravitational field. 

Therefore, remembering that the (2,2)-sector has signature (3,1) while the (3,3)-sector has signature (7,13), one can define
\begin{align}
	h = \mathrm{diag}~ (&0,0,0,0, \nonumber \\
	&1,-1,-1,-1, \nonumber \\
	&-1,-1,-1,-1,-1,-1,-1,1,1,1,1,1,1,1,1,1,1,1,1,1).
\end{align} 
and write the full Proca term as
\begin{equation}
	\Delta \mathscr{L} = \frac{1}{2} \pi \varkappa \lambda^4 B^{|\alpha|+|\beta|-2} ~\mathrm{tr}~ h A A.
\end{equation}
To have a unified notation, let us denote the generators of $\mathrm{so}(11,17)$ by $X_{ij}$ where $(X_{ij})_{kl} = \delta_{ik}\delta_{jl}-\delta_{il}\delta_{jk}$. In order to get an impression of what the mass matrix looks like (which will be handy to refer to in later sections), it will be useful to display the entries of $\Delta \mathscr{L}$ for $A = X_{ij}$ with $1 \le i,j \le 28$; i.e., $\mathrm{tr}~ h X_{ij} X_{ij}$. See equation (\ref{mass_matrix_illustration}). The block structure may be clearly seen, both as to magnitude and as to sign of the respective traces.

\begin{landscape}
	\tiny
	\renewcommand{\arraystretch}{1.618}
	\begin{equation}\label{mass_matrix_illustration}
		\begin{array}{*{28}c}
			0 & 0 & 0 & 0 & -1 & 1 & 1 & 1 & 1 & 1 & 1 & 1 & 1 & 1 & 1 & -1 & -1 & -1 & -1 & -1 & -1 & -1 & -1 & -1 & -1 & -1 & -1 & -1 \\
			0 & 0 & 0 & 0 & -1 & 1 & 1 & 1 & 1 & 1 & 1 & 1 & 1 & 1 & 1 & -1 & -1 & -1 & -1 & -1 & -1 & -1 & -1 & -1 & -1 & -1 & -1 & -1 \\
			0 & 0 & 0 & 0 & -1 & 1 & 1 & 1 & 1 & 1 & 1 & 1 & 1 & 1 & 1 & -1 & -1 & -1 & -1 & -1 & -1 & -1 & -1 & -1 & -1 & -1 & -1 & -1 \\
			0 & 0 & 0 & 0 & -1 & 1 & 1 & 1 & 1 & 1 & 1 & 1 & 1 & 1 & 1 & -1 & -1 & -1 & -1 & -1 & -1 & -1 & -1 & -1 & -1 & -1 & -1 & -1 \\
			-1 & -1 & -1 & -1 & 0 & 0 & 0 & 0 & 0 & 0 & 0 & 0 & 0 & 0 & 0 & -2 & -2 & -2 & -2 & -2 & -2 & -2 & -2 & -2 & -2 & -2 & -2 & -2 \\
			1 & 1 & 1 & 1 & 0 & 0 & 2 & 2 & 2 & 2 & 2 & 2 & 2 & 2 & 2 & 0 & 0 & 0 & 0 & 0 & 0 & 0 & 0 & 0 & 0 & 0 & 0 & 0 \\
			1 & 1 & 1 & 1 & 0 & 2 & 0 & 2 & 2 & 2 & 2 & 2 & 2 & 2 & 2 & 0 & 0 & 0 & 0 & 0 & 0 & 0 & 0 & 0 & 0 & 0 & 0 & 0 \\
			1 & 1 & 1 & 1 & 0 & 2 & 2 & 0 & 2 & 2 & 2 & 2 & 2 & 2 & 2 & 0 & 0 & 0 & 0 & 0 & 0 & 0 & 0 & 0 & 0 & 0 & 0 & 0 \\
			1 & 1 & 1 & 1 & 0 & 2 & 2 & 2 & 0 & 2 & 2 & 2 & 2 & 2 & 2 & 0 & 0 & 0 & 0 & 0 & 0 & 0 & 0 & 0 & 0 & 0 & 0 & 0 \\
			1 & 1 & 1 & 1 & 0 & 2 & 2 & 2 & 2 & 0 & 2 & 2 & 2 & 2 & 2 & 0 & 0 & 0 & 0 & 0 & 0 & 0 & 0 & 0 & 0 & 0 & 0 & 0 \\
			1 & 1 & 1 & 1 & 0 & 2 & 2 & 2 & 2 & 2 & 0 & 2 & 2 & 2 & 2 & 0 & 0 & 0 & 0 & 0 & 0 & 0 & 0 & 0 & 0 & 0 & 0 & 0 \\
			1 & 1 & 1 & 1 & 0 & 2 & 2 & 2 & 2 & 2 & 2 & 0 & 2 & 2 & 2 & 0 & 0 & 0 & 0 & 0 & 0 & 0 & 0 & 0 & 0 & 0 & 0 & 0 \\
			1 & 1 & 1 & 1 & 0 & 2 & 2 & 2 & 2 & 2 & 2 & 2 & 0 & 2 & 2 & 0 & 0 & 0 & 0 & 0 & 0 & 0 & 0 & 0 & 0 & 0 & 0 & 0 \\
			1 & 1 & 1 & 1 & 0 & 2 & 2 & 2 & 2 & 2 & 2 & 2 & 2 & 0 & 2 & 0 & 0 & 0 & 0 & 0 & 0 & 0 & 0 & 0 & 0 & 0 & 0 & 0 \\
			1 & 1 & 1 & 1 & 0 & 2 & 2 & 2 & 2 & 2 & 2 & 2 & 2 & 2 & 0 & 0 & 0 & 0 & 0 & 0 & 0 & 0 & 0 & 0 & 0 & 0 & 0 & 0 \\
			-1 & -1 & -1 & -1 & -2 & 0 & 0 & 0 & 0 & 0 & 0 & 0 & 0 & 0 & 0 & 0 & -2 & -2 & -2 & -2 & -2 & -2 & -2 & -2 & -2 & -2 & -2 & -2 \\
			-1 & -1 & -1 & -1 & -2 & 0 & 0 & 0 & 0 & 0 & 0 & 0 & 0 & 0 & 0 & -2 & 0 & -2 & -2 & -2 & -2 & -2 & -2 & -2 & -2 & -2 & -2 & -2 \\
			-1 & -1 & -1 & -1 & -2 & 0 & 0 & 0 & 0 & 0 & 0 & 0 & 0 & 0 & 0 & -2 & -2 & 0 & -2 & -2 & -2 & -2 & -2 & -2 & -2 & -2 & -2 & -2 \\
			-1 & -1 & -1 & -1 & -2 & 0 & 0 & 0 & 0 & 0 & 0 & 0 & 0 & 0 & 0 & -2 & -2 & -2 & 0 & -2 & -2 & -2 & -2 & -2 & -2 & -2 & -2 & -2 \\
			-1 & -1 & -1 & -1 & -2 & 0 & 0 & 0 & 0 & 0 & 0 & 0 & 0 & 0 & 0 & -2 & -2 & -2 & -2 & 0 & -2 & -2 & -2 & -2 & -2 & -2 & -2 & -2 \\
			-1 & -1 & -1 & -1 & -2 & 0 & 0 & 0 & 0 & 0 & 0 & 0 & 0 & 0 & 0 & -2 & -2 & -2 & -2 & -2 & 0 & -2 & -2 & -2 & -2 & -2 & -2 & -2 \\
			-1 & -1 & -1 & -1 & -2 & 0 & 0 & 0 & 0 & 0 & 0 & 0 & 0 & 0 & 0 & -2 & -2 & -2 & -2 & -2 & -2 & 0 & -2 & -2 & -2 & -2 & -2 & -2 \\
			-1 & -1 & -1 & -1 & -2 & 0 & 0 & 0 & 0 & 0 & 0 & 0 & 0 & 0 & 0 & -2 & -2 & -2 & -2 & -2 & -2 & -2 & 0 & -2 & -2 & -2 & -2 & -2 \\
			-1 & -1 & -1 & -1 & -2 & 0 & 0 & 0 & 0 & 0 & 0 & 0 & 0 & 0 & 0 & -2 & -2 & -2 & -2 & -2 & -2 & -2 & -2 & 0 & -2 & -2 & -2 & -2 \\
			-1 & -1 & -1 & -1 & -2 & 0 & 0 & 0 & 0 & 0 & 0 & 0 & 0 & 0 & 0 & -2 & -2 & -2 & -2 & -2 & -2 & -2 & -2 & -2 & 0 & -2 & -2 & -2 \\
			-1 & -1 & -1 & -1 & -2 & 0 & 0 & 0 & 0 & 0 & 0 & 0 & 0 & 0 & 0 & -2 & -2 & -2 & -2 & -2 & -2 & -2 & -2 & -2 & -2 & 0 & -2 & -2 \\
			-1 & -1 & -1 & -1 & -2 & 0 & 0 & 0 & 0 & 0 & 0 & 0 & 0 & 0 & 0 & -2 & -2 & -2 & -2 & -2 & -2 & -2 & -2 & -2 & -2 & -2 & 0 & -2 \\
			-1 & -1 & -1 & -1 & -2 & 0 & 0 & 0 & 0 & 0 & 0 & 0 & 0 & 0 & 0 & -2 & -2 & -2 & -2 & -2 & -2 & -2 & -2 & -2 & -2 & -2 & -2 & 0 
		\end{array}
	\end{equation}
	\normalsize
	\begin{equation}
		\mathrm{tr}~h X_{ij} X_{ij}, \qquad 1 \le i,j \le 28. \nonumber
	\end{equation}
\end{landscape}

\subsubsection{Reduction to the Totally Isotropic Subspace}\label{reduction_to_isotropic}

In this section, we focus on the (3,3)-sector of gauge fields $A^{(3,3)}$. Considered by themselves, these generate an $SO(7,13)$ subgroup of the full gauge group with Lie algebra $\mathrm{so}(7,13)$. If we look at the Killing form, however, it is easy to see why it cannot be invariant under the full gauge group. One finds that $\frac{1}{2} \cdot 6 \cdot 7 = 21$ modes have positive mass squared, $\frac{1}{2} \cdot 12 \cdot 13 = 78$ have negative mass squared and $7 \cdot 13 = 91$ are massless. But, as we have said, not all of these generate gauge transformations preserve the value of the Proca Lagrangian $\Delta \mathscr{L}$. In order to see this, we first write out the Proca term explicitly as
\begin{equation}
	\Delta \mathscr{L}_{(3,3)} = \pi \varkappa \lambda^4 B^2 \left( \sum_{1 \le i < j \le 7} a_{ij}^2 - \sum_{8 \le i < j \le 20} a_{ij}^2 \right),
\end{equation}
where we take for the gauge field the general form
\begin{equation}
	A^{(3,3)} = \sum_{1 \le i < j \le 20} a_{ij} X_{ij}.
\end{equation}
Consider first the zero modes in the off-diagonal blocks. Without loss of generality, choose one of them, say, $X_{7,8}$. Upon exponentiation to obtain a rotation matrix $R_{(7,8)}(\theta) = \exp (-\theta X_{7,8})$, one finds a $2 \times 2$ block of the form
\begin{equation}
	\begin{pmatrix}
		\cos \theta & - \sin \theta \\
		\sin \theta & \cos \theta \\
	\end{pmatrix}
\end{equation}
in the $i=7,8$ and $j=7,8$ positions, 1 on the rest of the diagonal and zero everywhere else. Compute how the Proca term transforms under the adjoint action of $R_{(7,8)}(\theta)$ as follows:
\begin{align}\label{rotated_proca}
	\Delta \mathscr{L}_{(3,3)} = \pi \varkappa \lambda^4 B^2 & \left[ \sum_{1 \le i < j \le 6} a_{ij}^2 - \sum_{9 \le i < j \le 20} a_{ij}^2 
	+ \right. \nonumber \\
	& \left. \left( \sum_{1 \le i \le 6} a_{i7}^2 + \sum_{9 \le i \le 20} a_{7i}^2 - \sum_{1 \le j \le 6} a_{j8}^2 - \sum_{9 \le j \le 20} a_{8j}^2 \right) \cos 2 \theta + \right. \nonumber \\
	& \left. \left( \sum_{1 \le i \le 6} a_{i7} a_{i8} + \sum_{9 \le j \le 20} a_{7j} a_{8j} \right) \sin 2 \theta 
	\right].
\end{align}
Equation (\ref{rotated_proca}) is manifestly not invariant for the most general field configuration and the same obviously applies for any rotation generated by $X_{ij}$ with $1 \le i \le 7$, $8 \le j \le 20$. Thus, we have eliminated the off-diagonal blocks and can reduce the possible gauge algebra in the (3,3)-sector from $\so(7,13)$ to $\so(7) \oplus \so(13)$.

But now evidently the 21 modes in $\so(7)$ have positive mass squared while the 78 modes in $\so(13)$ have negative mass squared. In other words, the Proca term corresponds to a non-degenerate indefinite quadratic form of signature $(21,78)$, for we have already eliminated the degenerate off-diagonal block. Now, being indefinite, it admits non-zero isotropic vectors $\varv$ such that $\| \varv \| = 0$. In fact, the maximum number of linearly independent isotropic vectors is equal to its isotropy index, viz., $\min(21,78) = 21$.

Therefore, it is expedient to decompose $\so(7) \oplus \so(13)$ into a direct sum $V^{(3,3)} \oplus W$ of a totally isotropic subspace $V$ of dimension 21 and an anisotropic subspace $W$ of dimension 78. The vectors in $W$ drop out of the problem now because they are not massless, and we presume their mass to be so large that there cannot be any excitations of these fields at currently accessible energies. In {\S}\ref{chromodynamic_scale} below we shall estimate the scale of chromodynamical symmetry breaking and verify our surmise.

What about the totally isotropic subspace $V^{(3,3)}$? First of all, it is acted upon by the residual gauge symmetry with Lie algebra $\so(7)_+ \oplus \so(13)_-$ (where the subscript indicates whether the vectors in the relevant summand have positive or negative mass squared, respectively), for it may be written as the span of 28 linearly independent vectors of the form
\begin{equation}
	\varv_i = \varv_i^+ + \varv_i^-, \qquad (i=1,\ldots,21)
\end{equation}
where $\varv_i^+ \in \so(7)_+$ and $\varv_i^- \in \so(13)_-$. Thus, given rotations $R_+ \in SO(7)_+$ and $R_- \in SO(13)_-$ we have 
\begin{equation}
	\varv_i \mapsto R_+ \varv_i^+ R_+^{-1} + R_- \varv_i^- R_-^{-1}
\end{equation}
(remember that the $\varv_i^\pm$ are matrices in $\so(7,13)$ and therefore transform under the adjoint action). Thus, the residual gauge group $SO(7)_+ \otimes SO(13)_-$ acts by transforming vectors among the subspaces $V^\pm = \mathrm{span} (\varv_i^\pm: i=1,\ldots,21)$, respectively. Therefore, we are free to reduce the problem to the gauge equivalence classes by selecting a convenient representative for $V^{(3,3)}$. The easiest possibility to decide upon would be the following:
\begin{equation}
	\varv_{ij} = X_{i+8,j+8}+X_{i+15,j+15}, \qquad (1 \le i < j \le 7).
\end{equation}
Thus, $V^{(3,3)} = \mathrm{span}~ \varv_{ij} = \so(7)$ describes the gauge degrees of freedom in the (3,3)-sector left over after we fix this choice of the totally isotropic subspace.

\subsubsection{Identification of the Chromodynamical Degrees of Freedom}\label{identification_of_chromodynamics}

In {\S}\ref{reduction_to_isotropic} we have seen that the residual gauge symmetry in the (3,3)-sector must fall within the totally isotropic subspace $V^{(3,3)}$ which is isomorphic to $\so(7)$. At this point, another consideration becomes pertinent. Remember that the full initial gauge algebra for the Yang-Mills field equation without Proca term contains all sectors $A^{(|\alpha|,|\beta|)}$ with $|\alpha|,|\beta|=1,2,3$. Therefore we cannot treat the (3,3)-sector in isolation but must look at how it interacts with the others. Evidently the $A^{(1,1)}$, $A^{(2,2)}$ and $A^{(3,3)}$ blocks commute among themselves. But they are not completely independent of one another since they can couple indirectly through the off-diagonal sectors. Thus, we have to inspect more closely how the diagonal sectors couple to the off-diagonal.

To state the problem in some generality, consider the Lie group $SO(n+m)$ along with its subgroups $SO(n)$ resp. $SO(m)$ contained in the upper left resp. lower right $n \times n$ resp. $m \times m$ diagonal blocks. At the Lie algebra level, we have $\so(n) \oplus \so(m) \subset \so(n+m)$ with the obvious embedding. Consider now the span $T$ of the Lie algebra generators lying in the off-diagonal $n \times m$ block. Now, $T$ does not form an invariant subspace with respect to the adjoint action of $SO(n+m)$, but it is invariant under the factors $SO(n)$ resp. $SO(m)$ acting individually. In fact, consider $x \in T_1$ of the form $x = \sum_{j=n+1,\ldots,n+m} x_j X_{1j}$; i.e., it occupies the uppermost row in the upper right block and the leftmost row in the lower left block inside $\so(n+m)$. What strikes us as a surprising fact is the following: while $SO(m)$ acts on $\so(m)$ via its adjoint action $y \mapsto A y A^{-1}$, it acts on $T_1$ directly via its left action $x \mapsto A x$, not, that is to say, via the adjoint action. The same applies to the other row resp. columns inside $T$; i.e., $T$ decomposes into a direct sum of $n$ subspaces of dimension $m$ each, namely, $T = T_1 \oplus \cdots \oplus T_n$, which individually are invariant under the $SO(m)$ action.

This simple observation, obvious enough in itself, turns out to be crucial to our problem. For, as always, we require preservation of the Proca term and it will not be the case that all of the initial gauge degrees of freedom do preserve it. Our present concern is to find out what happens inside the (3,3)-sector. 

The first matter to be dealt with will be to describe the structure of the (1,3) and (2,3) sectors. The (1,3)-sector has $4 \cdot 20 = 80$ dimensions which we will write as the span
\begin{equation}
	A^{(1,3)} = \spn \left( b_{ij} X_{i+8,j}, 1 \le i \le 20, 1 \le j \le 4 \right)
\end{equation}
and its contribution to the Proca term is given by
\begin{align}
	\Delta \mathscr{L}_{(1,3)} &=  \frac{1}{2} \pi \varkappa \lambda^4 B^2 ~\mathrm{tr}~ h A^{(1,3)} A^{(1,3)} \nonumber \\
	&= \frac{1}{2} \pi \varkappa \lambda^4 B^2 \left( \sum_{1 \le i \le 7, 1 \le j \le 4} b_{ij}^2 - \sum_{8 \le i \le 20, 1 \le j \le 4} b_{ij}^2 \right).
\end{align}
Here, we have a non-degenerate quadratic form of signature (28,52). A totally isotropic subspace $V^{(1,3)}$ would have dimension 28 but the other 52 massive modes do \textit{not} fall out of the problem altogether, since, as we shall see in {\S}\ref{chromodynamic_scale} below, their estimated masses are expected to be comparable to the electroweak scale. Therefore, it is conceivable that they \textit{could} figure in physics well below the Planck regime. Nevertheless, it is also predicted that any interactions they mediate must be just as weak as those mediated by the comparably massive intermediate vector bosons in the $\su(2)$ sector, viz. the $Z^0$ and the $W^\pm$, which may explain why they have gone undetected up to now. 

The (2,3)-sector, on the other hand, has dimension 80 as well and equals the span
\begin{equation}
	A^{(2,3)} = \spn \left( c_{ij} X_{i+8,j+4}, 1 \le i \le 20, 1 \le j \le 4 \right).
\end{equation}
We see from this that the (2,3)-sector has 21 positive modes, 13 negative modes and 46 zero modes, leaving its contribution to the Proca term as follows: 
\begin{align}
	\Delta \mathscr{L}_{(2,3)} &= \pi \varkappa \lambda^4 B^2 ~\mathrm{tr}~ h A^{(2,3)} A^{(2,3)} \nonumber \\
	&= \pi \varkappa \lambda^4 B^2 \left( - \sum_{8 \le i \le 20} c_{i1}^2 + \sum_{1 \le i \le 7, 2 \le j \le 4} c_{ij}^2 \right).
\end{align}
In the next section, we shall show how to eliminate the degeneracy spanned by the 34 zero modes. For now, consider only the span of the remaining 7 positive and 39 negative modes forming an indefinite quadratic form of signature (7,39). As previously with the (1,3) and (3,3)-sectors, we go to a totally isotropic subspace $V^{(2,3)}$ of dimension 7 and eliminate the other 39 massive modes. According to the estimate in {\S}\ref{chromodynamic_scale} below, moreover, the massive modes in the $(2,2)$ sector must be entirely negligible in physics well below the Planck regime.

We must take care to define $V^{(2,3)}$ in such a way as to preserve the unbroken electromagnetic $\u(1)_Y$ symmetry explicitly. For the action of the electromagnetic gauge algebra $\u(1)_Y$ receives, in view of equation (\ref{electroweak_sector}), a representation of the following form. From {\S}\ref{spont_symm_breaking}, we have the definitions
\begin{align}
	B_0 &= \cos \theta_W \gamma - \sin \theta_W Z \nonumber \\
	A_0 &= \sin \theta_W \gamma + \cos \theta_W Z,
\end{align}
whence 
\begin{align}
	2 A_0 - B_0 &= \left( 2 \sin \theta_W - \cos \theta_W \right) \gamma + \left(2 \cos \theta_W + \sin \theta_W \right) Z = \sqrt{5} Z \nonumber \\
	2 A_0 + B_0 &= \left( 2 \sin \theta_W + \cos \theta_W \right) \gamma + \left( 2 \cos \theta_W - \sin \theta_W \right) Z = \frac{4}{\sqrt{5}} \gamma + \frac{3}{\sqrt{5}} Z.
\end{align}
Hence, $\gamma$ itself appears only in the 23 and 32 components in equation (\ref{electroweak_sector}) and generates the electromagnetic $U(1)_Y$ via
\begin{equation}
	\begin{pmatrix} \cos \frac{4}{\sqrt{5}} \gamma & - \sin \frac{4}{\sqrt{5}} \gamma \\ \sin \frac{4}{\sqrt{5}} \gamma & \cos \frac{4}{\sqrt{5}} \gamma \\ \end{pmatrix}
\end{equation}
in the 6-th and 7-th rows and columns of $\so(11,17)$ with unit entries on the rest of the diagonal. We want $\tr ~h V^{(2,3)} V^{(2,3)}$ to be invariant under this action. The following definition evidently satisfies this stipulation:
\begin{equation}
	V^{(2,3)} = \spn \left( \varv_i = X_{i+15,5} + \frac{1}{\sqrt{2}} X_{i+8,6} + \frac{1}{\sqrt{2}} X_{i+8,7}, 1 \le i \le 7 \right).
\end{equation}

Now, the $\so(7)$ gauge algebra in the (3,3)-sector acts on $V^{(1,3)}$ and $V^{(2,3)}$ in the manner just outlined above. In order to picture the geometrical situation, it will be helpful to regard $V^{(1,3)}$ resp. $V^{(2,3)}$ as isomorphic to the span $\imo$ of the pure imaginary octonions (we refer to the review by Baez \cite{baez}, who covers the basic properties relating to the octonions we quote without further justification in what follows). In octonionic terms, the gauge Lie algebra in question has the structure
\begin{equation}
	\so(7) \cong \der(\vvmathbb{O}) \oplus \mathrm{ad}_{\imo} \cong \g_2 \oplus \mathrm{ad}_{\imo},
\end{equation}
where $\der(\vvmathbb{O})$ are the derivations of the algebra of octonions and $\mathrm{ad}_{\imo}$ indicates the transformation for every $a \in \imo$,
\begin{equation}
	\mathrm{ad}_a = L_a - R_a,
\end{equation}
with $L_a$ and $R_a$ denoting the linear transformation of $\imo$ into itself given by multiplication from the left resp. from the right by $a$. Our strategy will be to reduce $\so(7)$ to the exceptional Lie algebra $\g_2$ and thence to a copy of $\su(3)$ sitting inside $\g_2$. These subgroups of $\so(7)$ have the following geometrical interpretation: $\g_2$ generates rotations $g$ of $\imo$ that preserve the 7d cross-product in the sense that $ga \times gb = g (a \times b)$ for all $a,b \in \imo$. For every element $a \in \imo$, its stabilizer subgroup is generated by a copy of $\su(3) \subset \g_2$.

Now, we wish to connect these abstract notions with the concrete constructions we have been pursuing. The reason the cross-product in $\imo$ is pertinent to our problem is that it is yielded by the action of our restricted gauge Lie algebra $\so(7)$ on $V^{(1,3)}$ resp. $V^{(2,3)}$. To employ a neutral notation, we will refer to one or the other totally isotropic subspaces as $\imo$. Given $X \in \mathrm{ad}_{\imo}$, one may readily check that for $Y \in \imo$,
\begin{equation}
	\mathrm{ad}_X Y = [X,Y] = X \times Y,
\end{equation}
where in the last equation we treat $X, Y$ as lying in $\imo$. It will be instructive to exhibit the identity $[X,Y] = X \times Y$ explicitly, for it winds up being the heart of the matter. In order to do so, we must first define the cross-product in $\imo$. 

To define the cross-product, one has recourse to the multiplication table of unit octonions $1, e_1, e_2, e_3, e_4, e_5, e_6, e_7$, our choice of which we reproduce here:
\begin{equation}
	\left[
	\begin{array}{c|ccccccc}
		1 & e_1 & e_2 & e_3 & e_4 & e_5 & e_6 & e_7 \\
		\hline
		e_1 & -1 & e_3 & -e_2 & e_5 & -e_4 & -e_7 & e_6 \\
		e_2 & -e_3 & -1 & e_1 & e_6 & e_7 & -e_4 & -e_5 \\
		e_3 & e_2 & -e_1 & -1 & e_7 & -e_6 & e_5 & -e_e \\
		e_4 & -e_5 & -e_6 & -e_7 & -1 & e_1 & e_2 & e_3 \\
		e_5 & e_4 & -e_7 & e_6 & -e_1 & -1 & -e_3 & e_2 \\
		e_6 & e_7 & e_4 & -e_5 & -e_2 & e_3 & -1 & -e_1 \\
		e_7 & -e_6 & e_5 & e_4 & -e_3 & -e_2 & e_1 & -1 \\
	\end{array}
	\right].	
\end{equation}
In octonionic terms, the cross-product is given by $a \times b = ab + \langle a, b \rangle$. For unit imaginary octonions one can read off the products from the multiplication table and just ignore the scalar part $\langle a, b \rangle$. When one does so, one obtains the cross-product restricted to $\imo$ in terms of components as follows:
\begin{equation}\label{cross_product}
	\begin{pmatrix} a_1 \\ a_2 \\ a_3 \\ a_4 \\ a_5 \\ a_6 \\ a_7 \\ \end{pmatrix} 
	\times
	\begin{pmatrix} b_1 \\ b_2 \\ b_3 \\ b_4 \\ b_5 \\ b_6 \\ b_7 \\ \end{pmatrix} =
	\begin{pmatrix} 
		a_2 b_3 - a_3 b_2 - a_5 b_4 + a_4 b_5 + a_7 b_6 - a_6 b_7  \\
		- a_1 b_3 + a_3 b_1 - a_6 b_4 + a_4 b_6 - a_7 b_5 - a_5 b_7  \\
		a_1 b_2 - a_2 b_1 + a_4 b_7 - a_7 b_4 + a_6 b_5 - a_5 b_6  \\
		- a_1 b_5 + a_5 b_1 - a_2 b_6 + a_6 b_2 - a_3 b_7 + a_7 b_3 \\
		a_1 b_4 - a_4 b_1 - a_2 b_7 + a_7 b_2 + a_3 b_6 - a_6 b_3  \\
		a_1 b_7 - a_7 b_1 + a_2 b_4 - a_4 b_2 - a_3 b_5 + a_5 b_3  \\
		- a_1 b_6 + a_6 b_1 + a_2 b_5 - a_5 b_2 + a_3 b_4 - a_4 b_3  \\
	\end{pmatrix}.  
\end{equation}
We single out the 7-dimensional subspace $\mathrm{ad}_{\imo} \subset \so(\imo)$ as follows:
\begin{align}\label{adjoint_action}
	A &= a ~\mathrm{ad}_{e_1} + b ~\mathrm{ad}_{e_2} + c ~\mathrm{ad}_{e_3} + d ~\mathrm{ad}_{e_4} + e ~\mathrm{ad}_{e_5} + f ~\mathrm{ad}_{e_6} + g ~\mathrm{ad}_{e_7} \nonumber \\
	&= \begin{pmatrix}
		0 & c & -b & e & -d & -g & f \\
		-c & 0 & a & f & g & -d & -e \\
		b & -a & 0 & g & -f & e & -d \\
		-e & -f & -g & 0 & a & b & c \\
		d & -g & f & -a & 0 & -c & b \\
		g & d & -e & -b & c & 0 & -a \\
		-f & e & d & -c & -b & a & 0 \\
	\end{pmatrix}
\end{align}
Frequently in what follows we shall identify $\mathrm{ad}_{\imo}$ with $\imo$ itself via $\mathrm{ad}_{e_i} = e_i$. Since $\so(\imo)$ is 21-dimensional, each coefficient appears three times in equation (\ref{adjoint_action}). With this description, we have realized the decomposition
\begin{equation}
	\so(\imo) \cong \g_2 \oplus \spn~ A.
\end{equation}

\begin{proposition}
	The left action of $A$ on $\imo$ is equivalent to the cross-product $\varv \mapsto A \times \varv$. 
\end{proposition}
\begin{proof}
	One can readily verify the statement by direct computation from equations (\ref{cross_product}) and (\ref{adjoint_action}). The reason it works, of course, is that by construction the cross-product is supposed to be the non-scalar part of octonionic left multiplication\footnote{Namely, for pure imaginary octonions $a$ and $b$, the cross product satisfies the identity $a \times b = ab + \langle a, b \rangle$.} and, by the observation above in the second paragraph of this section, in our set-up $A$ acts as an operator on $\imo$ by multiplication from the left, where its entries may be inferred from the octonionic multiplication table.
\end{proof}
To find the generators of the 14-dimensional subalgebra $\g_2$, one has to select two orthogonal spaces in each coefficient in a manner consistent with the unit octonionic multiplication table. The procedure we shall follow is described by Dray and Manogue \cite{dray_manogue}. To define an automorphism of the pure imaginary octonions, one starts with a basic triple $e_1, e_2, e_3$ and supplements it with another unit, which we take to be $e_4$. Relative to this choice, one constructs $\g_2$ in the following manner.

For each $e_k$, $k=1,\ldots,7$, one can view it geometrically as generating a rotation of the planes spanned by $e_i$ and $e_j$ whenever $e_i e_j = e_k$ (where the order of the two factors specifies the counter-clockwise direction). Among the three elements in equation (\ref{adjoint_action}), one of them will always involve $e_4$. Thus, the other two correspond to a preferred pair of planes (for $k=4$ one may select these at will). Define $A_k$ as generating rotations by equal and opposite amounts with respect to this preferred pair and define $G_k$ as generating rotations by equal amounts about each of the two in the preferred pair and by twice as much in the opposite direction about the third element. This procedure gives $\g_2$ as the span
\begin{equation}
	\g_2 = \spn~ \left( A_1, A_2, A_3, A_4, A_5, A_6, A_7, G_1, G_2, G_3, G_4, G_5, G_6, G_7 \right),
\end{equation} 
where in components we have
\begin{align}
	a A_1 + b A_2 + c A_3 + d A_4 + e A_5 + f A_6 + g A_7 & =
	\begin{pmatrix}
		0 & c & -b & 0 & -d & -g & f \\
		-c & 0 & a & 0 & -g & d & -e \\
		b & -a & 0 & 0 & f & -e & 0 \\
		0 & 0 & 0 & 0 & 0 & 0 & 0 \\
		d & g & -f & 0 & 0 & c & -b \\
		g & -d & e & 0 & -c & 0 & a \\
		-f & e & 0 & 0 & b & -a & 0 \\
	\end{pmatrix} \\
	a G_1 + b G_2 + c G_3 + d G_4 + e G_5 + f G_6 + g G_7 & =
	\begin{pmatrix}
		0 & c & -b & -2e & -d & -g & f \\
		-c & 0 & a & -2f & g & -d & -e \\
		b & -a & 0 & -2g & -f & e & 2d \\
		2e & 2f & 2g & 0 & -2a & -2b & -2c \\
		d & -g & f & 2a & 0 & -c & -b \\
		g & d & -e & 2b & c & 0 & -a \\
		-f & e & -2d & 2c & b & a & 0 \\
	\end{pmatrix}.
\end{align}
Now, $\g_2$ is by definition the space of derivations of $\vvmathbb{O}$, or equivalently, consists of infinitesimal automorphisms of $\vvmathbb{O}$. The derivation property may be written out, for $X \in \g_2$ and $A, B \in \imo$, thus:
\begin{equation}\label{derivation_property}
	X (AB) = (XA)B + A(XB).
\end{equation}
But $\g_2$ may equivalently be regarded as the subspace of $\so(\imo)$ that preserves the cross-product, or, in other words, equation (\ref{derivation_property}) is equivalent to
\begin{equation}\label{derivation_cross_prod}
	X ( A \times B) = (XA) \times B + A \times (XB)	
\end{equation}
for any $X \in \g_2$ and for all $A, B \in \imo$. For $X \in \mathrm{ad}_\imo$, on the other hand, the derivation property will not hold in general.

\begin{proposition}\label{reduce_from_so7_to_g2}
	If $\mathrm{ad}_{\imo}$ were part of the gauge Lie algebra, it would lead to an inconsistency.
\end{proposition}
\begin{proof}
	Suppose $V \ne 0 \in \imo$. Take any $X \in V^\perp$ so that $X \times V \ne 0$.
	Then we may find an $A \in \mathrm{ad}_{\imo}$ such that
	\begin{equation}
		X ( A \times V) \ne XA \times V + A \times XV.
	\end{equation}
	But by reason of the identifications in place and the associativity of matrix multiplication, $X (A \times V) = X(AV) = (XA)V = (XA) \times V - \langle XA, V \rangle$. Therefore, $A \times XV \ne - \langle XA , V \rangle$ viewed as octonions. Then we can project both sides onto $\imo$ to conclude that $A \times XV \ne 0$. Now we can envisage the potential gauge transformation from two points of view. First, one could start from $V$ and apply the gauge transformation $A$ followed by the gauge transformation $X$, which would lead to an infinitesimal change in $V$ by the amount $XAV$. Second, one could start from the perfectly legitimate (under the assumption) gauge transformation $[X,A]$ which applied to $V$ results in an infinitesimal change by $[X,A]V = XAV - AXV$. But we have just seen that $AXV \ne 0$, contradiction.
\end{proof}
Therefore, the admissible gauge Lie algebra reduces from $\so(\imo)$ to $\g_2$. 
As a strictly mathematical statement, the restricted gauge algebra also has to satisfy the Jacobi identity in itself. That is, we can take the commutator inside $A^{(3,3)}$ or act on $V^{(1,3)}$ resp. $V^{(2,3)}$ first. Consequently, with $X, Y \in A^{(3,3)}$ and $Z \in V^{(2,3)}$ the Jacobi identity, viz.,
\begin{equation}\label{jacobi_id}
	[[X,Y],Z] = [X,[Y,Z]] - [Y,[X,Z]]
\end{equation}
must hold as a consistency condition. 

\begin{proposition}\label{reduce_from_g2_to_su3}
	If $X \in \g_2$, the Jacobi identity implies an inconsistency unless $XZ = 0$ where $Z \in \imo$ indicates the value of the gauge potential in $V^{(2,3)} \subset A^{(2,3)}$.
\end{proposition}
\begin{proof}
	In view of the observation above in the second paragraph of this section, however, we are free to treat the carrier space of the gauge Lie algebra as isomorphic to $\imo$, in which case we write $X,Y \in \mathrm{so}(\imo)$, and regard $Z$ itself as lying in $\imo$. Then the consistency condition (\ref{jacobi_id}) becomes for $X \in \g_2$ and $Y \in \mathrm{ad}_{\imo}$,
	\begin{align}
		[X,Y]Z &= XYZ - YXZ \nonumber \\
		&= X (Y \times Z) - Y \times (XZ) \nonumber \\
		&= (XY) \times Z + Y \times (XZ) - Y \times (XZ) \nonumber \\
		&= (XY) \times Z.
	\end{align}
	To go from the second to third lines, use equation (\ref{derivation_cross_prod}). But the left hand side reads
	\begin{equation}
		[X,Y]Z = [X,Y] \times Z = (XY) \times Z - (YX) \times Z,
	\end{equation}
where we use the fact that
\begin{equation}
[g_2,g_2] \subset g_2; \qquad [g_2, \mathrm{ad}_{\imo}] \subset \mathrm{ad}_{\imo}; \qquad 
[ \mathrm{ad}_{\imo}, \mathrm{ad}_{\imo} ] \subset g_2.
\end{equation}
Therefore, we must have that 
\begin{equation}
(YX) \times Z = (YX)Z = Y(XZ) = Y \times (XZ) = 0.
\end{equation}
If $XZ \ne 0$, take $Y \in XZ^\perp \subset \imo$ so that $Y \times (XZ) \ne 0$, contradiction. For instance, $Y=Z$ would work. Thus we conclude that for the sake of consistency we must have $XZ = 0$. 
\end{proof}
Apart from the singular case $Z=0$, this implies that $Z$ is fixed under the action of $X \in \g_2$. But the set of elements in $\g_2$ fixing a given non-zero element $Z \in \imo$ is isomorphic to $\su(3)$. Therefore, we have reduced the gauge Lie algebra from $\g_2$ down to a copy of $\su(3) \subset \so(\imo)$. We can stop here because $\su(3)$ closes on itself under the Lie bracket. The copy of $\su(3)$ inside $\so(\imo)$ so determined may depend on position in space. That is, the isomorphism between $V^{(2,3)} \cong \imo$ in the formulae above should be chosen in a position-dependent manner such that $e_4$ is parallel to $V \in V^{(2,3)}$ everywhere in space-time.

\begin{corollary}
	The effective gauge Lie algebra at third order in jets is isomorphic to a copy of $\su(3)$.
\end{corollary}

\begin{remark}
If the chromodynamical gauge degrees of freedom were to be viewed as existing by themselves, in isolation, then propositions \ref{reduce_from_so7_to_g2} and \ref{reduce_from_g2_to_su3} would not apply. The reduction to a copy $\su(3)$ occurs only because of the consistency conditions that arise when the chromodynamical degrees of freedom in the $(3,3)$ sector are placed into juxtaposition with the rest of the jet affine connection 1-forms.
\end{remark}

\subsection{Estimate of the Scale of Chromodynamical Symmetry Breaking}\label{chromodynamic_scale}

To see most perspicuously where the mass terms at various orders in the jets come from, write out the Proca-Yang-Mills equation from {\S}\ref{chapter_8} in full as,
\begin{equation}\label{proca_yang_mills_chap_9}
	\text{\TH} * \text{\TH} A  +  2 \pi \varkappa \lambda^4 B^{|\alpha|+|\beta|-2} \alpha^{1/2} \varrho * A = - 2 \pi \varkappa \lambda^4 B^2 * J_1.
\end{equation}
The second term on the left goes as $M^2$. If we insert again the dimensional factors, we find for the mass scale of the gauge bosons (when they are not massless),
\begin{equation}\label{mass_scale_relation}
	M_{|\alpha|,|\beta|} = 2 \lambda^2 B^{\frac{1}{2}(|\alpha|+|\beta|)-1} \ell_P^{3-\frac{1}{2}(|\alpha|+|\beta|)} \alpha^{1/4} m_p.
\end{equation}
Here, the numerical coefficient is chosen so as to yield $M_W$ for $|\alpha|=|\beta|=2$. Thus, the expected mass scale in the low-lying sectors may be computed as given in Table 1.

\begin{table}[h!]
	\begin{center}
		\footnotesize	
		\caption{Mass scales of the low-lying sectors from equation (\ref{mass_scale_relation}).}
		\begin{tabular}{c c c d{1.5} c d{2.5} c}
			\noalign{\vskip 3mm}
			\hline\hline
			\noalign{\vskip 1mm}
			$|\alpha|$ & $|\beta|$ & $M_{|\alpha|,|\beta|}$ & \multicolumn{1}{c}{Mass scale} & & \multicolumn{1}{c}{Mass scale} & \cr
			& & & \multicolumn{1}{c}{[$m_P$]} & & \multicolumn{1}{c}{[GeV/$c^2$]} & \cr
			\noalign{\vskip 1mm}
			\hline
			\noalign{\vskip 2mm}
			1 & 2 & $M_{11}$ & 4.71485 & $\times 10^{-70}$ & 5.7831 & $\times 10^{-51}$ \cr
			1 & 3 & $M_{13}$ & 6.58364 & $\times 10^{-18}$ & 80.3790 & \cr
			2 & 2 & $M_{22}$ & 6.58364 & $\times 10^{-18}$ & 80.3790 & \cr
			2 & 3 & $M_{23}$ & 9.19315 & $\times 10^{34}$ & 1.12238 & $\times 10^{54}$ \cr
			3 & 3 & $M_{33}$ & 1.28370 & $\times 10^{87}$ & 1.56725 & $\times 10^{106}$ \cr
			\noalign{\vskip 2mm}
			\hline\hline
		\end{tabular}	
	\end{center}	
\end{table}
Therefore, the gauge fields in the $(1,2)$ sector remain effectively massless while the massive modes in the $(3,3)$ sector become far too massive for them to play any role in physics below the Planck regime, as was supposed above and has now been justified. The scale $M_{13}$, however, is of course the same as the electroweak scale in the $(2,2)$ sector. We shall comment in {\S}\ref{off_diagonal_sectors} below on what may be the physical role of the gauge fields in the off-diagonal sectors.

\subsection{Physical Significance of the Reduction to the Effective Gauge Lie Algebra}

We have now to clarify the \textit{physical} meaning of the restriction of the gauge Lie algebra implicit in propositions \ref{reduce_from_so7_to_g2} and \ref{reduce_from_g2_to_su3} as \textit{mathematical} statements. These reductions will been seen to correspond to an induction from experience, once their content in the relativistic context has been teased out, in analogy to what happens with the equivalence principle that Einstein emplaces as the ground of the transition from the special theory of relativity to the general theory. In particular, the implied limitation on the possible form matter may take on in the theory reveals itself as being very natural, just as much so as is the imposition of Lorentz covariance for tensors in the special theory itself.

\subsubsection{General Principles Governing the Process of Selection of the Effective Gauge Lie Algebra}\label{selection_effective_gauge_algebra}

At this point, we wish to model our procedure on what Einstein does when originally formulating the general theory of relativity, as has admirably been explicated by Norton (see \cite{einstein_grundlage}, \cite{norton_field_equ}, \cite{norton_physical_content}). Einstein's key realizations are first, that one can interpolate between the Lorentz covariance of the special theory of relativity and the full diffeomorphism group of the general theory; and second, that his equivalence principle enters as a means by which to solve the so-called problem of superfluous structure that could, to a na{\"i}ve observer, appear to be present at each intermediate stage. For instance, when one works in the usual Cartesian frame in special relativity, the level sets of the time coordinate specify hypersurfaces of simultaneity and the straight lines through each position in space determined by variation along the time coordinate deliver a set of congruences, yet none of these structures has fundamental significance in the full theory, as one can see right away by contemplating any non-linear coordinate transformation.

As we shall note in a moment, a similar situation obtains in our problem. But first let us comment that we have, in the context of infinitesimal geometry to higher order, a natural sense in which one can interpolate a sequence of enlargements of the symmetry group, from the ordinary Lorentz group all the way, in the limit, to the full diffeomorphism group. Namely, start with the Lorentz group to first order, then consider the Lorentz group to second order and so forth. Here, by an $r$-th order Lorentz transform we mean a coordinate mapping that is polynomial in the uniformizing Cartesian coordinates (given by polynomials up to the $r$-th degree) and that preserves the extended Minkowskian metric defined on jets up to the $r$-th order. A group structure results by composing any two coordinate mappings, modulo monomials of degree $r+1$ and higher. 

The definition supplies us with more than an idle mathematical curiosity, for it leads to a postulate having physical import:
\begin{postulate}[Principle of stepwise covariance]
	The structure of the gauge group of the full theory is to be analyzed by truncating to a given order and imposing covariance with respect to Lorentz transformations up to that order.
\end{postulate}
What may be the physical content of this principle? It preserves a role, at any given stage, for some of the so-called superfluous structure adverted to above in the sense that the Cartesian coordinates in the uniformizing frame continue to have relevance. The Cartesian frame and the uniformizing coordinates derived from it enjoy a privileged status as they tie the abstract realm of theory to something interpretable in terms of an idealization of experience. Of course, some of the superfluous structure will go away as soon as one proceeds to the next higher order, and, in the limit of arbitrarily high order, it should disappear completely. But one cannot jump directly from here to infinity! What would be an example of such superfluous structure? For instance, the distinction between an electric and a magnetic field. As one knows very well, this distinction is relativized in the special theory of relativity, since an electric resp. magnetic field with respect to one inertial frame may go over to having non-zero magnetic resp. electric components when viewed in another inertial frame (see Thirring \cite{thirring_vol_1}, {\S}5.2). This situation holds for the electromagnetic field at first order in the jets, and the right way to handle it is to introduce the field strength tensor as a geometrical invariant. Something similar happens at second order: a second-order Lorentz transform can convert an electromagnetic field into an electroweak field and vice versa. Thus, we seek the appropriate physical concepts by which to interpret the covariant meaning of the generalized field strength tensor at each successive higher order (that is to say, in Cartan's formalism, the curvature 2-forms $R^\alpha_\beta$).

To put the matter in other terms, what we are dealing with is a graduated series of enlargements of the notion of an inertial frame. A subsequent consideration enters at this juncture, which we encapsulate as our second postulate:

\begin{postulate}[Principle of ascent]
	The analysis of the effective structure of the gauge group at any given order should proceed by starting at lowest order and working one's way up order by order.	
\end{postulate}
Thus, with respect to the problem presently at hand, one should start with the jet affine connection 1-forms in the $(1,1)$ sector, then go to the $(1,2)$ sector followed by the $(2,2)$, $(1,3)$, $(2,3)$ and $(3,3)$ sectors. Why does this approach make sense? For the $r$-th order Lorentz group decomposes into a succession of subgroups as the order increases. That is, one can begin with purely linear transformations, then consider purely quadratic transformations, then purely cubic etc. Roughly speaking, one wishes to fix as much of the superfluous structure as possible at a given order and to retain the choices made at each successive stage. As stated, the principle of ascent may seem to be somewhat vague for now, but we intend to exemplify it in practice in the immediately following subsection. Perhaps, one hopes, at a later stage of development of the theory it will be possible to abstract a systematic procedure.

There remains one further consideration to introduce here, which will equip us with a physical principle by means of which to carry out the said fixing of the superfluous structure at each given stage. In the absence of a regular theory of matter, the question as to which gauge transformations (by which we mean a continuous change of the Cartan frame) to retain in the effective theory depends ultimately on an hypothesis as to how matter may be expected to behave under such transformation. If one reckons it reasonable to suppose that the two configurations of matter related by the proposed gauge transformation ought to be treated as physically equivalent, then, in the corresponding sector of the gauge potentials, the two gauge fields are to be identified as belonging to the same gauge equivalence class. Thus, based on one's idea as to what configurations of matter should be physically distinguishable or not, as the case may be, one can implement a rule describing how to cut down the gauge field degrees of freedom to a complete set of representatives of the gauge equivalence classes thereby decided upon.

Thus, for instance, in the special theory of relativity configurations of matter that pass into one another under the operation of an element of the ordinary Lorentz group are to be identified as physically equivalent. Now, in carrying out the ascent, we propose to do the same thing with respect to Lorentz transformations of successively higher order---just that the procedure will necessarily be more difficult to visualize since it must involve infinitesimals of higher order and moreover will be unfamiliar, it never having been done before in the literature. Let us, nonetheless, specify what physical intuition tells us about how to imagine the procedure, based upon the state of post-Einsteinian theory as it has been advanced in Part II and {\S}III.1 and thus far in {\S}III.2.

The standpoint on the nature of matter adopted above in Part II is that it is to be thought of as consisting in a flow of deformations to the generalized curvature 2-forms, for, in the field equation of the general theory of relativity, the stress-energy term on the right hand side acts dynamically to reshape the curvature of space, as reflected by the Einstein tensor on the left hand side. The cosmological term figures in this context as a constant residual background of stationary dark energy. The lesson is that we are to picture matter itself as a kind of flow of dark energy from one place to another. At this point, we can bring into the scenario a methodological maxim, viz., Einstein's well-known Machian instrumentalism which plays a key role in his establishment of an epistemological foundation to the general theory in his 1916 paper. Hence, we have to visualize for ourselves what an observer could see, resp. could not see, when presented with the spectacle of a flow of dark energy, as indicated. To this end, we shall invoke what may be termed a microscopic elevator thought experiment. Suppose the observer, who exists on a scale somewhere between that of the gauge fields (i.e., subnunclear) and the Planck regime, to be enclosed inside a box so that the only spatio-temporal coincidences he is capable of noting would those that take place in his immediate vicinity. What could such an observer see?

As an analogy, imagine a tiny observer suspended in a flow of liquid water. Under such conditions, the local observables would be such things as the relative differentials in the hydrodynamical velocity field \textit{at the position of} the observer; hence, shear, vorticity etc. The said observer would \textit{not} be capable of recording the orientation of any co-moving minute bodies that might also be suspended in the liquid relative to a distant unmoving frame, such as those of the laboratory or of the fixed stars. Therefore, we wish to apply a similar consideration to our microscopic observer inside a box, and propose the following principle as physical:

\begin{postulate}[Principle of microscopic indiscernability]
	When determining the effective gauge Lie algebra in the post-Einsteinian general theory, one is free to remove as equivalent any two configurations of the gauge potentials (that is, the jet affine connection 1-forms) between which there are not expected to be any locally distinguishable configurations of matter.
\end{postulate}

\begin{remark}
	The reason we have to be so circumlocutory is that, in our theory, we do not have at our disposal any explicit representations of the matter sector (such as were available when the Glashow-Weinberg-Salam electroweak unification was worked out). Thus, for us, the matter sector has to be pictured as mentioned in terms of a flow of curvature deformations in the analytical form expressed by equation (\ref{gauge_matter_coupling}) for the \textit{phenomenological} contribution of matter to the action. Let us point out, however, that the stipulation of microscopic indiscernability does have content. For, one could in principle employ material test particles to survey the disposition of the generalized field strength tensor, and this device should be be no less practicable at the microscopic level. Thus, if two gauge potentials give rise to a field strength tensor who effect on the material test particles is locally indiscernable to the microscopic observer, these are to be cast as gauge equivalent according to our postulate.
\end{remark}

\subsubsection{Application to Chromodynamical Symmetry Breaking}

To recur to Einstein's 1905 thought experiment, the novelty he introduces consists precisely in the contention that a moving magnet described as accelerating the electrical charges in a nearby stationary conducting coil by means of the electrical field produced by the time-varying magnetic field \textit{cannot be distinguished from} the alternate description of the situation in terms of a stationary magnet inducing an electromotive force in the moving conducting coil through Faraday's law; only the \textit{relative velocity} between the magnetized body and the conducting coil is physically meaningful. In the same way, the postulates proposed in the previous section imply that many (though not all) of what might prima facie seem to be differing configurations among the affine jet connection 1-forms are, in fact, indistinguishable as far as the microscopic observer is concerned. Thus, to third order in the jets, what manifests itself in physics at terrestrial scales is only the effective gauge Lie algebra $\u(1)_Y \oplus \su(2) \oplus \su(3)$.

First of all, we have to ask ourselves a question concerning a matter of principle, namely what does it mean to measure a jet? A 1-jet is nothing but an idealization of a spatio-temporal displacement in the limit as it tends to zero, compared to anything else. A 2-jet would be a displacement of a displacement, and so forth. Thus, to measure a jet always means to look at a phenomenon occurring at scales small in comparison with the laboratory. The apparatus, mounted say on an optical table, should permit finely calibrated observations of anything within its ambit. Needless say, besides training a camera on a focal point of interest or using a stopwatch to time successive events of interest, one could also detect a jet indirectly via its influence on something else more readily measureable---as is usual in physics; for instance, consider how a mass spectrometer determines the mass of an ionic molecular species versus a weighing it on a scale. But, nothing about performing measurements on jets necessitates knowledge of the outside world. So measurements of jets constitute the kind of thing that fits admirably within the scope of the principle of microscopic indiscernability.  

Now it is clear that the principle of local Lorentz invariance coming from the special theory of relativity applies with that much the more force to a microscopic observer. Therefore, he will find it impossible to distinguish between jets that pass into one another under change of the laboratory's inertial frame. This circumstance holds immediate relevance as soon as we begin to consider 2-jets. For not only must we expect that two 2-jets related to each other via an ordinary Lorentz transform are to be treated as physically equivalent, we ought further to suppose that any generalized tensor built from such jets cannot contain an element of physical reality not belonging to the jets themselves; otherwise the element of physical reality would have to manifest itself in such a way as to render a distinction between jets we have just entertained as indistinguishable, a contradiction. Note that this observation implies a stronger statement than that of Lorentz covariance alone. For it maintains that the physically inequivalent tensors are not to be constructed from tensor products of the 2-jets, denoted as $J^2 \otimes J^2$ and the like, subsequently to be subjected to Lorentz covariance (hence, $(J^2 \otimes J^2)/SO(1,3)$ etc.), but that they must be constructed from the space of 2-jets modulo equivalence under Lorentz transform, that is, from $J^2/SO(1,3) \otimes J^2/SO(1,3)$. Therefore, the gauge Lie algebra must initially break from $\so(10)$ to $\so(4) = \so(3) \oplus \so(3)$. The further breaking from $\so(3) \oplus \so(3)$ to $\u(1)_Y \oplus \so(3) = \u(1)_Y \oplus \su(2)$ is described above in {\S}\ref{spont_symm_breaking}.

At the next level, any totally isotropic subspace in the $(3,3)$ sector is as good as any other: why? All a microlocal observer can tell is what the angle subtended by any two isotropic vectors may be. Thus, the diagonal elements $g \times g \in \so(7) \subset \so(7) \oplus \so(7) \subset \so(7) \oplus \so(13)$ represent true gauge degrees of freedom and we may select a given totally isotropic subspace as a representative of the gauge equivalence class. Note, the scenario is analogous to Einstein's original thought experiment in the special theory of relativity: configurations of matter connected by other elements ought to be treated as physically identical (if they were not, it would be just as strange as an apparent violation of ordinary Lorentz invariance). 

In order to understand what effective gauge degree of freedom exist inside $\so(7)$, let us examine an analogy to a symmetric top: if, according to microscopic indiscernability, one were barred from recording the orientation of the top with respect to a larger absolute reference frame, he would gauge away rotations about the long axis leaving just two Euler angles because the third Euler angle, registering differing orientations of the top about its long axis, is distinguishable only by reference to the outside world (as discerned by means of incoming light rays from the stars, for instance). How to apply this idea in the present context of 3-jets? Here, attach a label to any two 3-jets in $\so(7)$ corresponding to their cross product. Hence, $\so(7)/g_2$ just shifts among identifiable configurations with respect to the cross product. We may therefore pick the representative in each class where the cross product preserved. In other words, to each gauge equivalence class (element of $\so(7)/g_2$) we select a canonical representative in a consistent way, thus reducing the effective Lie algebra from $\so(7)$ to $g_2$.

Imagine now how $g_2$ acts on $\imo$ by selecting a basis of the latter and subjecting the basis elements to the motion generated by a given element of $g_2$. If we single out one such basis vector, its stabilizer group in $g_2$ will be a copy of $\su(3)$. Thus, as a collective the basis vectors in $\imo$ move in unison under $g_2/\su(3)$ only. But this implies that as a local observer who cannot compare with distant sources we could not distinguish differing cosets in $g_2/\su(3)$, for only the relative spatial orientations of the basis elements among themselves matter. Hence, the cosets in $g_2/\su(3)$ must, according to the principle of microscopic indiscernability, be regarded as gauge equivalent classes. Therefore, modulo $g_2/\su(3)$ we are left with $\su(3)$ alone as our surviving candidates for the effective Lie algebraic gauge degrees of freedom in the $(3,3)$ sector.

But under the action of elements within $\su(3)$ the relative disposition of the seven generators of $\imo$ \textit{does} vary. Hence, a local microscopic observer could detect it, in principle with an experiment sensitive to the angles subtended by any given pair of generators. In any event, as a matter of experience, we have excellent reason to suppose quarks transforming under $\su(3)$. In other words, the statement of experience that it is not possible to gauge away $\su(3)$ is consistent with the theoretical expectation. But we still face the issue of which copy of $\su(3)$ to adopt as representative among those related by conjugacy. The natural choice would be to fix $Z \in V^{(2,3)}$ itself: as it leads to simplified form of the theory, else one would in effect always be passing to this and back (for the same reason as, in celestial mechanics, it is convenient when analyzing the restricted three-body problem to go to the co-rotating frame in which Jacobi's constant takes on a certain value; cf. Thirring, vol. 1, {\S}4,4 \cite{thirring_vol_1}). 

\subsubsection{What may be the Role of the Off-diagonal Sectors?}\label{off_diagonal_sectors}

Now we have reduced $A^{(2,2)}$ to $\u(1)_Y \oplus \su(2)$ and $A^{(3,3)}$ to $\su(3)$. What about the off-diagonal sectors $A^{(1,2)}$, $A^{(1,3)}$ and $A^{(2,3)}$? As seen in the previous subsections, these do not survive to be part of the effective gauge Lie algebra since the putative gauge transformations they generate would fail to be symmetries of the Proca term in the Proca-Yang-Mills equation. Nevertheless, the corresponding vector fields remain part of the problem. The $A^{(1,2)}$ potentials cannot be excluded on grounds of their masses, since according to {\S}\ref{chromodynamic_scale} they must be nearly massless. The $A^{(1,3)}$ and $A^{(2,3)}$ potentials, on the other hand, will receive appreciable masses via the spontaneous symmetry breaking mechanism. In fact, the typical masses expected in the $(2,3)$ sector are so great that the corresponding potentials can be entirely eliminated except for a totally isotropic subspace. In the $(1,3)$ sector, however, the masses must be comparable to those of the electroweak gauge bosons, or $M_W$ resp. $M_Z$. Therefore, we predict a novel set of intermediate vector fields and forces corresponding to these off-diagonal potentials left over after the spontaneous symmetry breaking, some effectively massless and others having masses on the electroweak scale, whose phenomenology is completely unknown.

\subsection{Summary}

The conclusion we have reached in {\S}\ref{identification_of_chromodynamics} is that the basic principles of the post-Einsteinian general theory of relativity as outlined above in {\S}\ref{dynam_process} do indeed lead, at third order in the jets, to a strong nuclear force possessing the $SU(3)$ symmetry of classical chromodynamics. For the time being, we cannot address asymptotic freedom and the origin of quark confinement in quantum field theory without going well beyond the present state of the art.

It is to be stressed that no ad hoc arrangements are required to bring about this result. Rather, all one needs is the premise that spontaneous symmetry breaking at third order in the jets should behave in a similar fashion to what has already been worked out at second order in {\S}\ref{chapter_8}. The result follows from the already postulated structure of the pseudo-Riemannian signature of the generalized Minkowskian metric in flat space-time to third order in the jets. Hence, the fact that,  without any further modifications, it yields just the non-abelian $\mathrm{su}(3)$ Lie algebra corresponding to the Yang-Mills gauge fields for the strong nuclear force of the standard model constitutes a striking confirmation of the correctness of our entire view of fundamental forces. 

The great unresolved puzzle, needless to say, remains to characterize the nature of matter and to derive the pattern of fermionic fields appearing in the standard model of elementary particle physics. The infinitesimal geometry of higher order promises to underwrite the development of a fundamental theory of matter as well as of force. The following reflection seems to be pertinent. Given the admirable empirical successes of Einstein's general theory of relativity, when we propose to interpret matter as a dynamical flow of curvature perturbations (or perhaps one should say, the hydrodynamics of dark energy), we do not go beyond what we already have good reason to suppose exists in nature. Any speculative mechanistic hypothesis such as that of the superstring would seem to be superfluous.

	\section{Discussion}
	
The novel role played by infinitesimals in the foundations of differential geometry proposed here has already borne surprising fruit, for not only do the familiar concepts extend more or less immediately to the higher-order setting while acquiring a radically changed significance but also the consistency of the formalism has been underscored by its confirmed predictions of novel effects in orbital mechanics and by the ease with which unification of gravitation with the electroweak force proceeds. Yet the research program remains still very much in its infancy. We wish to close with a few suggestions for next steps.

There is evidently a need for an improved understanding of the implications of general covariance for the interrelations among fundamental forces, now to be viewed as united. So far, we have not considered arbitrary coordinates but worked in a Cartesian frame and with changes of inertial frame implemented by Lorentz transformation only. Just as in Maxwellian electrodynamics, a Lorentz transform can convert an electric field into a magnetic field or vice versa, one expects that suitable non-linear changes of coordinate could mix the electroweak gauge fields with the gravitational field. An appreciation of the physical significance of effects such as this is so far lacking. 

Likewise, a fuller exploration of the nature of higher inertia is called for; the idea of higher momentum as an intertwinor has been investigated only in a Cartesian frame, employing a non-canonical splitting of a 2-vector into 1-vector and 2-vector parts. How is the intertwining property rightly to be formulated in a coordinate-invariant manner? Moreover, we have stayed for the most part in the weak-field regime. What happens when one crosses into the strong-field regime, i.e., intermediate between the subnuclear scale of Yang-Mills gauge theory (on the order of $10^{-13} ~\mathrm{cm})$ and Planck scale (on the order of $10^{-33} ~\mathrm{cm})$? Here, the $\Gamma\Gamma$ cross-terms become material, but the physical implications of their contribution remain unclear, whether for astrophysics or for elementary particle physics. 

The actional perspective which has loomed large in the development of the theoretical approach of the present work is entirely satisfactory in the classical regime but stands in need of completion, if it is to apply outside this domain. One must expect quantum mechanics to become relevant when dealing with field configurations whose actions differ by an amount on the order of Planck's constant. In a unified treatment of the fundamental forces such as the dynamical philosophy naturally invites us to entertain, quantization at the field-theoretical level could succeed only as a theory of quantum gravity, for, in a unified picture, the electroweak and strong force as well cannot be neatly decoupled from the gravitational field.

What the present theory has to contribute is a fresh perspective. The problem of quantum gravity seems to have been ill-posed in the absence of higher infinitesimals, for post-Einsteinian theory suggests that non-abelian Yang-Mills gauge theory is only an approximation and it may well be impossible to prove existence of interacting quantum fields without leaving behind the restriction to this limiting situation. In particular, one could expect that a constructive theory at the quantum-field-theoretic level must include everything and, possibly (one hopes), that the presence of infinitesimal degrees of freedom might soften the severe singularities in operator-valued distributions that currently obstruct progress. If so, this would explain why efforts to construct non-abelian Yang-Mills gauge fields in four dimensions by themselves (the weightiest among them being Balaban's) have foundered.

The present work represents progress halfway towards Einstein's goal of a unified physics, in that the suggested actional perspective centered on the generalized Riemannian curvature specifies in principle how forces must enter into the theory, but we have as yet no corresponding theoretical framework by which to determine how matter itself ought to play a role in the formalism. For the time being, matter has been represented phenomenologically through the stress-energy tensor on the right-hand side of the field equations, just as in Einstein's and Maxwell's theories, but the guise under which it thus appears has not been derived from any fundamental geometrical idea.

\end{document}